\documentclass{article}
\usepackage[utf8]{inputenc}
\usepackage{graphicx}
\usepackage{amsmath,amsfonts,amsthm}
\usepackage[utf8]{inputenc}
\usepackage{hyperref}
\usepackage{xfrac}
\usepackage{tikz}
\usepackage{esint}
\usepackage{mathrsfs}
\newcommand\res{\mathop{\hbox{\vrule height 7pt width .5pt depth 0pt
			\vrule height .5pt width 6pt depth 0pt}}\nolimits}
\newcommand\LL{\res} 

\newcommand\eps{\varepsilon}
\newcommand\Id{\mathrm{Id}}

\newcommand\cof{\mathrm{cof}}
\newcommand\sgn{\mathrm{sgn}}

\newcommand\loc{\mathrm{loc}}
\newcommand\diam{\mathrm{diam}}
\newcommand\strokedint{\fint}

\newcommand\Ver{\mathrm{Vert}}
\newcommand\Lip{\mathrm{Lip\,}}

\newcommand\dist{\mathrm{dist}}

\newcommand\skw{\mathrm{skw}}

\newcommand\conv{\mathrm{conv}}

\newcommand\Span{\mathrm{span}}
\newcommand\ASpan{\mathrm{\text{aff-span}}}

\newcommand\SO{\mathrm{SO}}
\renewcommand\O{\mathrm{O}}

\newcommand\R{\mathbb{R}}
\newcommand\N{\mathbb{N}}
\newcommand\Z{\mathbb{Z}}

\newcommand\calL{\mathcal{L}}
\newcommand\calH{\mathcal{H}}

\newcommand\calT{\mathcal{T}}
\newcommand\calP{\mathcal{P}}
\newtheorem{theorem}{Theorem}[section]

\newtheorem{proposition}[theorem]{Proposition}
\newtheorem{lemma}[theorem]{Lemma}
\newtheorem{remark}[theorem]{Remark}
\newtheorem{corollary}[theorem]{Corollary}

\numberwithin{equation}{section}
\usepackage{fancyhdr}
\newcommand{\muj}{\mu_u}

\newcommand\Dnabla{{D}}
\newcounter{Nummer}
\newenvironment{schrittlist}{
\setcounter{Nummer}{0}
\begin{list}{{\em Step \arabic{Nummer}:}}
{\setlength\leftmargin{0pt}
\setlength\labelwidth{0cm}
\setlength\itemindent{.2cm}
\usecounter{Nummer}}
}{
\end{list}
}
\begin{document}
\begin{center}
{\Large Approximation of $SBV$ functions with {possibly} infinite jump set}\\[5mm]
{\today}\\[5mm]
Sergio Conti$^{1}$, Matteo Focardi$^{2}$ and Flaviana Iurlano$^{3}$\\[2mm]
{\em $^{1}$
 Institut f\"ur Angewandte Mathematik,
Universit\"at Bonn,\\ 53115 Bonn, Germany}\\[1mm]
 
{\em $^{2}$ DiMaI, Universit\`a di Firenze\\ 50134 Firenze, Italy}\\[1mm]
{\em $^{3}$ Sorbonne Université, CNRS, Université Paris Cité,\\ Laboratoire Jacques-Louis Lions, 75005 Paris, France}
\\[3mm]
    \begin{minipage}[c]{0.8\textwidth}
We prove an approximation result for functions $u\in SBV(\Omega;\R^m)$ such that $\nabla u$ is $p$-integrable, $1\leq p<\infty$, and $g_0(|[u]|)$ is integrable over the jump set (whose $\calH^{n-1}$ measure is possibly infinite), for some continuous,
nondecreasing, subadditive function $g_0$, with $g_0^{-1}(0)=\{0\}$. The approximating functions $u_j$ are piecewise affine with piecewise affine jump set; the convergence is that of $L^1$ for $u_j$ and the convergence in energy for $|\nabla u_j|^p$ and $g([u_j],\nu_{u_j})$ for suitable functions $g$. In particular, $u_j$ converges to $u$ $BV$-strictly, area-strictly,
and strongly in $BV$ after composition with a bilipschitz map. If in addition $\calH^{n-1}(J_u)<\infty$, we also have convergence of $\calH^{n-1}(J_{u_j})$ to $\calH^{n-1}(J_u)$.
\end{minipage}
\end{center}

\tableofcontents 
\section{Introduction and main result}

Approximation with regular objects is a fundamental tool in many problems in functional analysis and 
in the Calculus of Variations. For instance,
De Giorgi's theory of sets of finite perimeter depends crucially on the approximability with piecewise smooth sets, a key step in the theory of Sobolev spaces is approximation by smooth functions (for example, the proof of the chain rule depends on it), and similarly for functions of Bounded Variation. Indeed, in these cases a possible definition {of the relevant function space} is via relaxation of a functional defined on smooth maps, and the difficult {part} is proving that this is equivalent to the intrinsic definition {on} measurable sets or functions.

More specifically, approximation and density play an important role in relaxation, $\Gamma$-convergence, integral representation, semicontinuity and many other aspects of the Calculus of Variations in which the topology of the function space is complemented by a variational functional to be minimized. In these applications 
it is important to approximate in the relevant topology and in energy. 
{In this respect,} the literature contains many approximation results 
for free discontinuity problems, mainly focused on either linear growth or discontinuity sets with finite measure, as appropriate for example for models of {concentration of plastic slip or {for the} Griffith model of} brittle fracture.
{Our main aim here is} approximation in energy without the assumption that the jump set has finite measure. One natural application of our result is the study of {superlinear} models of
cohesive fracture. 

{The functional framework to settle {this} kind of problems is provided by (a suitable subspace of) 
the space of} Special functions of Bounded Variation, introduced by De Giorgi and Ambrosio in \cite{DeGiorgiAmbrosio} to model a large class of problems which are described by a volume energy and a surface energy (e.g., mixtures of liquids, liquid crystals, image segmentation, fracture mechanics,~...). Indeed, {$SBV(\R^n;\R^m)$ is the set of} functions $u\in BV(\R^n;\R^m)$ whose distributional derivative has no Cantor part:
\[Du=\nabla u\mathcal L^n+[u]\otimes \nu_u\mathcal H^{n-1}\res J_u,\]
where $\nabla u$ is the approximate gradient and $J_u,\nu_u,[u]=u^+-u^-$ are respectively the jump set, its normal, and the amplitude {of the jump}, {see \cite{AFP} for the definitions.}

In these problems, the general form of the energy is
\begin{equation}
F[u,A]:=\int_A \Psi(x,\nabla u)dx+\int_{J_u\cap A }
g(x,u^+,u^-,{\nu_u})d\calH^{n-1},
\end{equation}
for $A\subset \R^n$ open and bounded, $\Psi$ and $g$ satisfying suitable growth and regularity properties, $u\in SBV(A;\R^m)$. If one is interested in the (possibly constrained) global minimization of $F$, lower semicontinuity and coercivity are further required in order to apply the direct method of the Calculus of Variations and to establish the existence of a solution.

For many applications it is of crucial importance to be able to approximate $u\in SBV(A;\R^m)$ {in 
$L^1(A;\R^m)$ and }in the sense of the energy by a sequence $u_j$ of more regular functions (for example piecewise regular), i.e., in a way that $F[u_j,A]\to F[u,A]$ as $j\to \infty$. This was the aim of several works appeared in the recent years. Braides and Chiadò-Piat in \cite[{Sect.~5}]{BraidesChiadoPiat} focus on functions $u\in SBV^p\subset SBV$, $p>1$, i.e. such that $\nabla u\in L^p$ and $\mathcal H^{n-1}(J_u)<\infty$.
For functions $u\in SBV^p\cap L^\infty$ they provide an approximation $u_j\in SBV^p$, regular out of a {closed} rectifiable set, satisfying
\begin{equation}\label{BCP}
u_j\to u \text{ {strongly} in } BV,\quad \nabla u_j\to \nabla u\text{ in }L^p,\quad \calH^{n-1}(J_{u_j}\triangle J_u)\to0.
\end{equation}
{Cortesani in \cite{cortesani1997strong} and Cortesani and Toader in \cite{CortesaniToader}}, on the positive side, improve this result, by constructing for $u\in SBV^p{\cap L^\infty}$, $p>1$, a sequence $u_j$ whose jump set is in addition piecewise regular, and precisely polyhedral. Moreover they get
\begin{eqnarray*}
&\nabla u_j\to \nabla u\text{ in }L^p,\\
&\displaystyle\limsup_{j\to\infty}\int_{J_{u_j}\cap {\overline A}}g(x,u^+_j,u^-_j,\nu_{u_j})d\mathcal H^{n-1}\leq \int_{J_{u}\cap {\overline A}}g(x,u^+,u^-,\nu_{u})d\mathcal H^{n-1},
\end{eqnarray*}
on $A\subset\subset \Omega$. On the negative side, they do not obtain strong convergence in $SBV$. 

The strong convergence in $SBV$ {holds for} the {result} by De Philippis, Fusco and Pratelli in \cite[Theorem C]{DePhilippisFuscoPratelli}, in which, for $u\in SBV^p$, $p>1$, the authors construct $u_j$ regular out of 
{the closure of its jump set,  which is actually essentially closed being contained in a compact $C^1$ manifold with $C^1$ boundary, and differs from it {only by} an $\calH^{n-1}$-negligible set.}

The previous {four} results have been crucial for many applications involving a penalization on the measure of the jump set. The case in which the jump set of $u$ is allowed to have infinite measure is quite different and few approximations are available in the literature.
{An extension of the result by Cortesani and Toader
to $BV$ was obtained in \cite{amar2005new} {in the setting of $BV$ strict convergence}.}
In \cite{KristensenRindler}, the approximation of any $BV$ function is obtained in the area-strict sense through countably piecewise affine functions with the same trace as $u$ at the boundary. A different approximation is provided in $SBV$ in \cite[Theorem B]{DePhilippisFuscoPratelli}. 
Precisely, the authors prove that if $u\in SBV$ with $\nabla u\in L^p$, $p>1$, then it is possible to construct $u_j$ regular out of {the closure of its jump set, which is actually essentially closed being {(up to $\calH^{n-1}$-null sets) a compact} $C^1$ manifold with $C^1$ boundary,} and satisfying
\[u_j\to u \text{ {strongly} in }BV,\quad \nabla u_j\to \nabla u\text{ in }L^p\,.\]
{In particular, the convergence $\calH^{n-1}(J_{u_j}\setminus J_u)\to0$ is not ensured.
Moreover, in case $p=1$, {the jump of} $u_j$ can be additionally taken contained in the intersection of a compact $C^1$ manifold with $C^1$ boundary and of the jump set of the function to be approximated (see \cite[Theorem A]{DePhilippisFuscoPratelli}).}
{Related density results, with different functional settings such as $(G)SBD$ or $BH$, have been obtained 
in the last years (see for example \cite{Chambolle,iur12,ContiFocardiIurlano2017IntRepr,friedrich2018piecewise,FriedrichSolombrino2018quasistatic,ContiFocardiIurlano2019-Density,crismale2019approximation,chambolle2019density} {and \cite{AABU,ABC}}, respectively).}

Although {all the quoted results} are important advances, they are in general not enough for many applications, not providing any information on the convergence of the surface term or of the total energy $F$ {in case that the measure of the jump set is not finite}. An easy example is that of an energy $F$ where $\Psi=\Psi(\nabla u)$ is superlinear for large gradients and $g=g([u],\nu_u)$ is superlinear for small amplitudes, the natural domain of finiteness being {(a subset of)} $SBV$. In this case, the only result available in the literature is \cite[{Sect.~4}]{BellettiniChambolleGoldman}, which however applies only to $u\in GSBV$ with $\nabla u=0$ {$\calL^n$-a.e. on $\Omega$}. The approximants satisfy $\nabla u_j=0$ {$\calL^n$-a.e. on $\Omega$} and have jump sets of finite measure. The convergence is that of $L^1$ together with the convergence of the energies.

In this paper, we {develop an original multiscale technique to approximate} {functions} $u\in SBV$ with jump set of possibly infinite measure and $\nabla u\in L^p$, with $p\geq 1$.
{We stress that {it} encompasses at the same time both {superlinear,} cohesive-type and {Griffith,} brittle-type surface energies as shown in Section~\ref{hypotheses}.}

\begin{theorem}\label{theodensityintro}\sloppypar
 Let $\Omega\subseteq\R^n$ be an open bounded Lipschitz set, $u\in SBV(\Omega;\R^m)$ such that 
 $\nabla u\in L^p(\Omega;\R^{m\times n})$ for some $p\in[1,\infty)$, and $g_0(|[u]|)\in L^1(\Omega;\calH^{n-1}\res J_u)$,
 with $g_0:[0,\infty)\to[0,\infty)$ continuous, nondecreasing, subadditive, and $g_0^{-1}(0)=\{0\}$.
 
 Then there are sequences $u_j\in SBV\cap L^\infty(\Omega;\R^m)$ and 
$\Phi_j\in \Lip(\R^n;\R^n)$ such that
 \begin{enumerate}
  \item\label{theodensityintrosimpl} for each $j$ there is a {locally finite} decomposition of $\R^n$ in simplexes such that $u_j$ is affine {in the interior of} each of them;
  \item\label{theodensityintroL1} $u_j\to u$ in $L^1(\Omega;\R^m)$;  
  \item\label{theodensityintroLp} $\nabla u_j\to\nabla u$ in $L^p(\Omega;\R^{m\times n})$;
  \item\label{theodensityintrobilip}\sloppypar $\Phi_j$ is bilipschitz, with $\Phi_j(x){-x}\to 0$ in $L^\infty(\R^n;\R^n)$, $\Dnabla\Phi_j\to \Id$ in $L^\infty(\R^n;\R^{n\times n})$,
  {and $\Phi_j(x)=x$ for $x\in \R^n\setminus\Omega$};
  \item \label{theodensityintrog0} 
one can choose the orientation of
{the normal $\nu_j$ to}
$J_{u_j}$ so that
\begin{equation}\label{theodensityintrog1}
\lim_j\int_{J_u\cup \Phi_j^{-1}(J_{u_j})}g_0(|[u]-[u_j]\circ\Phi_j|)\,d\calH^{n-1}= 0
\end{equation}
{(with $[u]=0$ outside $J_u$, and similarly for $u_j$)},
and
\begin{equation}\label{theodensityintrog2}
\lim_j\int_{J_u\cup \Phi_j^{-1}(J_{u_j})}g_0(|[u]|+|[u_j]\circ\Phi_j|)
 \bigl|\nu_u - \nu_j\circ\Phi_j \bigr|
\,d\calH^{n-1}= 0;
\end{equation}
\item\label{theodensityintroHn1}
if $\calH^{n-1}(J_{u})<\infty$, then also
$\calH^{n-1}({J_u{\triangle}\Phi_j^{-1}(J_{u_j})})\to0$;
\item\label{theodensityintrosbv0}
if $\nabla u=0$ $\calL^n$-almost everywhere on $\Omega$, then $\nabla u_j=0$ $\calL^n$-almost everywhere on $\Omega$ for all $j$. If instead $u\in W^{1,p}(\Omega;\R^m)$ then $u_j\in W^{1,p}(\Omega;\R^m)$ for all $j$.
\end{enumerate}
\end{theorem}

Few remarks are in order. First, {since $\Phi_j(\Omega)=\Omega$, the integrals in \eqref{theodensityintrog1} and \eqref{theodensityintrog2} are over subsets of $\Omega$. Then,}
thanks to the subadditivity of $g_0$ from items \ref{theodensityintrobilip}  and \ref{theodensityintrog0}
{it follows that}
  \[
\lim_j\int_{J_{u_j}}g_0(|[u_j]|)\,d\calH^{n-1}= \int_{J_{u}}g_0(|[u]|)\,d\calH^{n-1}
  \]
  {(see the proof of Corollary~\ref{theodensity Psi g} below).}
Moreover, under suitable assumptions discussed in details in Section~\ref{hypotheses}, we can deduce 
the convergence of surface energies with density $g:\R^m\times S^{n-1}\to[0,\infty)$ depending suitably 
on the full jump and the normal. Finally, if $\Psi\in C^0(\R^{m\times n})$ has $p$-growth
(cf.~again Section~\ref{hypotheses}) then {\ref{theodensityintroLp} implies}
\[
 \lim_j\int_\Omega\Psi(\nabla u_j)\,dx=\int_\Omega\Psi(\nabla u)\,dx\,.
\]
In addition, the sequence $(u_j)_{j\in\mathbb N}$ can be chosen such that the convergence to $u$ is 
stronger, namely {strict} in $BV$ and in area, {see Corollary~\ref{theodensity Psi g area} below.

We stress that energies with bulk density $\Psi$ and surface density $g$ as above are in general not {$L^1$ or weakly$^\ast$}-$BV$ lower semicontinuous. Hence, our approximations can be used to prove relaxation formulas in the spirit of \cite{Ambrosio,BraidesCoscia,bou-bra-but}. This will be the object of future work in \cite{CFI23}.
	
The proof of Theorem \ref{theodensityintro} is obtained through an explicit construction in several steps. First, $u$ can be extended to a function {defined on}  a slightly larger set at a small energy cost. This is not achieved by local reflections at the boundary and a partition of unity process {as usually done}, which would require $u\in L^p$. It is {rather} pursued through a regularization of the normal vector at the boundary and the definition of a bilipschitz map which swaps an inner  neighborhood of the boundary with an outer one. Further details can be found in Section \ref{a:extension}.
	
We employ next a multiscale approach. More precisely, we
find a suitable scale $\delta>0$, such that $\nabla u$ is close to a constant and $J_u$ is close to a $C^1$ manifold in each cube of side $\delta$ of a partition of $\R^n$. This is the object of Proposition~\ref{propdeltafromtheta}. At this point, we introduce a second scale $\eps\ll\delta$. In each cube of side $\delta$ we consider a finer triangulation with simplexes of diameter less than $c\eps$ and volume larger than $c\eps^n$.
The heart of the paper is Proposition~\ref{propsimplex}, which, given the values of $u$ on the vertices of a single simplex, and two vectors for each edge, representing the cumulated jump and the {average} gradient of $u$ on the edge, provides a piecewise affine interpolation, whose gradient and jump can be estimated respectively only through the given gradient vector or the given jump vector (see Figure \ref{fig:trianglev}). Proposition \ref{propsimplex} is then employed in Proposition \ref{propproj} (see Figure \ref{figpropproj}) to define a global projection, with good energy estimates, of any $SBV$ function on the space of piecewise affine functions.

The proof of Theorem \ref{theodensityintro}
{contains a few additional steps,} since the direct application of Proposition \ref{propproj} to the given $u$ would provide a piecewise affine approximation with surface energy controlled {only} up to a multiplicative factor by the surface energy of $u$.
To avoid this problem, we first {consider the extensions $U^\pm$ of $u$ with respect to the $C^1$ manifold approximating $J_u$ in {each} cube of side $\delta$. We then apply the previous projection to $U^\pm$. We finally introduce a piecewise affine interpolation of the $C^1$ manifold and define the approximation of $u$ as the projections of $U^\pm$ on the two sides of it}. This is performed in the proof of Theorem \ref{theodensityintro} in Section \ref{s:globalconstruction}, see also Figure \ref{figgrids2}.

The structure of the paper is the following. In Section \ref{hypotheses} we provide {several consequences} of Theorem \ref{theodensityintro}, in particular we show that the approximating sequence can be constructed such that it converges also $BV$-{strictly}, area-{strictly and $BV$-strongly after composition with a bilipschitz map}. Section \ref{s:techresults} addresses two {key} technical issues:
the extension tool in Section \ref{a:extension}
and the regularization at scale $\delta$ in Section \ref{s:approxreg}. Section \ref{s:proof} is devoted to the proof of Theorem \ref{theodensityintro}. Precisely, Section \ref{s:onesimplex} contains the construction of a relevant piecewise affine interpolation on a single simplex. Section \ref{s:projection} applies such construction to produce a piecewise affine approximation of a given $SBV$ function. Finally, Section \ref{s:globalconstruction} provides the full proof of Theorem \ref{theodensityintro} by applying the {projection} of Section \ref{s:projection} to {the extensions of $u$ on the two sides of the regularized jump set and by defining the approximation of $u$ as such projections on the two sides of a
{suitable perturbation of a}
piecewise interpolation of the regularized jump set.}

\section{{Consequences of the approximation theorem}}\label{hypotheses}
We {discuss here some} consequences of Theorem~\ref{theodensityintro}.
To this aim we fix $p\in[1,\infty)$ and consider $\Psi\in C^0(\R^{m\times n})$ obeying for some ${C}>0$,
\begin{equation}\label{e:Psi growth}
{|\Psi (\xi)|\le {C} (|\xi|^p+1).}
\end{equation}
{Throughout the paper $C$ will denote a constant, possibly depending on the dimension (if not otherwise specified) and changing from line to line.}
Next we select {a function} $g_0:[0,\infty)\to[0,\infty)$ 
{which represents a modulus of continuity of the surface energy $g$ introduced below (see in particular \eqref{e:hp g3 bis})}
satisfying:
\newcommand\Hgzerouno{(H$^{g_0}_1$\!)}
\newcommand\Hgzerodue{(H$^{g_0}_2$\!)}
\newcommand\Hguno{(H$^g_1$)\!}
\newcommand\Hgdue{(H$^g_2$)\!}
\newcommand\Hgtre{(H$^g_3$)\!}
\newcommand\Hgquattro{(H$^g_{3'}$)\!}
\begin{itemize}
 \item[\Hgzerouno] $g_0$ is continuous, nondecreasing, and $g_0^{-1}(0)=\{0\}$,
 \item[\Hgzerodue] $g_0$ is subadditive, namely for every $(t,t')\in[0,\infty)\times[0,\infty)$
 \[ 
 g_0(t+t')\le g_0(t)+g_0(t')\,.
 \]
\end{itemize}
For example {either $g_0(t)=1\wedge t^q$ or $g_0(t)=t^q$,} for $q\in (0,1]$, will do. Note that by subadditivity and
continuity of $g_0$ in zero, for every $\lambda>0$ there is $C_{\lambda}>0$ such
that for all $t\in[0,\infty)$
\begin{equation}\label{e:g0 growth}
g_0(t)\le \lambda + C_{\lambda}{t}.
\end{equation}

Then we consider any function $g\in C^0(\R^m\times S^{n-1};[0,\infty))$, such that
\begin{itemize}
 \item[\Hguno] $g(-s,-\nu)=g(s,\nu)$ for all $(s,\nu)\in\R^m\times S^{n-1}$;
 \item[\Hgdue] for all $(s,s',\nu)\in\R^m\times\R^m\times S^{n-1}$ 
 \begin{equation}\label{e:hp g3}
 g(s+s',\nu)\le g(s,\nu)+{C}\,g_0(|s'|);
 \end{equation}
 \end{itemize}
 {and either}
\begin{itemize}
 \item[\Hgtre]
{ $g(0,\nu)=0$  for all $\nu\in S^{n-1}$}
 \end{itemize}
 {or}
\begin{itemize}
\item[\Hgquattro]
{there is $\alpha>0$ such that
 $g(0,\nu)\ge\alpha$  for all $\nu\in S^{n-1}$.}
\end{itemize}
{Thanks to assumption {\Hguno}, the surface energy with density $g$ {is well defined as it does}
not depend on the chosen orientation of the normal to the jump set}.
{Assumption {\Hgtre} is useful to model cohesive-type energies, such as for example the one of the Barenblatt model. Assumption {\Hgquattro} is instead useful for surface energies typical of brittle fracture, such as the one of the Griffith model (or, in the scalar case, of the Mumford-Shah model) for which $g$ is constant.}

Exchanging the roles of $s$ and $s+s'$ in \eqref{e:hp g3} yields that
\begin{equation}\label{e:hp g3 bis}
|g(s+s',\nu)- g(s,\nu)|\leq {C}\,g_0(|s'|).
\end{equation}
Moreover, {if {\Hgtre} holds,} the latter estimate {with $s'=-s$} implies that for all $(s,\nu)\in\R^m\times S^{n-1}$
 \begin{equation}\label{e:g0}
 g(s,\nu)\le {C}\,g_0(|s|)\,.
 \end{equation}
 {If instead {\Hgquattro} holds, then by continuity there is also $\beta>0$ such that
 $g(0,\nu)\le\beta$ for all $\nu$, and in particular {for all $(s,\nu)\in\R^m\times S^{n-1}$}
 \begin{equation}\label{eqgbetag0}
  g(s,\nu)\le \beta+{C}\,g_0(|s|).
 \end{equation}
}

For $u\in SBV(\Omega;\R^m)$ and for a Borel set $A\subseteq\Omega$, we define
the energy
\begin{equation*}
 E_{\Psi,g}[u,A]:=\int_A \Psi(\nabla u)dx+\int_{J_u\cap A }
g([u],{\nu_u})d\calH^{n-1}\,,
\end{equation*}
where for any $u\in SBV(\Omega;\R^m)$ we denote by $[u]$ the function which is 
the usual jump of $u$ on $J_u$ and 0 on $\Omega\setminus J_u$.

\begin{corollary}\label{theodensity Psi g}
 Under the assumptions of Theorem~\ref{theodensityintro}, the sequence $(u_j)_{j\in\N}$ 
 introduced there satisfies
\begin{equation}\label{e:Psi conv}
\lim_j\int_\Omega\Psi(\nabla u_j)dx=\int_\Omega\Psi(\nabla u)dx\,,
\end{equation}
\begin{equation} \label{e:g conv}
\lim_j\int_{J_{u_j}}g([u_j],\nu_{u_j})d\calH^{n-1}=\int_{J_{u}}g([u],\nu_u)d\calH^{n-1}
\end{equation}
for all functions $\Psi\in C^0(\R^{m\times n})$ satisfying \eqref{e:Psi growth}, 
and {all} $g\in C^0(\R^m\times S^{n-1};[0,\infty))$ satisfying \Hguno, \Hgdue, and \Hgtre. In particular,
\[
\lim_jE_{\Psi,g}[u_j,\Omega]=E_{\Psi,g}[u,\Omega]\,.
\]
\end{corollary}
We stress that the assumptions of Theorem \ref{theodensityintro} include in particular integrability of 
$g_0(|[u]|)$ and ensure via {\eqref{e:Psi growth} and}
\eqref{e:g0} that $E_{\Psi,g}[u,\Omega]$ is finite.
\begin{proof}
Standing the $L^p$ convergence of  $(\nabla u_j)_{j\in\N}$ to $\nabla u$,
we may consider a subsequence, which we do not relabel, such that
 \[
 \limsup_j\left|\int_\Omega(\Psi(\nabla u_j)-\Psi(\nabla u))dx\right|
 \]
is actually a limit, and $(\nabla u_j)_{j\in\N}$ converges to $\nabla u$ $\calL^n$-almost everywhere on $\Omega$.
Thanks to Egorov's theorem, for every $\eps>0$
{there is $E$ with
$|\Omega\setminus E|\leq\eps$ such that
$\nabla u\in L^\infty(E;\R^{m\times n})$ and
$(\nabla u_j)_{j\in\N}$ converges to $\nabla u$ uniformly on $E$.} Therefore, we may use
\eqref{e:Psi growth} and item (iii) in Theorem~\ref{theodensityintro} to get
\begin{align*}
 \limsup_j&\left|\int_\Omega(\Psi(\nabla u_j)-\Psi(\nabla u))dx\right|
 =\limsup_j\left|\int_{\Omega\setminus E}(\Psi(\nabla u_j)-\Psi(\nabla u))dx\right|\\
&\leq C\Big(\int_{\Omega\setminus E}|\nabla u|^pdx+|\Omega\setminus E|\Big)\,.
\end{align*}
The conclusion then follows as $\eps\downarrow 0$ by absolute continuity. {As the limit is unique, convergence holds for the entire sequence.}

We next deal with \eqref{e:g conv}. To this aim we first use the Area formula
(cf. \cite[Theorem~2.91]{AFP}, {with $f=\Phi_j$ and $E=\Phi_j^{-1}(J_{u_j})$})
{which reads
\begin{equation}\label{eqareaguj}
 \int_{J_{u_j}} g([u_j],\nu_{u_j})
 d\calH^{n-1}= \int_{\Phi_j^{-1}(J_{u_j})}g([u_j]\circ\Phi_j,\nu_{u_j}\circ\Phi_j)
 {{\bf J}_{n-1}d^{\Phi_j^{-1}(J_{u_j})}\Phi_j}
 d\calH^{n-1}
.
\end{equation}
We write ${{\bf J}_{n-1}d^{\Phi_j^{-1}(J_{u_j})}\Phi_j}$
for the tangential Jacobian and remark that if
	$\Phi_j$ is differentiable then 
	${\bf J}_{n-1}d^{\Phi_j^{-1}(J_{u_j})}\Phi_j=
	|\cof(\Dnabla\Phi_j)(\nu_{u_j}\circ\Phi_j)|$.
For $\calH^{n-1}$-almost every  $x\in \Phi_j^{-1}(J_{u_j})$,
the map $\Phi_j$ is differentiable in $x$ in the directions of the tangent space.
The same holds for $y\mapsto \Phi_j(y)-y$, which is Lipschitz with Lipschitz constant {bounded by} $\|\Dnabla\Phi_j-\Id\|_{L^\infty(\R^n)}$. Therefore
for $\calH^{n-1}$-almost every $x\in \Phi_j^{-1}(J_{u_j})$ we have
$|{\bf J}_{n-1}d^{\Phi_j^{-1}(J_{u_j})}\Phi_j-1|(x)\le C \|\Dnabla\Phi_j-\Id\|_{L^\infty(\R^n)}$.
By \ref{theodensityintrobilip}, the last expression converges to 0.

We observe that by subadditivity of $g_0$
\begin{equation*}\begin{split}
  &\int_{\Phi_j^{-1}(J_{u_j})}g_0(|[u_j]\circ\Phi_j|)d\calH^{n-1}\\
  &\le 
\int_{J_u\cup \Phi_j^{-1}(J_{u_j})}g_0(|[u]-[u_j]\circ\Phi_j|)\,d\calH^{n-1}
+
\int_{J_u}g_0(|[u]|)\,d\calH^{n-1}.
\end{split}\end{equation*}
Using 
 \eqref{theodensityintrog1}
 and the assumption that $g_0(|[u]|)\in L^1(\Omega, \calH^{n-1}\LL J_u)$ we obtain that
\begin{equation}\label{eqPhihujg0finite}
  \int_{\Phi_j^{-1}(J_{u_j})}g_0(|[u_j]\circ\Phi_j|)d\calH^{n-1}
  \le C<\infty
\end{equation}
 for all $j$.
By {\eqref{eqareaguj},} the growth condition in \eqref{e:g0}, {{and the} last step \eqref{eqPhihujg0finite},} we obtain}
\begin{align}\label{e:salto circle Phij}
& \left|\int_{J_{u_j}} g([u_j],\nu_{u_j})d\calH^{n-1}-
 \int_{\Phi_j^{-1}(J_{u_j})}g([u_j]\circ\Phi_j,\nu_{u_j}\circ\Phi_j)d\calH^{n-1}\right|
 \notag\\
 &\leq  {\|1-{{\bf J}_{n-1}d^{\Phi_j^{-1}(J_{u_j})}\Phi_j}\|_{L^\infty(\Phi_j^{-1}(J_{u_j});\calH^{n-1})}}
 \int_{\Phi_j^{-1}(J_{u_j})}g([u_j]\circ\Phi_j,\nu_{u_j}\circ\Phi_j)d\calH^{n-1}\notag\\
 &\leq o(1) \int_{\Phi_j^{-1}(J_{u_j})}g_0(|[u_j]\circ\Phi_j|)d\calH^{n-1}=o(1)\,.
\end{align}
Using {\eqref{theodensityintrog1} and} \eqref{theodensityintrog2} in Theorem \ref{theodensityintro} and  the fact that $g_0$ is nondecreasing with $g_0^{-1}(0)=\{0\}$ we deduce
$\chi_{\Phi_j^{-1}(J_{u_j})}\to1$ and  
$\nu_{u_j}\circ\Phi_j\to\nu_u$, $\calH^{n-1}$-almost everywhere on $J_u$.
Thus, $\chi_{\Phi_j^{-1}(J_{u_j})}\nu_{u_j}\circ\Phi_j\to\nu_u$, $\calH^{n-1}$-almost everywhere on $J_u$.
Dominated convergence, which we can use by \eqref{e:g0} and integrability of $g_0(|[u]|)$, then yields
\begin{equation}\label{eqchangenormal}
 \limsup_j\int_{J_{u}{\cap \Phi_j^{-1}(J_{u_j})}} \left| g([u],\nu_{u_j}\circ\Phi_j)-g([u],\nu_u)\right|d\calH^{n-1}=0.
\end{equation}
Moreover, {\eqref{theodensityintrog1} in} {Theorem \ref{theodensityintro}\ref{theodensityintrog0}} yields {(with \eqref{e:g0})} that
\begin{equation}\label{eqchangenormal-1}
 \limsup_j\left(\int_{J_u\setminus \Phi_j^{-1}(J_{u_j})} g([u],\nu_u)d\calH^{n-1}
+\int_{\Phi_j^{-1}(J_{u_j})\setminus J_{u}} g([u_j]\circ\Phi_j,\nu_{u_j}\circ\Phi_j)d\calH^{n-1}\right)=0.
\end{equation}
Therefore, we conclude that
\begin{align*}
 & \limsup_j\left|\int_{J_{u_j}} g([u_j],\nu_{u_j})d\calH^{n-1}-\int_{J_u}g([u],\nu_u)d\calH^{n-1}\right|\\
 &\leq \limsup_j\left|\int_{\Phi_j^{-1}(J_{u_j})}g([u_j]\circ\Phi_j,\nu_{u_j}\circ\Phi_j)d\calH^{n-1}
 -\int_{J_{u}} g([u],\nu_u)d\calH^{n-1}\right|\\
 &\leq \limsup_j\int_{\Phi_j^{-1}(J_{u_j})\cap J_{u}}\big|g([u_j]\circ\Phi_j,{\nu_{u_j}\circ\Phi_j})
 - g([u],\nu_u)\big|d\calH^{n-1}\\
 &{\leq \limsup_j\int_{\Phi_j^{-1}(J_{u_j})\cap J_{u}}\big|g([u_j]\circ\Phi_j,{\nu_{u_j}\circ\Phi_j})
 - g([u],{\nu_{u_j}\circ\Phi_j})\big|d\calH^{n-1}}\\
 &\leq {C}\limsup_j \int_{\Phi_j^{-1}(J_{u_j})\cap J_u}
 \,g_0(|[u]-[u_j]\circ\Phi_j|)\,d\calH^{n-1}=0\,,
 \end{align*}
where we have used {\eqref{e:salto circle Phij} in the first inequality,
\eqref{eqchangenormal-1} in the second one,} \eqref{eqchangenormal} in the third one,
\eqref{e:hp g3 bis} in the fourth one, and {\eqref{theodensityintrog1} in} 
{Theorem~\ref{theodensityintro}\ref{theodensityintrog0}} in the last equality.
\end{proof}

We next show how to treat the case that $g$ is bounded from below, in which {{\Hgquattro} holds.}
{We stress that the case of the Mumford-Shah energy functional corresponds to the choices
$\Psi=|\cdot|^2$ and $g\equiv1$ as ${|[u]|}>0$ on $J_u$ for $u\in SBV$.}
\begin{corollary}\label{theodensityMS}
 Under the assumptions of Theorem~\ref{theodensityintro}, if $\calH^{n-1}(J_u)<\infty$ the sequence $(u_j)_{j\in\N}$ 
 introduced there satisfies
\begin{equation}\label{e:Psi convMS}
\lim_j\int_\Omega\Psi(\nabla u_j)dx=\int_\Omega\Psi(\nabla u)dx\,,
\end{equation}
\begin{equation} \label{e:g convMS}
\lim_j\int_{J_{u_j}}g([u_j],\nu_{u_j})d\calH^{n-1}=\int_{J_{u}}g([u],\nu_u)d\calH^{n-1}
\end{equation}
for all functions $\Psi\in C^0(\R^{m\times n})$ satisfying \eqref{e:Psi growth}, 
and $g\in C^0(\R^m\times S^{n-1};[0,\infty))$ satisfying \Hguno, \Hgdue, and \Hgquattro. In particular,
\[
\lim_jE_{\Psi,g}[u_j,\Omega]=E_{\Psi,g}[u,\Omega]\,.
\]
\end{corollary}
\begin{proof}
 The proof is very similar to the one of Corollary \ref{theodensity Psi g}. The first part, until \eqref{eqPhihujg0finite}, is identical.
Using 
\eqref{eqgbetag0}, 
{$\calH^{n-1}(J_u)<\infty$,}
\ref{theodensityintroHn1} in Theorem \ref{theodensityintro}, and \eqref{eqPhihujg0finite} we have
 \begin{equation}\label{eqPhihujg0finite2}
\begin{split}
&\int_{\Phi_j^{-1}(J_{u_j})}g([u_j]\circ\Phi_j,\nu_{u_j}\circ\Phi_j)d\calH^{n-1}\\
&  \le \beta (\calH^{n-1}(J_u)+
\calH^{n-1}(\Phi_j^{-1}(J_{u_j})\setminus J_u))+
 {C} \int_{\Phi_j^{-1}(J_{u_j})} g_0(|[u_j]\circ\Phi_j|)d\calH^{n-1}\\
&  \le C<\infty
  \end{split}
\end{equation} 
 for all $j$.
We use the latter {and \eqref{eqareaguj}}  to conclude that
\begin{align}\label{e:salto circle Phij2}
& \left|\int_{J_{u_j}} g([u_j],\nu_{u_j})d\calH^{n-1}-
 \int_{\Phi_j^{-1}(J_{u_j})}g([u_j]\circ\Phi_j,\nu_{u_j}\circ\Phi_j)d\calH^{n-1}\right|
 \notag\\
 &\leq  {\|1-{{\bf J}_{n-1}d^{\Phi_j^{-1}(J_{u_j})}\Phi_j}\|_{L^\infty(\Phi_j^{-1}(J_{u_j})                           ;\calH^{n-1})}}
 \int_{\Phi_j^{-1}(J_{u_j})}g([u_j]\circ\Phi_j,\nu_{u_j}\circ\Phi_j)d\calH^{n-1} \nonumber\\
 &=o(1)
\end{align}
which replaces \eqref{e:salto circle Phij}.
 Using \ref{theodensityintroHn1} in Theorem \ref{theodensityintro},
$\chi_{\Phi_j^{-1}(J_{u_j})}\to1$ pointwise 
$\calH^{n-1}$-almost everywhere on $J_u$.
As above, $\chi_{\Phi_j^{-1}(J_{u_j})}\nu_{u_j}\circ\Phi_j\to\nu_u$ $\calH^{n-1}$-almost everywhere on $J_u$.
From $\calH^{n-1}(J_u)<\infty$ {and integrability of $g_0(|[u]|)$} we obtain that
$\beta + g_0(|[u]|)\in L^1(\Omega;\calH^{n-1}\LL J_u)$. 
Dominated convergence, which we can use by 
\eqref{eqgbetag0}, then yields
\begin{equation}\label{eqchangenormal3}
 \limsup_j\int_{J_{u}{\cap \Phi_j^{-1}(J_{u_j})}} \left| g([u],\nu_{u_j}\circ\Phi_j)-g([u],\nu_u)\right|d\calH^{n-1}=0.
\end{equation}
Moreover, items \ref{theodensityintrog0} and \ref{theodensityintroHn1} in
Theorem~\ref{theodensityintro} and \eqref{eqgbetag0},  yield that
\begin{equation}\label{eqchangenormal-2}
 \limsup_j\left(\int_{J_u\setminus \Phi_j^{-1}(J_{u_j})} g([u],\nu_u)d\calH^{n-1}
+\int_{\Phi_j^{-1}(J_{u_j})\setminus J_{u}} g([u_j]\circ\Phi_j,\nu_{u_j}\circ\Phi_j)d\calH^{n-1}\right)=0.
\end{equation}
The rest of the proof is unchanged.
\end{proof}

We can actually strengthen the conclusions of Corollary~\ref{theodensity Psi g}
{and Corollary~\ref{theodensityMS}}
by constructing an approximating sequence converging in a stronger sense.
\begin{corollary}\label{theodensity Psi g area}
In Corollary~\ref{theodensity Psi g}
and Corollary~\ref{theodensityMS} 
the sequence $(u_j)_{j\in\N}$
can be chosen to additionally satisfy
 \begin{equation}\label{eqconvpushf}
{
\lim_j | (\Phi_j)_\#Du_j-Du|(\Omega)=0}
 \end{equation}
{and
\begin{equation}\label{eqconvcomposit}
\lim_j \| u_j\circ \Phi_j- u\|_{BV(\Omega)}=0\,.
\end{equation}
In particular,}
\begin{equation}\label{eq:convgrad}
 \lim_j\int_\Omega|\nabla u_j|dx=\int_\Omega|\nabla u|dx\,,
 \end{equation}
 \begin{equation}\label{eq:convarea}
\lim_j\int_\Omega\sqrt{1+|\nabla u_j|^2}dx=\int_\Omega\sqrt{1+|\nabla u|^2}dx\,,
 \end{equation}
 \begin{equation}\label{eq:convsalti}
\lim_j\int_{J_{u_j}}|[u_j]|d\calH^{n-1}=\int_{J_{u}}|[u]|d\calH^{n-1}\,,         
 \end{equation}
 {so that} $(u_j)_{j\in\N}$ converges to $u$ {strictly} in $BV(\Omega;\R^m)$ and in area.
\end{corollary}
\begin{proof}
The proof is based on the fact that the construction of the sequence in Theorem \ref{theodensityintro} does not depend on the details of the energy considered.
We define the auxiliary functions
$\widetilde{g}_0:{[0,\infty)}\to[0,\infty)$, $\widetilde{g}:\R^m\to[0,\infty)$, by
\[
 \widetilde{g}_0(t):=g_0(t)+t\,,\hskip5mm {t\in[0,\infty)\,,}
\]
and
\[
 \widetilde{g}(s):=|s|\,,\hskip5mm  {s\in\R^m\,.}
\]
It is easy to check that $\widetilde{g}_0$ satisfies \Hgzerouno-\Hgzerodue, and moreover that both
$g$ and $\widetilde{g}$ satisfy \Hguno-{\Hgdue} with respect to $\widetilde{g}_0$. 
Further, $\widetilde g$ satisfies \Hgtre.
In addition, if $u\in SBV(\Omega;\R^m)$, having assumed that $g_0(|[u]|)\in L^1(\Omega;\calH^{n-1}\res J_u)$,
we infer that $\widetilde{g}_0(|[u]|)\in L^1(\Omega;\calH^{n-1}\res J_u)$.
Therefore, we may consider the sequence $(u_j)_{j\in\N}$ provided by Theorem~\ref{theodensityintro} with surface density $\widetilde{g}_0$.
Thus, to {get \eqref{eq:convgrad}--\eqref{eq:convsalti}} it is sufficient to apply Corollary~\ref{theodensity Psi g} {with}
{$\Psi_1(\xi):=|\xi|$} and $\widetilde{g}$, and then  {$\Psi_2(\xi):=\sqrt{1+|\xi|^2}$} and {$\widetilde{g}$.}
{One applies either
Corollary~\ref{theodensity Psi g} 
or Corollary~\ref{theodensityMS} 
with densities $\Psi$ and $g$ to obtain convergence of the energy.
From  Theorem~\ref{theodensityintro}\ref{theodensityintrog0} for $\tilde g_0$ we obtain
 \[{
\lim_j\int_{J_u\cup \Phi_j^{-1}(J_{u_j})}|[u_j]\circ\Phi_j-[u]|d\calH^{n-1}=0\,,   }        
 \]
and with \ref{theodensityintroLp} and \ref{theodensityintrobilip}
we conclude $ | (\Phi_j)_\#Du_j-Du|(\Omega)\to0$.}
{It remains to prove \eqref{eqconvcomposit}. Recalling Theorem \ref{theodensityintro} \ref{theodensityintroL1} and \ref{theodensityintrobilip} and \eqref{eqconvpushf}, it is enough to check that
\[\lim_j\int_\Omega|\nabla(u_j\circ\Phi_j)-\nabla u|dx=0.\]
Let ${\mathcal S_j}$ be the decomposition of Theorem \ref{theodensityintro} \ref{theodensityintrosimpl}, then $u_j$ is affine in $T$, for all $T\in {\mathcal S_j}$,
and the chain rule gives
\begin{eqnarray*}&\displaystyle\int_\Omega|\nabla(u_j\circ\Phi_j)-\nabla u|dx=\sum_{T\in{\mathcal S_j}}
\int_{\Omega\cap\Phi_j^{-1}(T)}|\nabla(u_j\circ\Phi_j)-\nabla u|dx\\
&=\displaystyle\sum_{T\in{\mathcal S_j}}
\int_{\Omega\cap\Phi_j^{-1}(T)}|(\nabla u_j\circ\Phi_j)\Dnabla\Phi_j-\nabla u|dx.\end{eqnarray*}
The triangular inequality yields
\begin{multline*}\int_{\Omega\cap\Phi_j^{-1}(T)}|(\nabla u_j\circ\Phi_j)\Dnabla\Phi_j-\nabla u|dx
\leq\int_{\Omega\cap\Phi_j^{-1}(T)}|\nabla u_j\circ\Phi_j||\Dnabla\Phi_j-\Id|dx\\+
\int_{\Omega\cap\Phi_j^{-1}(T)}|(\nabla u_j-\nabla u)\circ\Phi_j|dx+
\int_{\Omega\cap\Phi_j^{-1}(T)}|\nabla u\circ\Phi_j-\nabla u|dx
\end{multline*}
and all terms tend to zero by Theorem \ref{theodensityintro} \ref{theodensityintrobilip} and \eqref{eqconvpushf}. This gives the conclusion.}
\end{proof}

{Finally, we extend the above approximation results to functions belonging to $(GSBV(\Omega))^m$ 
with energy $E_{|\cdot|^p,g_0}$ finite (we refer to \cite[Section 4.5]{AFP} for the basic notation and theory).
\begin{corollary}\label{theodensity GSBV}
 Let $\Omega\subseteq\R^n$ be an open bounded Lipschitz set, $u\in L^1(\Omega;\R^m)\cap(GSBV(\Omega))^m$ be such that 
 $\nabla u\in L^p(\Omega;\R^{m\times n})$ for some $p\in[1,\infty)$, and $g_0(|[u]|)\in L^1(\Omega;\calH^{n-1}\res J_u)$, with $g_0:[0,\infty)\to[0,\infty)$ continuous, nondecreasing, subadditive, and $g_0^{-1}(0)=\{0\}$.

Then, there exists a sequence $(u_j)_{j\in\N}\subseteq SBV\cap L^\infty(\Omega;\R^m)$ 
such that all the conclusions in Theorem~\ref{theodensityintro} hold, and in addition
\begin{equation}\label{e:Psi conv GSBV}
\lim_j\int_\Omega\Psi(\nabla u_j)dx=\int_\Omega\Psi(\nabla u)dx\,,
\end{equation}
\begin{equation} \label{e:g conv GSBV}
\lim_j\int_{J_{u_j}}g([u_j],\nu_{u_j})d\calH^{n-1}=\int_{J_{u}}g([u],\nu_u)d\calH^{n-1}
\end{equation}
for all functions $\Psi\in C^0(\R^{m\times n})$ satisfying \eqref{e:Psi growth}, 
and all $g\in C^0(\R^m\times S^{n-1};[0,\infty))$ satisfying \Hguno, \Hgdue, and {\Hgtre}. {In addition, if $\calH^{n-1}(J_u)<\infty$, \eqref{e:g conv GSBV} holds for all $g\in C^0(\R^m\times S^{n-1};[0,\infty))$ satisfying \Hguno, \Hgdue, and \Hgquattro.}

Moreover, the sequence $(u_j)_{j\in\N}$ can be chosen such that \eqref{eq:convgrad}
and \eqref{eq:convarea} hold.
\end{corollary}
\begin{proof}
We argue by density by constructing a sequence $(\tilde{u}_k)_{k\in\N}\subset SBV\cap L^\infty(\Omega;\R^m)$ converging in $L^1(\Omega;\R^m)$ and in energy to $u$.
This is well-known nowadays, in any case for the readers' convenience we recall the definition.
To this aim we fix a sequence $(a_k)_{k\in\N}\subset(0,\infty)$ such that $a_k<a_{k+1}$, $a_k\uparrow\infty$, and such that there are functions $\mathcal{T}_k\in {C^1_c}(\R^n;\R^m)$ satisfying $\mathcal{T}_k(z){=z}$ if $|z|\leq a_k$, $\mathcal{T}_k(z)=0$ if $ |z|\geq a_{k+1}$,
and $\|\Dnabla \mathcal{T}_k\|_{L^\infty(\R^n;\R^m)}\leq 1$.
Then, the sequence $\tilde{u}_k:=\mathcal{T}_k(u){\in SBV(\Omega;\R^m)}$ converges to $u$ in $L^1(\Omega;\R^m)$,
$\nabla \tilde{u}_k=\nabla u$ $\calL^n$-almost everywhere on $\Omega_k:=\{x\in\Omega:\,|u(x)|\leq a_k\}$, $J_{\tilde{u}_k}\subseteq J_u$, $\nu_{\tilde{u}_k}=\nu_u$ $\calH^{n-1}$-almost everywhere on $J_{\tilde{u}_k}$. Moreover, as
\[
\calH^{n-1}(\{x\in J_u:\,|u^\pm(x)|=\infty\})=0
\]
(see \cite[Proposition 2.12, Remark 2.13]{AlFoc}) we infer that 
$\chi_{J_{\tilde{u}_k}}\to \chi_{J_u}$, 
$\tilde{u}^\pm_k(x)=u^\pm(x)$  $\calH^{n-1}$-almost everywhere on $J_{\tilde{u}_k}\cap\Omega_k$,  $|[\tilde{u}_k]|\le |[u]|$ and $\tilde{u}_k^\pm\to u^\pm$
$\calH^{n-1}$-almost everywhere on $J_u$. Therefore, we get
\[
\lim_k\int_\Omega|\nabla \tilde{u}_k|^pdx=\int_\Omega|\nabla u|^pdx\,,
\]
and thanks to the subadditivity and monotonicity properties of $g_0$ also that
\[
\lim_k\int_{J_u}g_0(|[u]-[\tilde{u}_k]|)d\calH^{n-1}=0
\]
(for detailed proofs of similar properties see, for instance, \cite[Lemma 6.1]{AlFoc} and \cite[Proposition 4.8]{ContiFocardiIurlano2022}).

Next, for every $k\in\N$ we apply Theorem~\ref{theodensityintro} in order to get a sequence $(\tilde{u}_{k,j})_{j\in\N}$ approximating $\tilde{u}_k$ and satisfying all
the conditions in that statement. Eventually, we conclude thanks to a diagonalization argument in view of the properties of $g_0$, $\Psi$ and $g$ and the arguments in {Corollaries~\ref{theodensity Psi g} or \ref{theodensityMS}, and in Corollary~\ref{theodensity Psi g area}}.
\end{proof}
}

\section{Technical results}\label{s:techresults}
{In this {section} we collect the key technical tools we use to prove Theorem  \ref{theodensityintro}.
For $SBV$ functions having finite energy according to Theorem  \ref{theodensityintro}, we establish 
first an extension result, and then some measure theoretic properties crucial for our constructions.}
\subsection{Extension}\label{a:extension}
\newcommand\lipschitzg{{G}}
In this section we prove an extension result for $SBV$ functions. A standard local reflection argument would work for  
 $SBV\cap L^p$ functions. However, 
 {in the present setting it is not clear that having finite energy implies finiteness of the $L^p$ norm. In particular, {to the aim of applications}, both for approximation {via $\Gamma$-convergence} and for {the determination of} relaxation {of variational integrals,} it is not natural to assume additionally $u\in L^p$ ({see for example {the forthcoming paper}} \cite{CFI23}), therefore we} avoid the extra $L^p$
integrability condition. {To this aim}, we introduce a global reflection argument based on a bilipschitz map reflecting a neighborhood
of $\partial\Omega$ in $\Omega$, outside of $\Omega$ itself.

{The general strategy is standard, but to the best of our knowledge the details are new. For example, a similar result was obtained in \cite[Th.~3.1]{CagnettiScardia2011} with a more complex construction using the solution of an ODE 
(see \cite[Eq.~(3.5)]{CagnettiScardia2011} 
instead of the specific formula \eqref{eqdeffxt} below for the construction of the reflection.}

\begin{theorem}\label{theolipschitzset}
 Let $\Omega\subseteq\R^n$ be a bounded Lipschitz set. Then there are an open set $\omega\subseteq\R^n$ with $\partial\Omega\subset\omega$
  and a bilipschitz map $\Phi:\omega\to\omega$ such that $\Phi(x)=x$ for $x\in\partial\Omega$ and $\Phi(\omega\cap\Omega)=\omega\setminus\overline\Omega$.
\end{theorem}
We recall that a map $f:E\to f(E)\subseteq\R^n$, for $E\subseteq\R^n$, is bilipschitz 
if there is
 $L>0$ such that
\begin{equation}
 \frac1L |x- y|\le |f(x)-f( y)|\le L|x- y|
\end{equation}
for all $x, y\in E$. 
This is the same as saying that $f$ is
 injective, Lipschitz, with a Lipschitz inverse $f^{-1}:f(E)\to E$. 

Moreover, a set $\Omega$ is Lipschitz if for every $x\in\partial\Omega$ there are {$\eps_x>0$}, $\lipschitzg_x\in\Lip(\R^{n-1})$
and an isometry $A_x:\R^n\to\R^n$ such that $A_x0=x$ and
\begin{equation}\label{eqdeflipschitzset}
 B_{{\eps_x}}(x)\cap\Omega
 =B_{{\eps_x}}(x)\cap A_x\{(y',y_n)\in\R^{n-1}\times \R: y_n<\lipschitzg_x(y')\}.
\end{equation}
{Obviously $\lipschitzg_x(0)=0$; if $|y'|< \eps_x/(\Lip(\lipschitzg_x)+1)$ then ${A_x}(y',{\lipschitzg_x}(y'))\in B_{\eps_x}{(x)}\cap \partial\Omega$.}
If $\Omega$ is bounded, there are $\eps_0$ and $L_0$ such that 
{one can choose}
$\eps_x\ge \eps_0$ and $\Lip(\lipschitzg_x)\le L_0$
for all $x\in\partial\Omega$.

{We start with defining a smooth vector field playing the role of the normal field to $\partial\Omega$, {which under our hypotheses is only {a function in} $L^\infty(\partial \Omega;S^{n-1})$}.}
\begin{lemma}\label{lemmalipnormal}
 Let $\Omega\subseteq\R^n$ be a bounded Lipschitz set. Then there are $\gamma>0$ and a map $\psi\in C^\infty_c(\R^n;\R^n)$ such that
 $\psi(x)\cdot\nu(x)\ge \gamma$ for {$\calH^{n-1}$}-almost every $x\in\partial\Omega$ and $|\psi|=1$ on $\partial\Omega$.
\end{lemma}
This is well-known (see, for example, \cite[Lemma~4.1] {ContiMaggi2008}), for completeness we include the short proof.
\begin{proof}
{The compact set $\partial\Omega$ can be covered by a finite family of balls $B_i:=B_{r_i}({z_i})$, such that in each of the larger balls $B_i^*:=B_{2r_i}({z_i})$ 
\eqref{eqdeflipschitzset} reads
\begin{equation}
 B_i^*\cap \Omega=B_i^*\cap {A_i} \{(x',x_n): x_n<\lipschitzg_i(x')\}
\end{equation}
for some $\lipschitzg_i\in \Lip(\R^{n-1})$ {and isometry $A_i$}. 
If $x'$ is such that $y:={A_i}(x', \lipschitzg_i(x'))\in\partial\Omega\cap B_i^*$
and $\lipschitzg_i$ is differentiable at $x'$,
then the outer normal
obeys $\nu(y)=R_i (-D\lipschitzg_i(x'),1)/\sqrt{1+|D\lipschitzg_i|^2(x')}$, {where $R_i:=\Dnabla A_i\in \O(n)$,}
so that
\begin{equation}
\nu\cdot R_ie_n\ge \gamma^*
:=\frac{1}{\sqrt{1+\max_i (\Lip(\lipschitzg_i))^2}}>0
\end{equation}
$\calH^{n-1}$-almost everywhere on $B_i^*\cap\partial\Omega$.
We fix cutoff functions  $\theta_i\in C^\infty_c(B^*_i;[0,1])$ with $\theta_i=1$ on $B_i$, and set $\psi^*(x):=\sum_i \theta_i(x) R_i e_n$. Then for {$\calH^{n-1}$-}almost every point $x\in \partial\Omega$ we have
\begin{equation}\label{eqpsixnux}
\psi^*(x)\cdot\nu(x)=\sum_{i: x\in B^*_i} \theta_i(x) (R_i e_n)\cdot\nu(x)
\ge \sum_{i: x\in B^*_i} \theta_i(x) \gamma^*\ge\gamma^*,  
\end{equation}
since at least one of the $\theta_i(x)$ equals 1.}

 {It only remains to normalize. Condition \eqref{eqpsixnux} implies $|\psi^*|\ge\gamma^*$ on $\partial\Omega$, and therefore} $|\psi^*|> \gamma^*/2$ in a neighborhood of $\partial\Omega$.
 We select
 $\varphi\in C^\infty(\R^n;[0,1])$ such that $\varphi=0$ on $\partial\Omega$, and $\varphi>0$ on the set $|{\psi^*}|\le\gamma^*/2$. Then $\psi:=\psi^*/\sqrt{|\psi^*|^2+\varphi}$ has the desired properties {with $\gamma:=\gamma^*/\max|\psi^*|(\partial\Omega)$.}
\end{proof}

\begin{lemma}\label{lemxyPxy} Let $\Omega\subseteq\R^n$ be a bounded Lipschitz set, $\psi$ as in  Lemma~\ref{lemmalipnormal}. 
 There are $\rho>0$ {and} $c>0$ such that for any $x,y\in\partial\Omega$ with $|x-y|<\rho$
 one has
 \begin{equation}\label{eqxyPxy}
  |x-y|\le c |(\Id-\psi(y)\otimes \psi(y))(x-y)|.
 \end{equation}
\end{lemma}
\newcommand\muno{m_\Pi}
\newcommand\mdue{m_\perp}
\newcommand\nuuno{\nu_\Pi}
\newcommand\nudue{\nu_\perp}
\begin{proof}
We can assume $x\ne y$.
After a change of coordinates, and choosing $\rho$ sufficiently small, we can assume that $y=0$, 
and that
\begin{equation}\label{eqlippro}
\partial\Omega\cap B_{(1+L)\rho}(0)=B_{(1+L)\rho}(0)
\cap \{(z', \lipschitzg(
z')): z'\in \R^{n-1}\}
\end{equation}
for some $L$-Lipschitz function $\lipschitzg:\R^{n-1}\to\R$.
The values of $\rho$ and $L$ have bounds that depend only on $\Omega$.
Let $m:=\psi(0)$, so that $P:=\Id-m\otimes m$ is the projection onto the space orthogonal to $m$.  
Condition \eqref{eqxyPxy} then translates into
 \begin{equation}\label{eqxPx}
  |x|\le c |Px|
 \end{equation}
 {for any $x\in\partial\Omega\cap {B_\rho(0)}$. As both sides of  \eqref{eqxPx} are continuous, it suffices to prove it for $\calH^{n-1}$-almost every $x$. By  \eqref{eqlippro},}
we have $x=(x',\lipschitzg(x'))$ for some $x'\in\R^{n-1}\setminus\{0\}$.  
 Define $\hat x:={(\hat x',0):=(\frac{x'}{|x'|},0)}\in S^{n-2}{\times\{0\}}\subset\R^{n-1}{\times\{0\}}$.
 Let $\Pi:=\Span\{\hat x, e_n\}$, we remark that $\hat x\cdot e_n=0$ and that $x\in \Pi$.
 Let $\muno$ be the orthogonal projection of $m$ on $\Pi$, namely
 \begin{equation}\label{e:muno}
 \muno:=(\hat x\otimes \hat x + e_n\otimes e_n)m 
 =(\hat x\cdot m)\hat x +(e_n\cdot m)e_n
 \in\Pi, 
 \end{equation} 
 and $\mdue:=m-\muno$, so that $m=\muno+\mdue$.
Then, {as $x\in\Pi$ and $m_\perp\in\Pi^\perp$ we have}
\begin{equation}\label{eqPxmumd}
 |Px|^2=|(\Id-(\muno+\mdue)\otimes (\muno+\mdue))x|^2
 =|x-(\muno\cdot x)\muno|^2+|\mdue|^2|\muno\cdot x|^2.
\end{equation}
We distinguish two cases. If $|\mdue|\ge \frac12\gamma$, {with $\gamma$ the constant from 
Lemma~\ref{lemmalipnormal},}
the first term leads to
\begin{equation}
 |Px|\ge |x-(\muno\cdot x)\muno|\ge |x|-|\muno|^2|x| = |\mdue|^2|x|
 \ge \frac{\gamma^2}{4}|x|
\end{equation}
which concludes the proof of \eqref{eqxPx} in this case. 

Assume now that $|\mdue|\le \frac12\gamma$.
{For any $z'\in\R^{n-1}$ with $|z'|<\rho$ we have
$z:=(z',\lipschitzg(z'))\in\partial\Omega\cap B_{\rho(1+L)}(0)$, and if $\lipschitzg$ is differentiable in $z'$ the outer normal is}
\begin{equation}\label{eqnormalz}
 \nu({z})=\frac{1}{\sqrt{1+|D\lipschitzg|^2(z')}} \begin{pmatrix} -D\lipschitzg(z')\\ 1 \end{pmatrix}.
\end{equation}
{Recalling $\psi(z)\cdot \nu(z)\ge\gamma$, }
\begin{equation}\begin{split}
 m\cdot\nu{(z)} 
 {=\psi(z)\cdot\nu{(z)}+(\psi(0)-\psi(z))\cdot \nu(z)
 \ge\gamma - \|\psi\|_{C^1}\,|z|}
\end{split}\end{equation}
{so that, if $\rho<\gamma/(4\|\psi\|_{C^1}(1+L))$,
the condition $|\mdue|\le \frac12\gamma$
implies}
\begin{equation}\label{eqmunonugamma}\begin{split}
 \muno\cdot\nu{(z)} 
 =
 m\cdot\nu{(z)} -\mdue\cdot\nu{(z)}  \ge \gamma-
 {\frac14\gamma}-
 |\mdue|\ge\frac1{4}\gamma.
\end{split}\end{equation}
{By \eqref{eqnormalz} and} \eqref{e:muno},
\begin{equation}\label{eqmunonu}
 \muno\cdot\nu(z)= \frac{1}{\sqrt{1+|D\lipschitzg|^2(z')}} \left[  (e_n\cdot m)
 -(\hat x\cdot m){\hat x'}\cdot D\lipschitzg(z') \right].
\end{equation}
With a slight abuse of notation we have used the dot to denote the inner product both in $\R^{n-1}$ and $\R^n$.

Let $\zeta(t):=(t{\hat x'} ,\lipschitzg(t {\hat x'}))$, for $t\in[0,|x'|]$.
{For $\calH^{n-1}$-almost every choice of $x'\in B'_\rho$ we have that for $\calH^1$-almost every $t$ the function $\lipschitzg$ is differentiable in $t\hat x'$.}
Clearly $\zeta(0)=0$, $\zeta(|x'|)=x$, 
{and $\zeta$ is Lipschitz with}
\begin{equation}
 D\zeta(t)=\hat x + ({\hat x'\cdot}D\lipschitzg(t{\hat x'})) e_n.
\end{equation}
We define
\begin{equation}
 \muno^\perp:=(e_n\otimes \hat x-\hat x\otimes e_n) m =  (\hat x\cdot m) e_n-(e_n\cdot m)\hat x\in\Pi\,,
\end{equation}
and compute
\begin{equation}
\begin{split}
 \muno^\perp \cdot x &= 
 \int_0^{|x'|} \muno^\perp \cdot {D\zeta}(t)dt
 = \int_0^{|x'|} \left[(\hat x\cdot m) {(\hat x'\cdot D\lipschitzg(t{\hat x'}))} - (e_n\cdot m) \right]
 dt.\end{split}
\end{equation}
Using first \eqref{eqmunonu} and then \eqref{eqmunonugamma},
\begin{equation}
\begin{split}
 \muno^\perp \cdot x &= 
-\int_0^{|x'|} \muno\cdot\nu(\zeta(t)) \sqrt{1+|D\lipschitzg|^2({t \hat x'})} dt
 \le - |x'|\frac\gamma{4}.
 \end{split}
\end{equation}
With $|\muno^\perp |{\le} 1$ and 
$\muno^\perp\cdot m=0$ (note that $m_\Pi^\perp\in\Pi$ and $m_\Pi^\perp\cdot m_\Pi=0$) we obtain
\begin{equation}
 |Px|\ge{\left|x\cdot \frac{m_\Pi^\perp}{|m_\Pi^\perp|}\right|{\ge} |x\cdot m_\Pi^\perp|}
 \ge \frac\gamma{ 4}
 |x'|\,.
\end{equation}
Recalling that $|x|\le |x'|+|\lipschitzg(x')|\le (1+L)|x'|$, this concludes the proof of \eqref{eqxPx}, and therefore of \eqref{eqxyPxy}
{(with $c=4(1+L)/\gamma^2$).}
  \end{proof}
  
{We are now ready to prove Theorem~\ref{theolipschitzset}. Before starting, we}
recall that by Brower's invariance of domain theorem any injective continuous map $f:E\subseteq\R^n\to\R^n$ is open, in the sense that if $E$ is open then $f(E)$ is open.

\begin{proof}[Proof of Theorem~\ref{theolipschitzset}.] {\em Step 1.}
Let $\psi$ be as in Lemma~\ref{lemmalipnormal}, 
{$\rho$ as in Lemma~\ref{lemxyPxy},}
and 
 define $f:\partial\Omega\times \R\to\R^n$ by 
 \begin{equation}\label{eqdeffxt}
 f(x,t):=x+t\psi(x).
 \end{equation}
 We claim that there are $\eps>0$ and $C>0$, depending only on $\Omega$, such that 
for all $x,y\in\partial\Omega$, all $t,s\in(-\eps,\eps)$,
 \begin{equation}\label{eqxytsf}
  |x-y|+|t-s|\le C | f(x,t)-f(y,s)|.
 \end{equation}
 In order to prove \eqref{eqxytsf} we write
 \begin{equation}\label{eqfxytdsfno}
  f(x,t)-f(y,s)=x-y+t\psi(x)-s\psi(y).
 \end{equation}
We shall choose $\eps\le\rho$. If $|x-y|\le\rho$ we can use Lemma~\ref{lemxyPxy}.
 Let $P_y$ be the projection onto $\psi(y)^\perp$. Then
\begin{equation}
 P_y( f(x,t)-f(y,s))=P_y(x-y)+tP_y(\psi(x)-\psi(y)).
 \end{equation}
 For the second term we use 
 $|\psi(x)-\psi(y)|
 \le  \|\psi\|_{C^1} |x-y|$.
  For the first term, we use \eqref{eqxyPxy}. We then obtain
  \begin{align}\label{e:f bilipschitz}
   |f(x,t)-f(y,s)|&\ge |P_y(x-y)|-|t|\, |\psi(x)-\psi(y)|\notag\\
   &\ge \frac1c |x-y| - \eps \|\psi\|_{C^1} |x-y|
   \ge \frac1{2c} |x-y|
  \end{align}
provided that $\eps\le 1/(2c\|\psi\|_{C^1})$.
To estimate $t-s$ we write \eqref{eqfxytdsfno} as
\begin{equation}\label{fxtfyss}
   f(x,t)-f(y,s)=
   (t-s)\psi(x)+
   x-y+s(\psi(x)-\psi(y))
\end{equation}
so that
\begin{equation}\label{eq:t-s}
|t-s|\le |f(x,t)-f(y,s)|+
   |x-y|+|s|\,\|\psi\|_{C^1}\,|x-y|
\end{equation}
which, recalling that $|s|\|\psi\|_{C^1}\le 1/(2c)$, 
{together with}
{\eqref{e:f bilipschitz}}
concludes the proof of  \eqref{eqxytsf} in this case.

 Assume now that $|x-y|>\rho$, still with $|t|,|s|<\eps$. Then \eqref{eqfxytdsfno} gives
 \begin{equation}
 | f(x,t)-f(y,s)|\ge |x-y|-|t|-|s|
 \ge \frac12 |x-y| +\frac12\rho-2\eps.
 \end{equation}
Choosing $\eps\le\rho/4$ {and recalling \eqref{eq:t-s}, this} concludes the proof of  \eqref{eqxytsf}.
 
 {\em Step 2.}
 We define 
 $\omega:=f(\partial\Omega\times (-\eps,\eps))$ {and check that $f|_{\partial\Omega\times(-\varepsilon,\varepsilon)}$ is bilipschitz.}
 Let $y,y'\in\omega$. Then there are $x,x'\in\partial\Omega$, $t,t'\in (-\eps,\eps)$, such that
 \begin{equation}
  y=f(x,t), \hskip5mm y'=f(x',t'),
 \end{equation}
which by \eqref{eqxytsf} {and} \eqref{fxtfyss} in Step 1  implies
\begin{equation}
 \frac1C (|x-x'|+|t-t'|)\le  |y-y'| \le C (|x-x'|+|t-t'|).
\end{equation}
{Hence $f|_{\partial\Omega\times(-\varepsilon,\varepsilon)}$ is bilipschitz.
Let now} $\Phi:\omega\to\omega$ be
\begin{equation}
\Phi(y):=f(f_x^{-1}(y), -f_t^{-1}(y))
\end{equation}
where $f^{-1}_x$ and $f^{-1}_t$ denote the components of $f^{-1}$. 
Obviously $\Phi(y)=y$ if $y\in\partial\Omega\subseteq\omega$. 
We check that $\Phi$ is bilipschitz. {Indeed, arguing as above, setting}
 \begin{equation}
  Y:=f(x,-t)=\Phi(y), \hskip5mm Y':=f(x',-t')=\Phi(y'),
 \end{equation}
{we get}
\begin{equation}
 \frac1C (|x-x'|+|t-t'|)\le  |Y-Y'| \le C (|x-x'|+|t-t'|),
\end{equation}
therefore
$\Phi$ is bilipschitz.

It remains to show that $\omega$ is open {if $\eps$ is sufficiently small.}
Assume $\eps\le\eps_0/{(1+L_0)}$, with $\eps_0$, $L_0$ {the quantities introduced right after} \eqref{eqdeflipschitzset}.
{Select $y\in \omega$, and let}
 $x\in\partial\Omega$, $s\in (-\eps,\eps)$
 {be such that $y=f(x,s)$.}
  Choose
 $A_x$, {$G_x$} {as} in \eqref{eqdeflipschitzset} and
let $h:B'_{\eps_0/{(1+L_0)}}\times (-\eps,\eps)\to{\omega}$ be defined by
\begin{equation}
 h(z',t):=f( A_x(z',{\lipschitzg_x}(z')), t).
\end{equation}
The map $h$ is injective and Lipschitz, and therefore open.
Therefore $\omega$ contains an open neighborhood of $y=f(x,s)$.
\end{proof}
{\begin{remark}If $\psi$ were only $L^\infty$, the map $f(x,t)$ may not be invertible in
$\partial\Omega\times(-\eps,\eps)$, for all choices of $\eps>0$. This happens for example if
$\partial\Omega$ is (locally) the graph of the function $\chi_{(0,1)}{(x)|x|}^{1+\alpha}$ for $x\in (-1,1)$, where $0<\alpha<1$ {(see \eqref{e:f bilipschitz}).}
	\end{remark}}

\begin{theorem}\label{propextension}
	Let $\Omega\subseteq\R^n$ {be} a bounded {open} Lipschitz set, $\theta>0$, and $u\in SBV(\Omega;\R^m)$ such that
$E_{|\cdot|^p,g_0}[u,\Omega]<\infty$, for some $p\geq1$ and $g_0$ satisfying \Hgzerouno-\Hgzerodue.
Then there are an open set $\Omega'$ with $\overline\Omega\subset \Omega'$ and {$|\Omega'|\leq|\Omega|+\theta$}, 
and a function $U\in SBV(\R^n;\R^m)$ such that
	$U=u$ on $\Omega$, $|DU|(\partial\Omega)=0$, 
{
\begin{equation}\label{e:Psi U}
\int_{\Omega'}|\nabla U|^p\,dx\leq\int_{\Omega}|\nabla u|^p\,dx+\theta\,,
\end{equation}
and
\begin{equation}\label{e:g0 U}
\int_{\Omega'\cap J_U}g_0(|[U]|)d\calH^{n-1}\leq\int_{J_u}g_0(|[u]|)d\calH^{n-1}+\theta\,.
\end{equation}
}
In particular, 
	\begin{equation}
	E_{|\cdot|^p,g_0}[U,\Omega']\le E_{|\cdot|^p,g_0}[u,\Omega]+2\theta.
	\end{equation}
	If $\calH^{n-1}(J_u)<\infty$, then 
	additionally 
	{$\calH^{n-1}(J_U)<\infty$ and}
	$\calH^{n-1}(J_U\cap\Omega'\setminus\Omega)<\theta$; if $\nabla u=0$ $\calL^n$-almost everywhere on $\Omega$ then also $\nabla U=0$ $\calL^n$-almost everywhere on $\Omega'$. If $u\in W^{1,p}(\Omega;\R^m)$ then $U\in W^{1,p}(\Omega';\R^m)$.
\end{theorem}
\begin{proof} 
We consider the open set $\omega$ and the bilipschitz map $\Phi$ provided by Theorem \ref{theolipschitzset}.
Thus, \cite[Theorem 3.16]{AFP} yields that $u\circ\Phi^{-1}\in SBV(\omega\setminus\overline{\Omega};\R^m)$, 
with 
\[
 |\operatorname{Lip}(\Phi)|^{1-n}\Phi_\#|D(u|_{\Omega\cap\omega})|\leq
|D(u\circ\Phi^{-1})|\leq |\operatorname{Lip}(\Phi^{-1})|^{n-1}\Phi_\#|D(u|_{\Omega\cap\omega})|\,,
\] 
where $\Phi_\#$ denotes the push forward of measures. In particular, by the Coarea formula
(cf. \cite[Theorem 2.93]{AFP}) we conclude that 
\begin{equation}\label{e:Psi U 0}
\int_{\omega\setminus\overline{\Omega}}|\nabla(u\circ\Phi^{-1})|^p\,dx\leq
C\int_{\Omega\cap\omega}|\nabla u|^p\,dx\,,
\end{equation}
and by the Coarea formula between rectifiable sets (cf. \cite[Theorem 3.2.22]{Federer69})
\begin{equation}\label{e:g0 U 0}
\int_{J_{u\circ\Phi^{-1}}}g_0(|[u\circ\Phi^{-1}]|)d\calH^{n-1}\leq C\int_{J_u\cap(\Omega\cap\omega)}g_0(|[u]|)d\calH^{n-1}\,,
\end{equation}
for a constant $C$ depending only on $n$ and $\operatorname{Lip}(\Phi^{-1})$.

Having fixed $\theta>0$, up to restricting $\omega$, we may assume that both right-hand sides of \eqref{e:Psi U 0}
and \eqref{e:g0 U 0} are actually less than or equal to $\theta$, and in addition that $\omega$ is Lipschitz with
$|\omega\setminus\overline\Omega|\leq\theta$. 

Then, to conclude, set $\Omega':=\Omega\cup\omega$ and
\begin{equation}
U(x):=\begin{cases}
u(x) &  x\in\Omega, \cr
u(\Phi^{-1}(x)) &  x\in\omega\setminus\overline{\Omega},\cr
0 & x\in\R^n\setminus\overline{\Omega'}\,.
\end{cases}
\end{equation}
By construction $U=u$ on $\Omega$. Moreover, recalling that $\Phi|_{\partial\Omega}$ is the identity and that $\omega$ 
is Lipschitz, \cite[Corollary~3.89]{AFP} implies that $U\in SBV(\R^n;\R^m)$, $|DU|(\partial\Omega)=0$, and 
moreover that
$DU\res\Omega'=Du\LL\Omega+D(u\circ\Phi^{-1})\LL (\omega\setminus\overline{\Omega})$. Finally, estimates \eqref{e:Psi U} and \eqref{e:g0 U} readily follows
from \eqref{e:Psi U 0} and \eqref{e:g0 U 0}. 

In the case
$\calH^{n-1}(J_u)<\infty$ the additional estimate follows from the same proof,  using $1+g_0$ in place of $g_0$ in \eqref{e:g0 U 0}. Similarly, if either $\nabla u=0$ or $J_u=\emptyset$, the same property is immediately inherited by $U$ on $\Omega\cup\omega$.
\end{proof}

\subsection{Approximate regularity on an intermediate scale}\label{s:approxreg}

{Given a $SBV$ function satisfying the hypotheses of Theorem \ref{theodensityintro}, we} show that at an intermediate scale, denoted by $\delta$, the regular part of the gradient is uniformly approximately continuous, and the jump is approximately given by a fixed jump concentrated on a $C^1$ manifold.
Having fixed 
$g_0$ satisfying \Hgzerouno-\Hgzerodue, for every $u\in SBV(\Omega;\R^m)$, we introduce the notation
\begin{equation}\label{e:muj}
\muj:=
g_0(|[u]|)\calH^{n-1}\LL J_u\,.
\end{equation}
We remark that $\muj$ is a finite measure on $\Omega$ concentrated on a $\sigma$-finite set with respect to {$\calH^{n-1}\LL J_u$.}
The main result is the following. 
\newcommand{\Qlargez}{{Q^*_z}}
\begin{proposition}\label{propdeltafromtheta}
Let $\Omega\subseteq\R^n$ be open {and bounded},
{$p\in[1,\infty)$,}
and $g_0$ be satisfying \Hgzerouno-\Hgzerodue.
There is a constant $C$ depending on $p$, $n$ {and $m$}, such that
for every $u\in SBV(\Omega;\R^m)$ with $\muj(\Omega)<\infty$ and
{$\nabla u\in L^p(\Omega;\R^{m\times n})$, and for}
every $\theta>0$,
after fixing an orientation of $J_u$
there is $\delta\in(0,\theta]$ such that, setting 
$A_\delta:=\{z\in \Omega\cap \delta\Z^n: \dist(z,\partial\Omega{)}>{\delta\sqrt {n}}\}$
and
$\Qlargez:=z+(-\delta,\delta)^n$,
the following holds:
there are
$R:{A_\delta}\to\SO(n)$,  $s:{A_\delta}\to \R^m$,  $\eta:{A_\delta}\to\R^{m\times n}$, $\varphi:{A_\delta}\to {C^1_c}(\R^{n-1})$, and
{$x:{A_\delta}\to\R^n$}
such that, setting
\begin{equation}
 L_z:=x_z+R_z\{(y',\varphi_z(y')): y'\in\R^{n-1}\},
\end{equation}
one has $\|D\varphi_z\|_{L^\infty}\le \theta$ and
\begin{equation}
\begin{split}\label{e:stima deltafromtheta}
\sum_{z\in {A_\delta}} &\int_{\Qlargez}|\nabla u-\eta_z|^p dx +
\sum_{z\in {A_\delta}}
\int_{\Qlargez\cap J_u\setminus L_z} g_0(|[u]|) d\calH^{n-1}\\
&{+\sum_{z\in {A_\delta}}\int_{\Qlargez\cap L_z}
g_0(|[u]|)|\nu_u-R_ze_n| d\calH^{n-1}}\\
&+\sum_{z\in {A_\delta}}\int_{\Qlargez\cap L_z}
g_0(|[u]-s_z|) d\calH^{n-1}
\le C\theta{(1+\mu_u(\Omega)+|\Omega|)}\,.
\end{split}
\end{equation}
If $\calH^{n-1}(J_u)<\infty$ then additionally 
\begin{equation}\label{eqfinitejump}
\sum_{z\in {A_\delta}}\calH^{n-1}(\Qlargez\cap (J_u{\triangle} L_z))\le C \theta.
\end{equation}
\end{proposition}

Before proving it we introduce a {preliminary} pointwise result {for the jump part of the energy}. 
\begin{lemma}\label{lemmaBV}
 Let $\Omega\subseteq\R^n$ {be} open, and {$g_0$} be satisfying \Hgzerouno-\Hgzerodue.
Let  $u\in SBV(\Omega;\R^m)$ with $\muj(\Omega)<\infty$. Then, for $\calH^{n-1}$-almost every $x\in J_u$ there are 
$R_x\in\SO(n)$, $s_x\in\R^m\setminus\{0\}$, 
 $\varphi_x\in C^1(\R^{n-1})$ such that $\varphi_x(0)=0$, $D\varphi_x(0)=0$, 
 and, letting $L_x:=x+R_x\{(y',\varphi_x(y')): y'\in\R^{n-1}\}$,
\begin{equation}\label{eqlemmaBVc}
\begin{split}
  \lim_{r\to0} \frac{1}{\muj(B_r{(x)})} &\left[ \int_{B_r(x)\cap J_u\setminus L_x} {g_0(|[u]|)}
d\calH^{n-1}\right.\\
&\left.+\int_{B_r(x)\cap L_x}
\big(g_0(|[u]-s_x|)+
{g_0(|[u]|)|\nu_u-R_xe_n|}\big)\, d\calH^{n-1}
\right]=0
.\end{split}
\end{equation}
If $\calH^{n-1}(J_u)<\infty$, then additionally
\begin{equation}\label{lemmaBVfin}
  \lim_{r\to0} \frac{\calH^{n-1}(B_r(x)\cap (J_u\triangle L_x))}{\muj(B_r{(x)})}=0.
\end{equation}
\end{lemma}
\begin{proof}
We first observe that for $\calH^{n-1}$-almost every $x\in J_u$ 
{by \cite[Th.~2.83(i)]{AFP}}
we have
\begin{equation}\label{eqmuug0}
 \lim_{r\to0}\frac{\muj(B_r{(x)})}{\omega_{n-1}r^{n-1}}={g_0(|[u](x)|)}\neq 0,\end{equation}
therefore {to prove \eqref{eqlemmaBVc}} it suffices to show that
{for $\calH^{n-1}$-almost every $x\in J_u$
there is 
$L_x$ as stated such that, setting $s_x:=[u](x)$, one has 
\begin{equation}\label{eqconvgusx1}
   \lim_{r\to0}\frac{1}{r^{{n-1}}}
  \int_{B_r(x)\cap L_x} {g_0(|[u]-s_x|)}
d\calH^{n-1}  =0,
\end{equation}
\begin{equation}\label{eqlemrn}
  \lim_{r\to0} \frac{1}{r^{n-1}}  
  \int_{B_r(x)\cap J_u\setminus L_x} {g_0(|[u]|)}d\calH^{n-1}=0,
\end{equation}
{and
\begin{equation}\label{eqconvgusx1nu}
   \lim_{r\to0}\frac{1}{r^{{n-1}}}
  \int_{B_r(x)\cap L_x\cap J_u} |\nu_u-R_xe_n|
d\calH^{n-1}  =0.
\end{equation}}%
We first observe that \eqref{eqconvgusx1} implies \eqref{eqlemrn}. Indeed, subadditivity and monotonicity of $g_0$ imply
$g_0(|s_x|)\le g_0(|[u]|)+g_0(|[u]-s_x|)$ and therefore
\begin{equation}\begin{split}
&\mu_u(  B_r(x)\setminus L_x)
 =
\mu_u(  B_r(x))-\mu_u(B_r(x)\cap L_x)\\
\le& 
\mu_u(  B_r(x))-g_0(|s_x|)\calH^{n-1}(B_r(x)\cap L_x)
+ \int_{B_r(x)\cap L_x} {g_0(|[u]-s_x|)}
d\calH^{n-1}.
\end{split}
\end{equation}
We divide by $\omega_{n-1}r^{n-1}$ and take the $\limsup$ as $r\to0$ to obtain,
with \begin{equation}
\lim_{r\to0} \frac{\calH^{n-1}(B_r(x)\cap L_x)}{\omega_n r^{n-1}}=1,
         \end{equation}
 \eqref{eqmuug0} and \eqref{eqconvgusx1},
\begin{equation}
 \limsup_{r\to0} \frac{1}{\omega_{n-1}r^{n-1}}\mu_u(  B_r(x)\setminus L_x)
 =0,
\end{equation}
which is \eqref{eqlemrn}. Therefore it remains to prove
\eqref{eqconvgusx1}.}

 For any $j>0$ let $A_j:=\{x\in J_u:\, |[u](x)|\ge 2^{-j}\}$. As $u\in SBV(\Omega;\R^m)$, we have $\calH^{n-1}(A_j)<\infty$, and $A_j$ is {countably $(n-1)$-}rectifiable. Therefore for $\calH^{n-1}$-almost every $x\in A_j$ there are {$R_x^j$, $\varphi_x^j$}  as in the statement such that
 {the corresponding set $L_x^j$ obeys}
 \begin{equation}\label{eqajrectif}
  \lim_{r\to0}\frac{\calH^{n-1}(B_r(x)\cap (A_j{\triangle} {L_x^j}))}{r^{n-1}}=
  0
 \end{equation}
 {and $\nu_u(x)=R_x^je_n$.}
As $|[u]|\in L^1(\Omega;\calH^{n-1}\LL J_u)$, and $\calH^{n-1}(A_j)$ is finite,
for $\calH^{n-1}$-almost every $x\in A_j$ 
\begin{equation}\label{eqconvgusxa1bbb}
   \lim_{r\to0}\frac{1}{r^{n-1}}
  \int_{A_j\cap B_r(x)} |[u]-s_x|d\calH^{n-1}  =0
\end{equation}
{and similarly
\begin{equation}\label{eqconvgusxa1bbbnu}
   \lim_{r\to0}\frac{1}{r^{n-1}}
  \int_{A_j\cap B_r(x)} |\nu_u-\nu_u(x)|d\calH^{n-1}  =0.
\end{equation}
}We recall that for $x\in J_u$ we defined $s_x=[u](x)\ne0$.
We first show that \eqref{eqajrectif} 
and \eqref{eqconvgusxa1bbb} imply
\begin{equation}\label{eqconvgusxa1}
   \lim_{r\to0}\frac{1}{r^{n-1}}
  \int_{{L_x^j}\cap B_r(x)} |[u]-s_x|d\calH^{n-1}  =0
\end{equation}
{for {$\calH^{n-1}$-}almost every $x\in A_j$.} Indeed, we have
\begin{align*}
   \int_{{L_x^j}\cap B_r(x)} |[u]-s_x|d\calH^{n-1}&\leq
   \int_{A_j\cap B_r(x)} |[u]-s_x|d\calH^{n-1}\\ & + 
(|s_x|{+2^{-j})}\calH^{n-1}(({L_x^j} \setminus A_j)\cap B_r(x))
\end{align*}
and the conclusion then follows from \eqref{eqajrectif} and \eqref{eqconvgusxa1bbb}.

{We next show that} \eqref{eqconvgusxa1} implies that for $\calH^{n-1}$-almost every $x\in A_j$
\begin{equation}\label{eqconvgusx1j}
   \lim_{r\to0}\frac{1}{r^{{n-1}}}
  \int_{{L_x^j}\cap B_r(x)} {g_0(|[u]-s_x|)}
d\calH^{n-1}  =0.
\end{equation}
 Indeed, using \eqref{e:g0 growth}, namely that for every $\lambda>0$ there is $C_{\lambda}>0$ such that
 ${g_0}(t)\le \lambda + C_\lambda t$ for all $t\in [0,\infty)$, with
\eqref{eqconvgusxa1} we obtain 
 \begin{equation}
   {\limsup_{r\to0}}\frac{1}{r^{n-1}} \int_{B_r(x)\cap {L_x^j}} {g_0(|[u]-s_x|)}
d\calH^{n-1}
  \le \lambda \lim_{r\to0}\frac{\calH^{n-1}(B_r(x)\cap {L_x^j})}{r^{n-1}}
=\lambda\omega_{n-1}.
\end{equation}
 Since $\lambda$ was arbitrary this concludes the proof of \eqref{eqconvgusx1j}.
 
Let $N\subseteq J_u$ be an $\calH^{n-1}$-null set such that \eqref{eqconvgusx1j} holds for all $x\in J_u\setminus N$ and all $j$ such that $x\in A_j$. For any $x\in J_u\setminus N$ we define $L_x$ as $L_x^j$ for the smallest $j\in N$ such that $x\in A_j$.
This proves \eqref{eqconvgusx1}.
{Condition \eqref{eqconvgusx1nu} follows similarly from
\eqref{eqconvgusxa1bbbnu} using
$A_j\subseteq J_u$ and
\eqref{eqajrectif}.}

 Assume now that $\calH^{n-1}(J_u)<\infty$. Then we can replace 
 \eqref{eqajrectif} by
  \begin{equation}\label{eqajrectif2}
  \lim_{r\to0}\frac{\calH^{n-1}(B_r(x)\cap (J_u{\triangle} L_x))}{r^{n-1}}=
  \lim_{r\to0}\frac{\calH^{n-1}(B_r(x)\cap (A_j{\triangle} L_x))}{r^{n-1}}=
  0.
 \end{equation}
 The second equality leads as above to \eqref{eqlemmaBVc};
 from the first one, one immediately obtains
\eqref{lemmaBVfin} (using again \eqref{eqmuug0}).
 \end{proof}

\begin{proof}[Proof of Proposition~\ref{propdeltafromtheta}]
As $\nabla u\in L^p(\Omega;\R^{m\times n})$, there is $f\in C^0(\overline\Omega;\R^{m\times n})$ such that
$\|\nabla u-f\|_{L^p(\Omega)}^p\le\theta$. 
Let ${\delta}>0$ be such that $|f(x)-f(y)|^p\le \theta$ for all $ x,y\in\overline\Omega$ with $|x-y|\le { \delta\sqrt {n}}$.
For any 
$z\in {A_\delta}$ we set $\eta_z:=f(z)$ and obtain
\begin{equation}\label{e:stima gradiente Lp}
\sum_{z\in {A_\delta}} \int_{\Qlargez}|\nabla u-\eta_z|^p dx 
\le {2^{p-1}}
\sum_{z\in {A_\delta}} \int_{\Qlargez}(|\nabla u-f|^p +
|f-\eta_z|^p) dx 
\le {2^{p+n}}(1 + |\Omega|)\theta
\end{equation}
as any point $x\in\R^n$ belongs to at most $2^n$ of the cubes $\Qlargez$, $z\in  {A_\delta}$.
This treats the first term.

The jump terms are treated using Lemma~\ref{lemmaBV}.
For $\muj$-almost every $x\in \Omega$ there are $\hat R_x\in \SO(n)$, {$\hat s_x\in\R^m\setminus\{0\}$,} and $\hat \varphi_x\in C^1(\R^{n-1})$ as stated, we define $\hat R_x:=\Id$, {$\hat s_x:=0$}, and $\hat \varphi_x:=0$ on the others.
We recall that $\hat L_x=x+\hat R_x\{(y',\hat \varphi_x(y')): y'\in\R^{n-1}\}$, and  define, for any $x\in\Omega$, the measure
\begin{equation}\label{eqdefhatnux}
\begin{split}
\hat m_x:=&g_0(|[u]|)\calH^{n-1}\LL(J_u\setminus \hat L_x)
 + g_0(|[u]- {\hat s_x}|)\calH^{n-1}\LL \hat L_x\\
& {+g_0(|[u]|)\,|\nu_u-\hat R_xe_n|\,\calH^{n-1}\LL \hat L_x}
 \,.
 \end{split}
\end{equation}
By Lemma~\ref{lemmaBV}, for $\muj$-almost every $x\in \Omega$ 
\begin{equation}\label{eqnuxmuj}
 \lim_{r\to0} \frac{{\hat m_x}(B_r(x))}{\muj(B_r(x))}=0
\text{ and } 
\lim_{r\to0} \|{D}\hat\varphi_x\|_{L^\infty(B'_{{3r}})}=0
\end{equation}
{(we write $B'$ for balls in $\R^{n-1}$).}
We define, {for $k\in\N_{>0}$,}
\begin{equation}\begin{split}
 E^{k,\theta}:=&\{x\in \Omega: 
 \|D\hat \varphi_x\|_{L^\infty(B'_{{\frac 3k}})}>{\frac13\theta}\} \\
 &\cup \{x\in \Omega: 
 \exists r\in (0,\frac1k ] \text{ with } 
 \hat m_x({B_r(x)\cap\Omega})\ge\theta \muj({B_r(x)\cap\Omega})
 \}.
 \end{split}
\end{equation}
Obviously $E^{k',\theta}\subseteq E^{k,\theta}$ if $k<k'$. By \eqref{eqnuxmuj}, for $\muj$-almost every $x\in {\Omega}$ there is $k$ such that $x\not\in E^{k,\theta}$.  
Therefore
\begin{equation}
 \muj(\bigcap_{k\in\N} E^{k,\theta})=0.
\end{equation}
We select $k_\theta{>2/\theta}$ such that $\muj(E^{{k_\theta},\theta})\le\theta$, and assume that $\delta$ is such that $2\delta\sqrt n\le 1/{k_\theta}$. 

For {some} $(s, {x,} R, \varphi):{A_\delta}\to \R^m\times {\Omega\times} \SO(n)\times {C^1_c}(\R^{n-1})$ (still to be defined) and any $z\in {A_\delta}$ we intend to estimate an error measure defined in analogy to \eqref{eqdefhatnux}
by $S_z:={x_z+}R_{z}\{(y',\varphi_{z}(y')): y'\in\R^{n-1}\}$  and
\begin{equation}\label{eqdefmz}
 m_z:=g_0(|[u]|)\calH^{n-1}\LL(J_u\setminus  S_z)
 + \big[ g_0(|[u]- s_z|)
{ + g_0(|[u]|) |\nu_u-R_ze_n|
\big]} \calH^{n-1}\LL S_z,
\end{equation}
namely to prove
\[
\sum_{z\in {A_\delta}}m_z(\Qlargez)\leq C\theta{(1+\mu_u(\Omega))}\,,
\]
with a constant $C>0$ depending on $n$, $p$.

Let $F:=\{z\in {A_\delta}: \Qlargez\subseteq E^{k_\theta,\theta}\}$.
If $z\in F$, we set $s_z:=0$, {$x_z:=z+2\delta e_n$}, $R_z{:=}\Id$,
{$\varphi_z:=0$,} so that
{$S_z\cap \Qlargez=\emptyset$ and}
{$m_z\LL \Qlargez=\muj\LL \Qlargez$.} Therefore
\begin{equation}\label{e:stima muj F}
 \sum_{z\in F} m_z(\Qlargez)=
 \sum_{z\in F} \muj(\Qlargez){\le 2^n}
  \muj(\bigcup_{z\in F}\Qlargez)\le {2^n}\muj(E^{k_\theta,\theta})\le{2^n}\theta.
\end{equation}
Consider now
 $z\in {A_\delta}\setminus F$, and select $x_z\in \Qlargez\setminus E^{{k_\theta},\theta}$.
 We set
$s_z:=\hat s_{x_z}$, $R_z:=\hat R_{x_z}$, 
and 
\begin{equation}
S_z{:=\hat L_{x_z}}=x_z+R_z\{(y',\hat\varphi_{x_z}(y')): y'\in\R^{n-1}\}\,.
\end{equation}
Note that with this choice $m_z=\hat m_{x_z}$. Moreover,
as $x_z\in \Qlargez$ we have
{$\Qlargez\subset B_{2\sqrt n\delta}(x_z)\cap\Omega$}.
As $2\delta\sqrt n\le 1/{k_\theta}$, $x_z\not\in E^{k_\theta,\theta}$ implies 
$ \hat m_{{x_z}}({B_{2\delta\sqrt n}(x_z)\cap\Omega})<\theta \muj({B_{2\delta\sqrt n}(x_z)\cap\Omega})$, so that 
\begin{equation}
 \sum_{z\in {A_\delta}\setminus F} m_z(\Qlargez)\le \sum_{z\in {A_\delta}\setminus F} \theta\muj({B_{2\delta\sqrt n}(x_z)\cap\Omega} ).
\end{equation}
Each ball $B_{2\sqrt n\delta}(x_z)$ overlaps 
with a finite number {$C(n)$} of cubes with center in $\delta\mathbb{Z}^n$
and side {of length} $\delta$, which implies that they have finite overlap. Therefore
\begin{equation}\label{e:stima muj Adelta meno F}
 \sum_{z\in {A_\delta}\setminus F} m_z(\Qlargez)\le
 {C}\theta\muj(\Omega)
\,.
\end{equation}
Recall that $\hat \varphi_{x_z}$ satisfies $\|D\hat\varphi_{x_z}\|_{L^\infty(B'_{3/k_\theta})}\leq{\frac13}\theta$
{and $\hat\varphi_{x_z}(0)=0$. We fix $\alpha_z\in C^1_c(B'_{3/k_\theta};[0,1])$ such that
$\alpha_z=1$ on $B'_{1/k_\theta}$ and $\|D\alpha_z\|_{L^\infty}\le \frac23k_\theta$, and set $\varphi_z:=\alpha_z\hat\varphi_{x_z}$.
Then  $\varphi_z\in C^1_c(\R^{n-1})$ with $\|D\varphi_z\|_{L^\infty(\R^{n-1})}\le \|D\hat\varphi_{x_z}\|_{L^\infty(B'_{3/k_\theta})}
(\|\alpha_z\|_{L^\infty} +  \frac3{k_\theta}\|D\alpha_z\|_{L^\infty} )\le\theta$, and $\varphi_{z}=\hat\varphi_{x_z}$ on $B'_{1/k_\theta}(x_z)$.}

Combining this remark with the results in \eqref{e:stima gradiente Lp},
\eqref{e:stima muj F}, \eqref{e:stima muj Adelta meno F} gives the first assertion.

Assume now that additionally $\calH^{n-1}(J_u)<\infty$. We proceed in the same way, replacing the measure $\hat m_x$ defined in \eqref{eqdefhatnux}
by $\hat M_x:=\hat m_x+ \calH^{n-1}\LL (J_u {\triangle} \hat L_x)$ and $\mu_u$ by $\hat\mu_u:=(g_0(|[u]|)+1)\calH^{n-1}\LL J_u$. By \eqref{lemmaBVfin} in Lemma~\ref{lemmaBV} we obtain that \eqref{eqnuxmuj} holds with $\hat M_x$ in place of $\hat m_x$, so that we can define $E^{k,\theta}$ with $\hat M_x$ and $\hat\mu_u$. Similarly, we consider in place of
$m_z$ defined in \eqref{eqdefmz} the measure
$M_z:=m_z+\calH^{n-1}\LL (J_u {\triangle}  S_z)$. The rest of the proof is unchanged, replacing $m_z$ by $M_z$ and $\mu_u$ by $\hat\mu_u$.
\end{proof}

\section{Proof of the approximation theorem}\label{s:proof}

\subsection{Explicit construction on a single simplex}\label{s:onesimplex}
We show how to construct the piecewise affine approximation in a single simplex, assuming that the values at the vertices and the jumps on the sides are {given}.  On each edge we shall use a function of the form illustrated {on the right-hand side of} Figure~\ref{fig:trianglev}.
{For simplicity we deal here only with scalar functions, the construction will then be applied componentwise.}
 
We consider points  $A_1,\dots, A_{n+1}\in\R^n$ such that {their convex envelope,} the simplex 
$T:=\conv(\{A_1,\dots, A_{n+1}\})$, has positive measure. 
 {The basic construction is outlined in general for} values $u_1, \dots, u_{n+1}{\in \R}$ of the function on the vertices, and
 jumps 
  $s_{ij}{\in\R}$ on the (oriented) edges, with $s_{ij}=-s_{ji}$ (which obviously implies $s_{ii}=0$). We then define the average gradients on the edges
  $\xi_{ij}:=u_j-u_i -s_{ij}$. 
  The definition of $\xi$ implies that whenever $\{i,j,k\}\subseteq\{1,\dots, {n}\}$ then
 \begin{equation}\label{eqxisrn}
  \xi_{ij}+\xi_{jk}+\xi_{ki}+
  s_{ij}+s_{jk}+s_{ki}=0.
 \end{equation}
The compatibility conditions arising from longer paths are not independent, as each path can be written as a concatenation of triangles.
On the edge joining $A_i$ with $A_j$, we require our target function to take the form
{(see Figure~\ref{fig:trianglev})}
  \begin{equation}\begin{split}
 v(A_i+t(A_{j}-A_i))=&
 u_i+t \xi_{ij} + s_{ij}\chi_{t>1/2}\\
 =&u_i+t(u_j-u_i) + s_{ij}(\chi_{t>1/2}-t).\end{split}
\end{equation} 
\begin{figure}
 \begin{center}
\includegraphics[width=4.8cm]{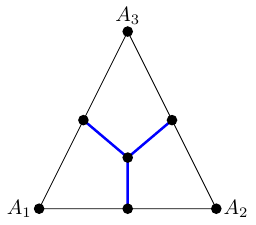} 
\hskip3mm
\includegraphics[width=4.8cm]{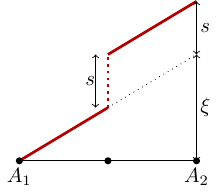} 
 \end{center}
\caption{Sketch of the construction in Proposition~\ref{propsimplex} {in the $2$d case}. 
Left: decomposition of the triangle. The blue lines represent the jump set of $v$.
Right: profile along a single edge. The parameter $s$ denotes the jump in the middle, the
parameter $\xi$ the rest of the height change, which corresponds to the uniform slope
in the rest.}
\label{fig:trianglev} 
\end{figure}

\newcommand\Rsk{\R_*}
\begin{proposition}\label{propsimplex}
Let $A_1,\dots, A_{n+1}\in\R^n$ {be} such that $T:=\conv(\{A_1,\dots, A_{n+1}\})$ 
has positive measure.
There is a decomposition of $T$ into $n+1$ {closed} polyhedra $T_1,\dots, T_{n+1}$ with disjoint interior such that the following holds.
Let $u_1, \dots, u_{n+1}\in \R$,
fix $s\in \Rsk:=\R^{(n+1)^2}_\skw$, 
and define
$\xi\in \Rsk$ by $\xi_{ij}+s_{ij}=u_j-u_i$.
Then there is $v:T \to\R$ affine {in each $T_j\setminus\bigcup_{i\ne j}\partial T_i$} and such that
\begin{equation}\label{e:stimeseparate}
 |\nabla v|\le {C\frac{\diam(T)^{n-1}}{|T|}} |\xi|\,, \hskip1cm |[v]|\le {3}|s|,
\end{equation}
\begin{equation}\label{e:lengthjump}
	\calH^{n-1}(J_v\cap T)\le \calH^{n-1}(\partial T),
\end{equation}	
and with
\begin{equation}\label{e:vedge}
 v(A_i+t(A_{j}-A_i))=
  u_i+t \xi_{ij} + s_{ij}\chi_{t>1/2}
\end{equation} 
for all ${i<j}\in\{1,\dots, n+1\}$, $t\in[0,1]$.
The constant {$C$} {depends only on $n$}. {The function $v$ depends linearly on $\{u_i\}\cup\{s_{ij}\}$.}

The function $v$ on a face {of $T$} does not depend on the opposing vertex. Precisely, for any $k$, if $x\in \conv(\{A_1, \dots, A_{n+1}\}\setminus\{A_k\})$  then
$v(x)$ depends only on $A_i$, $u_i$ for $i\ne k$ and on $s_{ij}$ for $i,j\ne k$.

\end{proposition}
\begin{proof}
We first observe that each point $x\in T$ 
can be uniquely represented as $x=\sum_{i=1}^{n+1} \lambda_i A_i$
for some $\lambda\in\Lambda:=\{\lambda\in[0,1]^{n+1}: \sum_{i=1}^{n+1}\lambda_i=1\}$.
We define {the polyhedra $T_j$ by}
\begin{equation}\label{eqdefTj}
 T_j:=\{\sum_i\lambda_iA_i:
 {\lambda\in\Lambda},
 \lambda_j\ge \lambda_i\text{ for }i\ne j\},
\end{equation}
{see Figure~\ref{fig:trianglev}(left)}
(in the case that $T$ is regular, this amounts to the Voronoi decomposition of $T$).
We remark that the condition $\lambda_i\le \lambda_j=1-\sum_{k\ne j}\lambda_k$ for all {$i\ne j$} is equivalent to
\begin{equation}\label{e:lambdai}
2\lambda_i\leq1-\sum_{k\neq i,j}\lambda_k\quad\mbox{ for all }i\neq j\,,
\end{equation}
so that
\[
x\in T_j\Longleftrightarrow x=A_j+\sum_{i\neq j}\lambda_i(A_i-A_j)\quad \lambda\in\Lambda \mbox{ as in \eqref{e:lambdai}}\,.
\]
{We define} $v_j:T_j\to\R$ by
\[
v_j(x):=\hat v_j(\mathbb A_j^{-1}{(x-A_j)})\,,
\]
where $\mathbb A_j$ is the matrix with {columns} given by $A_i-A_j$ for $i\neq j$, and
{$\hat v_j:\R^n\to\R$ is defined by}
\begin{equation}\label{eqvor}
\hat v_j(\lambda_1,\ldots,\lambda_{j-1},\lambda_{j+1},
\ldots,
\lambda_{n+1}) :=u_j+\sum_{i\ne j} \lambda_i \xi_{ji}
\end{equation}
{so that $v_j(\sum_i\lambda_i A_i)=u_j+\sum_{i\ne j} \lambda_i\xi_{ji}$.}
We define $v$ by setting
\begin{equation}
 v:=v_j\text{ in } T_j {\setminus \bigcup_{i<j} T_i.}
\end{equation}
Obviously $A_j\in T_j$ and $v(A_j)=u_j$. {Further,
{for any $j$ the function}
$v$ is affine {in $T_j\setminus \bigcup_{i<j}\partial T_i$}, with
\[
\nabla v(x)=(\mathbb A_j^{-1})^T\nabla \hat v_j(\mathbb A_j^{-1}{(x-A_j)})=(\mathbb A_j^{-1})^T(\xi_{ji})_{i\neq j}
\]
{for all $x$ inside} $T_j$, and therefore
\[
|\nabla v|\leq\|\mathbb A_j^{-1}\|_{op}{|(\xi_{ji})_{i\neq j}|}\quad\mbox{ on } T_j
\]
from which we infer that 
\[
|\det\mathbb A_j||\nabla v|\leq\|{\mathrm{cof\,}}\mathbb A_j\|_{op}{|\xi|}\quad\mbox{ on } T_j\,.
\]
{By definition of $\mathbb A_j$, it holds $|\det\mathbb A_j|=n!|T|$. 
As {a} cofactor is a homogeneous polynomial of degree $n-1$, one obtains} $\|{\mathrm{cof\,}}\mathbb A_j\|_{op}\leq C\diam(T)^{n-1}$, for some dimensional constant $C<\infty$. {This proves the first bound in \eqref{e:stimeseparate}.}

To conclude that $v\in SBV(T)$, with the claimed estimates, we note that since by construction 
$v$ is affine on each $T_j$, it jumps only on the points}
$x=\sum_i \lambda_iA_i\in \partial T_j\cap \partial T_k$ for some $j\ne k$. Necessarily 
$\lambda_j=\lambda_k$, and the conditions
$\xi_{ij}+s_{ij}=u_j-u_i$, $\sum_i\lambda_i=1$ with $\lambda_i\in[0,1]$, the antisymmetry of $\xi$, and the compatibility condition in \eqref{eqxisrn}  imply that
\begin{equation}\begin{split}
 (v_j- v_k)(\sum_i \lambda_iA_i)
 &= u_j-u_k + \lambda_k\xi_{jk}-\lambda_j\xi_{kj} + 
 \sum_{i\not\in\{j,k\}} 
 \lambda_i(\xi_{ji}-\xi_{ki})\\
 &= \xi_{kj}+s_{kj} - (\lambda_k+\lambda_j)\xi_{kj}+ 
 \sum_{i\not\in\{j,k\}} 
 \lambda_i(\xi_{ji}+\xi_{ik})\\
 &= s_{kj} + 
 \sum_{i\not\in\{j,k\}} 
 \lambda_i(\xi_{kj}+\xi_{ji}+\xi_{ik})\\
 &= s_{kj} -
 \sum_{i\not\in\{j,k\}} 
 \lambda_i(s_{kj}+s_{ji}+s_{ik})\\
 &{= (\lambda_j+\lambda_k) s_{kj} -
 \sum_{i\not\in\{j,k\}} 
 \lambda_i(s_{ji}+s_{ik}).}
 \end{split}
\end{equation}
{Therefore}
 $v\in SBV(T)$ with $|[{v}]|\le {3}|s|$ 
and 
\begin{equation}\label{eqestcalhn1jv}
\calH^{n-1}(J_v)\leq\sum_{i\neq j}\calH^{n-1}(\partial T_i\cap\partial T_j)
\leq\calH^{n-1}(\partial T)\,.
\end{equation}
The last inequality is proven in
\eqref{eqsumTiTjDT} below.
{This concludes the proof of \eqref{e:stimeseparate} and \eqref{e:lengthjump}. Condition \eqref{e:vedge}
follows directly from the definition above.}

 By construction, it is clear that $v$ does not depend on the vertex $A_k$ on the opposing face $F_k$, since {on $F_k$ we have} $\lambda_k=0$ and $F_k\cap T_k=\emptyset$.

It remains to prove the geometric inequality that was used in the last step of \eqref{eqestcalhn1jv}.
By Fubini's theorem one easily checks the following: 
Consider a set $\alpha$ of $k+1$ points in $\R^n$, 
$0\le k<n$. Then for any $x\in\R^n$ one has
\begin{equation}\label{eqk1volume}
 \calH^{k+1}(\conv(\alpha\cup\{x\})) =\frac1{k+1} \calH^k(\conv(\alpha))  \cdot\dist(x,\ASpan(\alpha)),
\end{equation}
where $\ASpan(\alpha)$ is the smallest affine space that contains $\alpha$ (if $k=0$ then $\conv(\alpha)=\ASpan(\alpha)=\alpha$ and
$\calH^0(\conv(\alpha))=1$).

Fix now $i\ne j\in\{1,\dots, n+1\}$, 
and consider $\partial T_i\cap \partial T_j$.
Then
\begin{equation}\begin{split}
 \partial T_i\cap \partial T_j
 =&\{\sum_{p=1}^{n+1}\lambda_p A_p: \lambda\in\Lambda, \lambda_i=\lambda_j=\max_p \lambda_p\}\\
 =&\{2 \lambda_i \frac{A_i+A_j}{2}+\sum_{p\ne \{i,j\}}\lambda_p A_p: \lambda\in\Lambda, \lambda_i=\lambda_j=\max_p \lambda_p\}\\
 \subseteq&\{\sum_{p=1}^{n}\lambda^*_p A_p^*: \lambda^*\in\Lambda^*
, \lambda^*_1=\max_p \lambda^*_p 
 \},
 \end{split}
\end{equation}
where
$\Lambda^*:=\{\lambda^*\in[0,1]^{n}: \sum_{p=1}^{n}\lambda^*_p=1\}$,
$A_1^*:=\frac{A_i+A_j}{2}$
and  $\{A_p^*\}_{p=2,\dots, n}$ is any relabeling of the $n-1$ points in
$\alpha_{ij}:=\{A_1,\dots, A_{n+1}\}\setminus\{A_i,A_j\}$
{(the inclusion in the last step follows from the fact that $\lambda_1^*=2\lambda_i\ge{2} \lambda_p^*$ for all $p>1$).}
By symmetry, all $n$ sets $\{\lambda^*\in\Lambda^*: \lambda^*_i=\max_p \lambda^*_p\}$ 
have the same area, and as they are disjoint up to $\calH^{n-1}$-dimensional  null sets we obtain
\begin{equation}
\calH^{n-1}(\partial T_i\cap \partial T_j)\le \frac{1}{n} \calH^{n-1}\big(\conv\{A_p^*\}\big)
 =\frac{1}{n} \calH^{n-1}
\Big(  \conv\big(\alpha_{ij}\cup\{\frac{A_i+A_j}2\}\big)\Big)
 \end{equation}
so that
\eqref{eqk1volume} gives
\begin{equation}\label{eqhn1titja}
 \calH^{n-1}(\partial T_i\cap \partial T_j) 
\le  \frac1{n(n-1)} \calH^{n-2}(\conv(\alpha_{ij})) \cdot\dist\Big(\frac{A_i+A_j}2,\ASpan(\alpha_{ij})\Big).
\end{equation}
By convexity
\begin{equation}\label{eqhn1titjb}
\dist\Big(\frac{A_i+A_j}2,\ASpan(\alpha_{ij})\Big)
\le \frac12 \dist(A_i,\ASpan(\alpha_{ij}))
+\frac12\dist(A_j,\ASpan(\alpha_{ij}))\,.
\end{equation}
Let $F_i:=\conv(\{A_1,\dots, A_{n+1}\}\setminus\{A_i\})=\conv(\alpha_{ij}\cup\{A_j\})$ be the face opposite to {the} vertex $A_i$. By
\eqref{eqk1volume},
\begin{equation}\label{eqhn1tiF}
 \calH^{n-1}(F_i)=\frac1{n-1}\calH^{n-2}(\conv(\alpha_{ij}))\cdot\dist(A_j,\ASpan(\alpha_{ij})).
\end{equation}
Combining \eqref{eqhn1titja}, \eqref{eqhn1titjb} and \eqref{eqhn1tiF} gives
\begin{equation}
 \calH^{n-1}(\partial T_i\cap \partial T_j) 
 \le\frac1{2n} \calH^{n-1}(F_i)+\frac1{2n} \calH^{n-1}(F_j).
\end{equation}
We sum over all {pairs $(i,j)$ with} $i\ne j$ and obtain
\begin{equation}\label{eqsumTiTjDT}
\sum_{{i\neq j}} \calH^{n-1}(\partial T_i\cap \partial T_j) 
 \le\sum_{i=1}^{n+1} \sum_{j\ne i} \frac1n \calH^{n-1}(F_i)
=\sum_{i=1}^{n+1}\calH^{n-1}(F_i)
 =\calH^{n-1}(\partial T)
\end{equation}
which concludes the proof.
\end{proof}

\subsection{Projection on piecewise affine functions}\label{s:projection}
In this section, we use Proposition~\ref{propsimplex} to construct a good piecewise affine interpolation of any vectorial function $u\in SBV_\loc(\R^n;\R^m)$ over a suitable partition of $\R^n$ in simplexes. First, Lemma~\ref{lemmatriangRn} states the general properties of the chosen partition. Proposition~\ref{propsimplex} then can be applied {componentwise} in each simplex of a suitable shift of the partition. The resulting interpolation can be interpreted as a projection of $u$ over piecewise affine functions and enjoys good energy estimates, see Proposition~\ref{propproj}.

\begin{lemma}\label{lemmatriangRn}
 For any $n\ge 1$ there is a countable set of simplexes $\mathcal T_0\subseteq \calP(\R^n)$ such that, denoting by $\Ver(\tau_0)$ the set of vertices of $\tau_0\in \mathcal T_0$, one has: 
 \begin{enumerate}
  \item\label{enumgridvere} $\#\Ver(\tau_0)=n+1$; $|\tau_0|>0$ for all $\tau_0\in \mathcal T_0$;
  \item\label{enumgridintersect} 
 $\tau_0\cap \tau'_0=\conv(\Ver(\tau_0)\cap \Ver(\tau_0'))$, in particular 
$|\tau_0\cap \tau'_0|=0$ if $\tau_0\ne\tau'_0$;
 \item\label{enumgridcover} $\bigcup_{\tau_0\in \mathcal T_0} \tau_0=\R^n$;
 \item\label{enumgridZn} For any $\tau_0\in \mathcal T_0$ there is $z\in \Z^n$ such that 
$\Ver(\tau_0)\subseteq  \{z+\sum_i\lambda_ie_i:\, \lambda_i\in\{0,1\},\,i\in\{1,\ldots,n\}\}$,
with the $e_i$'s the canonical basis vectors;
 \item\label{lemmatriangRnper} If $\tau_0\in \mathcal T_0$, then $\tau_0+2e_i\in \mathcal T_0$, with $e_i$ {any} of the canonical basis vectors.
 \end{enumerate}
\end{lemma}

We recall that $\conv(\emptyset)=\emptyset$.
{Condition~\ref{enumgridvere} and condition~\ref{enumgridintersect} with $\tau_0=\tau'_0$ imply that $\tau_0$ is a closed simplex.}
Condition~\ref{enumgridZn} implies that for all $\eps>0$, {the rescaled simplex} $\eps{\tau_0}$ has diameter at most  $\eps\sqrt n$, and together with condition~\ref{enumgridvere} that its volume is at
least $\eps^n/n!$ {(indeed, it is $1/n!$ times the determinant of a matrix with entries in 
$\{-\eps,0,\eps\}$).}
 The last two imply that this is a refinement of the natural subdivision of $\R^n$ into {unitary} cubes,  with period $[0,2]^n$.
\begin{proof}
This can be obtained taking any partition, as for example the Freudenthal partition, of $[0,1]^n$, reflecting this along the coordinate axes to obtain a partition of $[-1,1]^n$, and then extending periodically.
\end{proof}

{In the rest of this section we define 
for any $\eps>0$ and $\zeta\in B_\eps$}
a projection 
$\Pi_{\eps,\zeta}$ on the space of functions that are affine on each 
polyhedron {in a refinement of $\zeta+\eps\calT_0$. The projection will be used on functions $u\in SBV(\R^n;\R^m)$. As it depends on point values, we shall only obtain a well-defined result for values of the translation $\zeta$ outside a null set. The null set, {however,} depends on $u$. To avoid this, the precise definition is given not on equivalence classes but on functions from $\R^n$ to $\R^m$.
In order to understand the key properties, it is useful to consider first the action of $\Pi_{\eps,\zeta}$ on elements of $SBV(\R^n;\R^m)$ (or, equivalently,
$ SBV_\loc(\R^n;\R^m)$, since $\Pi_{\eps,\zeta}$  is local). 
}

Let $u\in SBV(\R^n;\R^m)$.  For any couple of vertices $a\ne b\in\Ver(\tau)$
{of a simplex $\tau\in\zeta+\eps \calT_0$}, we
consider the slice $u_a^{b-a}(t):=u(a+t(b-a))$.
For {$\calL^{2n}$}-almost every pair $(a,b)$ we have
$u_a^{b-a}\in SBV({\R};\R^m)$ with
\begin{equation}\label{eqslicebv}
 u(b)-u(a)=\int_0^1 (u_a^{b-a})'(t)dt +
 \sum_{t\in (0,1)\cap J_{u_a^{b-a}}}
 [u_a^{b-a}](t)
\end{equation}
with
\begin{equation*}
(u_a^{b-a})'(t)=\nabla u(a+t(b-a))(b-a)
\end{equation*}
and 
\begin{equation*}
[u_a^{b-a}](t)=[u](a+t(b-a))
\sgn(\nu_{a+t(b-a)}\cdot (b-a))
\end{equation*}
(see \cite[Sect.~3.11]{AFP} 
or \cite[Th.~4.1(a)]{Braides}).
{The parameters $s$ and $\xi$ entering the piecewise affine construction in Proposition~\ref{propsimplex} are then defined from the two components of \eqref{eqslicebv}. Precisely,}
we define the jump over an edge $[a,b]$ 
by 
\begin{equation}\label{e:jumpedge}
 s_{[a,b]}:=\sum_{t\in(0,1){\cap J_{u_a^{b-a}}}} [u](a+t(b-a))
 {\sgn(\nu_{a+t(b-a)}\cdot (b-a))}
\end{equation}
({setting} it to zero if the sum does not converge {or is not defined})
and correspondingly the integral of the absolutely continuous part of the gradient by
\begin{equation}\label{e:gradientedge}
\xi_{[a,b]}:=u(b)-u(a)-s_{[a,b]}.
\end{equation}
For {$\calL^{2n}$}-almost all pairs $(a,b)$
\begin{equation}\label{eqxiabintb}
 \xi_{[a,b]} = \int_0^1 \nabla u(a+t(b-a)) (b-a) dt.
\end{equation}
By monotonicity and subadditivity of $g_0$,
\begin{equation}\label{eqgsabsubadd}
 g_0({|s_{[a,b]}|}) \le\sum_{t\in(0,1){\cap J_{u_a^{b-a}}}} {g_0(|[u](a+t(b-a))|)}\,.
\end{equation}
Similarly, 
for {$\calL^{2n}$}-almost all {pairs $(a,b)$},
using  \eqref{eqxiabintb} and Jensen's inequality, \begin{equation}\label{eqxisegmentnablau}
  {\left|\frac{\xi_{[a,b]}}{|b-a|}\right|^p \le
 \int_0^1 |\nabla u(a+t(b-a))|^pdt.}
\end{equation}
In Proposition~\ref{propproj} we will turn both estimates \eqref{eqgsabsubadd} and \eqref{eqxisegmentnablau} into estimates relating the energy over shifts of
the segment, averaged over all possible shifts of size less than $\eps$,
and integrals over $J_u$ and $\Omega$, respectively.

\begin{figure}
 \begin{center}
  \includegraphics[width=6cm]{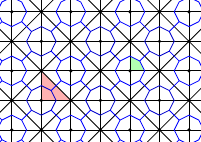}
 \end{center}
 \caption{Sketch of the construction for Proposition~\ref{propproj}. The dots mark the points on which $\Pi_{\eps,\zeta}{f}$ coincides with ${f}$, the black triangles (one of them is colored red) are {the elements of $\calT_{\eps,\zeta}$,} on which
 Proposition~\ref{propsimplex} is applied.
The blue {segments} are the (eventual) discontinuities introduced by the construction of Proposition~\ref{propsimplex} {and delimit the polygons which compose $\calT^*_{\eps,\zeta}$.}
The function $\Pi_{\eps,\zeta}u$
 is affine on the smaller polyhedra (one of them is colored green).
 \label{figpropproj}}
\end{figure}

\newcommand{\cpropjbounds}{{c_*}}
\begin{proposition}\label{propproj}
There is a locally finite subdivision of $\R^n$
into countably many essentially disjoint polyhedra, $\calT^*_0$, finer than $\mathcal T_0$ and with the same periodicity, and $C>0$ such that,
for any $\eps>0$ {and} $\zeta\in B_\eps$, to any function $f:\R^n\to\R^m$ one can associate a
function {$\Pi_{\eps,\zeta}f:\R^n\to\R^m$}, affine in the interior of each element of $\calT_{\eps,\zeta}^*:=\zeta+\eps\mathcal T_0^*$, so that the following holds:
\begin{enumerate}
\item\label{propprojgrids} 
If either ${\tau_0}\in \calT_0$ or ${\tau_0}\in \calT^*_0$ then $\diam({\tau_0})\le \sqrt n$ and $|{\tau_0}|\ge 1/C$;
\item\label{propprojprojection} $\Pi_{\eps,\zeta}$ is a projection, in the sense that $\Pi_{\eps,\zeta}\Pi_{\eps,\zeta}f=\Pi_{\eps,\zeta}f$ for all $f$, and it commutes with translations, in the sense that
$\Pi_{\eps,\zeta}[f(\cdot-\zeta)]=[\Pi_{\eps,0}f](\cdot-\zeta)$;
\item\label{propprojjump} {One has}
{$\Pi_{\eps,\zeta}{f}
\in  SBV_\loc(\R^n;\R^m)$ and,} with $\calT_{\eps,\zeta}:=\zeta+\eps\calT_0$,
\begin{equation}
|D\Pi_{\eps,\zeta}{f}|(\bigcup_{\tau\in \mathcal T_{\eps,\zeta}} \partial \tau )=0.
\end{equation}
If $u\in  SBV_\loc(\R^n;\R^m)$, then for {$\calL^n$}-almost every {$\zeta\in B_\eps$} one has
\[\calH^{n-1}(J_u\cap \bigcup_{\tau\in \mathcal{T}_{\eps,\zeta}} \partial \tau )=0.
\]
\item\label{propprojlocal} The function $\Pi_{\eps,\zeta} {f}$ on a set $\omega$ depends only on the value of ${f}$ on the set
$(\omega)_{\eps\sqrt n}$.
\item\label{propprojaffine} If ${A}:\R^n\to\R^{m}$ is affine and $\lambda\in\R$, 
then for any function ${f}$ one has $\Pi_{\eps,\zeta} (\lambda {f}+{A})= \lambda(\Pi_{\eps,\zeta}{f})+ {A} $;
if $u,v\in SBV_\loc(\R^n;\R^m)$ then for almost every {$\zeta\in B_\eps$} one has
$\Pi_{\eps,\zeta} (u+v)=\Pi_{\eps,\zeta}u+\Pi_{\eps,\zeta}v$;
\item\label{propprojbounds} 
For any $\eta\in\R^{m\times n}$ and $\tau_0\in \mathcal T_0$,
one has 
for any
$u\in SBV_\loc(\R^n;\R^m)$ that
\begin{equation}\label{e:stima bulk}
	\strokedint_{B_\eps}\Big(\int_{\zeta +\eps {\tau_0}}
|\nabla \Pi_{\eps,\zeta} u-\eta|^p dx\Big)d\zeta
\leq C\int_{{{(\eps \tau_0)_{\cpropjbounds\eps}}}}
|\nabla u-\eta|^p dx,
\end{equation}
\begin{equation}\label{e:stima superficie}
 \strokedint_{B_\eps}{\Big(} \int_{J_{\Pi_{\eps,\zeta}u}\cap(\zeta+\eps \tau_0)}g_0(|[\Pi_{\eps,\zeta}u]|)d\calH^{n-1}{\Big)}
 d\zeta\leq C\int_{J_u\cap {(\eps \tau_0)_{\cpropjbounds\eps}}
 }g_0(|[u]|)d\calH^{n-1},
 \end{equation}
\begin{equation}\label{e:stima superficiearea}
 \strokedint_{B_\eps} \calH^{n-1}(J_{\Pi_{\eps,\zeta}u}\cap(\zeta+\eps \tau_0))
 d\zeta\leq C\calH^{n-1} ( J_u\cap {(\eps \tau_0)_{\cpropjbounds\eps}}
 ),
 \end{equation}
\begin{align}\label{eqnorm1}
\strokedint_{B_\eps}&\Big(\int_{\zeta+\eps\tau_0}|\Pi_{\eps,\zeta} u-u|dx\Big)d\zeta
\le C \eps |Du|({(\eps \tau_0)_{\cpropjbounds\eps}}),
\end{align}
and, for every $n-1$-rectifiable set $\Sigma$,
\begin{equation}\label{eqpsurface}
 \strokedint_{B_\eps}{\Big(} \int_\Sigma | \Pi_{\eps,\zeta} u| d\calH^{n-1}{\Big{)}} d\zeta
 \le \frac{Ck_\Sigma}{\eps} \| u\|_{L^1((\Sigma )_{\cpropjbounds\eps})}
 +{Ck_\Sigma}| Du|((\Sigma )_{2\cpropjbounds\eps}),
\end{equation}
where
\begin{equation}
 k_\Sigma:=\sup_{r>0,x\in \R^n} \frac{\calH^{n-1}(\Sigma\cap B_r(x))}
 {r^{n-1}}.
\end{equation}
{Here $\cpropjbounds\in [1+\sqrt n,\infty)$ is a constant that depends only on $n$; $C$ may depend on $n$, $m$, $p$.}
\item\label{propprojregular} If
$u\in SBV_\loc(\R^n;\R^m)$ and
$\nabla u=0$ $\calL^n$-almost everywhere
then for almost every {$\zeta\in B_\eps$} one has $\nabla\Pi_{\eps,\zeta}u=0$ $\calL^n$-almost everywhere.
In particular, if $u=\chi_E$ for some set $E$ then there is a countable union of polygons $F_{\eps,\zeta}$ such that $\Pi_{\eps,\zeta}u=\chi_{F_{\eps,\zeta}}$.
If $\calH^{n-1}(J_u)=0$ then for almost every {$\zeta\in B_\eps$} one has
$\calH^{n-1}(J_{\Pi_{\eps,\zeta}u})=0$. 
\end{enumerate}
\end{proposition}
{Condition \eqref{e:stima bulk} easily implies that for any Borel set $\omega\subset\R^n$ and any $\eta$
\begin{equation}\label{e:stima bulkomega}
	\strokedint_{B_\eps}\Big(\int_{\omega}
|\nabla \Pi_{\eps,\zeta} u-\eta|^p dx\Big)d\zeta
\leq C\int_{{{(\omega)_{2\cpropjbounds\eps}}}}
|\nabla u-\eta|^p dx.
\end{equation}
Indeed, it suffices to sum \eqref{e:stima bulk} over all $\tau_0\in \calT_0$ such that
there is $\zeta\in B_\eps$ with $(\zeta+\eps\tau_0)\cap \omega\ne \emptyset$, which implies $\eps\tau_0\subseteq (\omega)_{(1+\sqrt n)\eps}$.
Analogous observations  hold for \eqref{e:stima superficie}, 
\eqref{e:stima superficiearea}
 and \eqref{eqnorm1}.}

We remark that {\eqref{eqpsurface} fails if we remove the derivative term
in the right-hand side.} 
		Consider for example the sequence $u_j(x):=\frac1j \langle jx_1\rangle$, {where $\langle x\rangle:=x-\lfloor x\rfloor$ denotes the fractional part of $x\in\R$,} which converges uniformly to $0$ {as $j\to\infty$}. As $\nabla u_je_1=1$ almost everywhere, for any $\eps$ and $\zeta$ we have
$\partial_1 \Pi_{\eps,\zeta}u_j=1$ almost everywhere, and since {$\Pi_{\eps,\zeta}u_j$} is piecewise affine on a scale $\eps$ we obtain
$\|\Pi_{\eps,\zeta}u_j\|_{L^\infty} \ge \frac12\eps$, which does not depend on $j$.
{Similarly, one
cannot estimate $\Pi_{\eps,\zeta}u$ in $L^1$ only in terms of the $L^1$ norm of $u$.}

\begin{proof}
The grid $\calT_0^*$ is defined decomposing each simplex $\tau_0\in\calT_0$ as in Proposition~\ref{propsimplex}.
The projection $\Pi_{\eps,\zeta}{f}$ is defined by application of
the construction in Proposition~\ref{propsimplex} {componentwise} in each simplex
$\tau\in\calT_{\eps,\zeta}=\zeta+\eps\calT_0$.

Precisely, let 
$\tau=\zeta+\eps\tau_0$, for some
$\tau_0\in\mathcal{T}_0$, and let $\{w_1,\dots, w_{n+1}\}:=\Ver(\tau)$ be its vertices. 
In order to define the cumulated jump over the edge $[w_i,w_j]$ we consider the slice
$ v^{{f}}_{ij}(t):={f}(w_i+t(w_j-w_i))$, for $t\in [0,1]$. If
$v^{{f}}_{ij}\in SBV((0,1);\R^m)$ then we set
\begin{equation}\label{eqdefsij}
 {s_{ij}^{{f}}}:=\sum_{t\in (0,1)\cap J_{v^{{f}}_{ij}}} [v^{{f}}_{ij}](t),
\end{equation}
otherwise we set ${s_{ij}^{{f}}}:=0$. The function $\Pi_{\eps,\zeta}{{f}}$ is then defined in $\tau$ using Proposition~\ref{propsimplex} {componentwise}.
As discussed above, if ${f}=u\in SBV_\loc(\R^n;\R^m)$ then for almost every choice of $\zeta$ one has that ${v_{ij}^u}\in SBV((0,1);\R^m)$ for all choices of {$\tau_0$} and of $i,j$.

\ref{propprojgrids}: The upper bound on the diameter follows from Lemma~\ref{lemmatriangRn} and the fact that $\calT_0^*$ is a refinement of $\calT_0$. The lower bound on the volume follows from the fact that both grids are locally finite and periodic.

\ref{propprojprojection}: {Assume for simplicity that $f$ is scalar.}
For any $\tau$ and $w_1,\dots, w_{n+1}$ as above, one easily obtains that $(\Pi_{\eps,\zeta}f)(w_i)=f(w_i)$. 
{
Let $s_{ij}^f$ be defined as in \eqref{eqdefsij}.
By
\eqref{e:vedge}
the function
$v_{ij}^{\Pi_{\eps,\zeta}f}$
has a unique jump point in $(0,1)$, which is located at $1/2$, and the amplitude of the jump is exactly $s_{ij}^f$.
Therefore $s_{ij}^{\Pi_{\eps,\zeta}f}=s_{ij}^f$ and} $\Pi_{\eps,\zeta}$ is a projection.
The relation to translations follows observing that
for any $\tau_0\in\calT_0$ the vertices of ${\zeta+\eps\tau_0}$ are translated
with respect to the vertices of ${\eps\tau_0}$ by $\zeta$.

\ref{propprojjump}: 
Let $\tau\ne\tau'\in \mathcal T_{\eps,\zeta}$ be such that ${\calH^{n-1}(\partial \tau\cap \partial \tau')>0}$, so that by Lemma~\ref{lemmatriangRn}\ref{enumgridintersect}  $\partial \tau\cap \partial \tau' =\conv(\Ver(\tau)\cap\Ver(\tau'))$. Proposition~\ref{propsimplex} implies that $\Pi_{\eps,\zeta} {f}|_{\partial \tau\cap \partial \tau'}$ only depends on the in-plane vertices
$\Ver(\tau)\cap\Ver(\tau')$, on the values of ${f}$ on such vertices, and on the jumps $s_{\sigma}$ on the in-plane edges $\sigma\subset\partial \tau\cap \partial \tau'$.
Hence $|D\Pi_{\eps,\zeta} {f}|(\bigcup_{\tau\in \mathcal T_{\eps,\zeta}} \partial \tau)=0$.
The second condition follows from the fact that $\calH^{n-1}\LL{J_u}$ is {$\sigma$-finite.}

\ref{propprojlocal}: Given a set $\omega\subset\R^n$, the function $\Pi_{\eps,\zeta} {f}|_{\omega}$
only depends on the values of ${f}$ on the vertices of the simplexes intersecting $\omega$. Since their diameter is by construction not greater than $\eps\sqrt n$, $\Pi_{\eps,\zeta} {f}|_{\omega}$ only depends on the value of ${f}$ on the neighborhood $(\omega)_{\eps\sqrt n}$.

\ref{propprojaffine}: 
It follows from the fact that the function constructed in Proposition~\ref{propsimplex} depends linearly on the prescribed values  on the vertices $u_i$ and jumps $s_{ij}$.

\ref{propprojbounds}:
By \ref{propprojaffine}, it suffices to prove the first bound in the case
$\eta=0$.
Let ${\tau_0}\in \mathcal T_0$, and $\{w_1,\dots, w_{n+1}\}=\Ver(\eps\tau_0)$.
For any $\zeta\in B_\eps$, by the uniform estimate in \eqref{e:stimeseparate}
we have
\begin{align}\label{e:stima grad proj}
{\int_{\zeta+\eps{\tau_0}}
{|\nabla \Pi_{\eps,\zeta} u|^p}dx
\leq C \eps^n\sum_{i,j}\Big|\frac{\xi_{\zeta+[w_i,w_j]}}{|w_i-w_j|}\Big|^p.}
\end{align}
Next, we claim that {for all $i,j$}
\begin{equation}\label{e:sectionsgradient}
\int_{B_\eps}
\left|\frac{\xi_{\zeta+[w_i,w_j]}}{{|w_i-w_j|}}
   \right|^pd\zeta\le
 \int_{{B_{(1+{\sqrt n})\eps}(w_i)}} |\nabla u|^p dx.
\end{equation}
Indeed, 
starting from \eqref{eqxisegmentnablau} and integrating over all translations
$\zeta\in B_\eps$ we {get, 
setting $\ell:=w_j-w_i$},
\begin{equation*}\begin{split}
{\int_{B_\eps}
\left|\frac{\xi_{\zeta+[w_i,w_j]}}{{|\ell|}}
   \right|^p}  d\zeta \leq&
 \int_{B_\eps}
  \int_0^1 |\nabla u(\zeta+{ w_i+t \ell})|^p dtd\zeta\\
=&
  \int_0^1
   \int_{B_\eps({ w_i}+t \ell)}
  |\nabla u(x)|^p dxdt
 \le
   \int_{B_{{(1+\sqrt n)}\eps}({w_i})}
  |\nabla u|^p dx
  \end{split}
\end{equation*}
{since $|\ell|=|w_i-w_j|\le\eps\sqrt n$.}
Therefore, from \eqref{e:stima grad proj} and \eqref{e:sectionsgradient} we conclude
\begin{align*}
	\strokedint_{B_\eps}&\Big(\int_{\zeta +\eps \tau_0}
{|\nabla \Pi_{\eps,\zeta} u|^p}dx\Big)d\zeta\\
    &  \leq C\sum_{i}
	\int_{B_{{({1+\sqrt n})}\eps}({ w_i})} |\nabla u|^p dx
\leq C\int_{{(\eps \tau_0)_{(1+\sqrt n)\eps}}} |\nabla u|^p dx,
\end{align*}
{which concludes the proof of \eqref{e:stima bulk}.}

Analogously, using \ref{propprojjump}, \eqref{e:stimeseparate}, and \eqref{e:lengthjump}, we get
\begin{align} \label{e:stima salto proj}
\int_{J_{\Pi_{\eps,\zeta} u}\cap (\zeta+ \eps {\tau_0})}
g_0({|[\Pi_{\eps,\zeta} u]|})d\calH^{n-1}
	\leq C\eps ^{n-1}
\sum_{i,j}g_0({|s_{\zeta+{[w_i,w_j]}}|}),
\end{align}
by monotonicity and subadditivity of $g_0$, where as before the $w_i$ are the vertices of
${\eps\tau_0}$. We claim that
\begin{equation}\label{e:sectionsjump}
\int_{B_\eps} {g_0({|s_{\zeta+{[w_i,w_j]}}|})}d\zeta \le C
\eps\int_{B_{{(2+{\sqrt n})}\eps}({ w_i})\cap J_u} {g_0(|[u]|)}d\calH^{{n-1}}\,.
\end{equation}
Indeed, we start from \eqref{eqgsabsubadd}, integrate over translations, and separate the component
$\zeta_\ell$ along $\ell$ from the orthogonal ones, which we denote by $\zeta'$, so that
$\zeta=\zeta'+{\zeta_\ell} \ell/|\ell|$.
Using the Coarea formula we estimate as follows
\begin{equation}\label{eqgsabsubaddb}
\begin{split}
{\int_{B_\eps} g_0({|s_{\zeta+{[w_i,w_j]}}|})d\zeta} &\le 
\int_{B_\eps} \sum_{t\in(0,1)} 
{g_0(|[u](\zeta+{ w_i}+t\ell)|)}d\zeta\\
&\le\int_{-\eps}^\eps \int_{B'_\eps}
 \sum_{t\in(0,1)} 
{g_0(|[u]({\zeta'}+{w_i}+t\ell+\zeta_\ell \frac{\ell}{|\ell|})|)}d\zeta'd{\zeta_\ell}\\
 &\le 2\eps \int_{B'_\eps}
 \sum_{t\in(-\eps, \eps (1+{\sqrt n}))} 
{g_0(|[u](\zeta'+{w_i}+t\frac\ell{|\ell|})|)}d\zeta'\\
 &\le 2\eps \int_{J_u\cap B_{{{(2+\sqrt n})}\eps}({w_i})}
{g_0(|[u]|)}\,\Big|{\nu_u}\cdot \frac{\ell}{|\ell|}\Big| d\calH^{n-1}{.}
 \end{split}
\end{equation}
Clearly, \eqref{e:sectionsjump} easily follows from \eqref{eqgsabsubaddb}.

Hence, by {\eqref{e:stima salto proj}} and \eqref{e:sectionsjump}
\begin{equation}\label{eq439}
	\strokedint_{B_\eps}\int_{J_{\Pi_{\eps,\zeta} u}\cap (\zeta +\eps {\tau_0})}
{g_0(|[\Pi_{\eps,\zeta} u]|)}d\calH^{n-1}d\zeta
	\leq C\int_{J_u\cap {(\eps\tau_0)_{{(2+\sqrt n)}\eps}}}
g_0(|[u]|)d\calH^{n-1}
\end{equation}
which concludes the proof of \eqref{e:stima superficie}.

The proof of \eqref{e:stima superficiearea} is similar.
Let $g_1:[0,\infty)\to[0,\infty)$ be defined by $g_1(0)=0$, $g_1(s)=1$ for $s\ne 0$.
The
{derivation of \eqref{e:stima salto proj}} above
uses only {\ref{propprojjump}, \eqref{e:stimeseparate}, \eqref{e:lengthjump}, and the fact} that $g_0$ is nondecreasing and subadditive, and
$g_1$ has the same properties. By \eqref{eqdefsij},
$s_{ij}^u=0$ if $v^u_{ij}$ does not jump on $[0,1]$,
{thus} we obtain instead of \eqref{e:stima salto proj} the estimate
\begin{align} \label{e:stima salto proj1}
\int_{J_{\Pi_{\eps,\zeta} u}\cap (\zeta+ \eps {\tau_0})}
g_1({|[\Pi_{\eps,\zeta} u]|})d\calH^{n-1}
	\leq C\eps ^{n-1}
\sum_{i,j}g_1({|s_{\zeta+{[w_i,w_j]}}|}).
\end{align}
The rest of the computation {leading to \eqref{eq439}} is unchanged. {This proves
\eqref{e:stima superficiearea}.}

Next, we estimate the $L^1$ distance of $u$ from $\Pi_{\eps,\zeta} u$.
By \eqref{e:stimeseparate} for ${\tau_0}\in\mathcal T_0$ one has
the pointwise estimate
\begin{align*}
\strokedint_{B_\eps}&\int_{\zeta+\eps{\tau_0}}|\Pi_{\eps,\zeta} u-u|dx\,d\zeta\\
&\leq\strokedint_{B_\eps}\int_{\zeta+\eps{\tau_0}}
\sum_{ i}|u(\zeta+{w_i})-u(x)|
dx\,d\zeta+
C\eps^{n}\strokedint_{B_\eps}\sum_{i,j} |\xi_{\zeta+{[w_i,w_j]}}|d\zeta.
\end{align*}
We observe that for all choices of $i$, $\zeta$ and $x$ we have $x\in \zeta+\eps\tau_0\subseteq {(\eps\tau_0)_\eps}$
and $\zeta+w_i\in  \zeta+\eps\tau_0\subseteq {(\eps\tau_0)_\eps}$.
Therefore, each addend in the first term can be estimated using Poincar\'e's inequality for $BV$ functions by
\begin{align*}
\strokedint_{B_\eps}\int_{\zeta+\eps{\tau_0}}
&|u(\zeta+{w_i})-u(x)|
dx\, d\zeta\le  \frac{1}{|B_\eps|}
\int_{(\eps\tau_0)_\eps} \int_{(\eps\tau_0)_\eps}  |u(y)-u(y')| dy dy'\\
\le&  {2}\frac{|{(\eps\tau_0)_\eps} |}{|B_\eps|}
\int_{(\eps\tau_0)_\eps} |u(y)-\bar u| dy
\le  C \eps |Du|{((\eps\tau_0)_\eps)},
\end{align*}
where $\bar u$ denotes the average of $u$ in $(\eps\tau_0)_\eps$.
For the second one, we write using
\eqref{e:sectionsgradient} {with $p=1$}
\begin{align*}
\eps^n\strokedint_{B_\eps}\sum_{i,j} |\xi_{\zeta+{[w_i,w_j]}}|d\zeta\le C\eps\int_{(\eps\tau_0)_{C\eps}}
|\nabla u| dx.
\end{align*}
Combining the two gives \eqref{eqnorm1}.

{Finally, we prove} \eqref{eqpsurface}: for any $\tau_0\in \calT_0$ {and $\zeta\in B_\eps$}, we have
$\calH^{n-1}(\Sigma\cap (\zeta+\eps\tau_0))\le C k_\Sigma \eps^{n-1}$.
As the map $\Pi_{\eps,\zeta}u$ is affine on each element
of $\zeta+\eps\calT_0^*$, we have
\begin{equation}
 \|\Pi_{\eps,\zeta}u\|_{L^\infty(\zeta+\eps\tau_0)}\le \frac{C}{\eps^n} \|\Pi_{\eps,\zeta}u\|_{L^1(\zeta+\eps\tau_0)}.
\end{equation}
We sum over all $\tau_0$ such that $\zeta+\eps\tau_0$ intersects $\Sigma$ and obtain
\begin{equation}
 \int_\Sigma | \Pi_{\eps,\zeta} u| d\calH^{n-1}
 \le  \frac{Ck_\Sigma}{\eps} \| \Pi_{\eps,\zeta} u\|_{L^1((\Sigma )_{\sqrt n \eps})}
\end{equation}
{for a.e.\ $\zeta\in B_\eps$.}
Then we use \eqref{eqnorm1} and a triangular inequality to conclude.

\ref{propprojregular}: If $\nabla u=0$ $\calL^n$-almost everywhere, then
{\eqref{e:stima bulk} with $\eta=0$ implies $\nabla \Pi_{\eps,\zeta}u=0$  $\calL^n$-almost everywhere for $\calL^n$-almost every $\zeta$.}
{If $u$ takes values in $\{0,1\}$, then for each application of Proposition \ref{propsimplex} we have that $\xi_{ij}=0$,
therefore the constructed function is piecewise constant and takes values in the set $\{u_1, \dots, u_{n+1}\}\subseteq\{0,1\}$.}

Finally, if $\calH^{n-1}(J_u)=0$ then necessarily $u\in {W^{1,1}_\loc}(\R^n;\R^m)$.
In turn, the slices of $u$ are Sobolev functions for $\calL^n$-almost every $\zeta\in B_\eps$,
so that $J_{\Pi_{\eps,\zeta}u}=\emptyset$, in turn implying that $\Pi_{\eps,\zeta}u$ is actually 
continuous and in ${W^{1,1}_\loc}(\R^n;\R^m)$
(alternatively, this follows also {from \eqref{e:stima superficiearea}).}
\end{proof}

\subsection{Global construction}\label{s:globalconstruction}
We are now ready to establish Theorem~\ref{theodensityintro}.
The proof contains two different scales, denoted by $\delta$ and $\eps$ in the following. The scale $\delta$ is the one at which the function $u$ has approximately regular jump and gradient, and is identified in Proposition~\ref{propdeltafromtheta}.
The second scale $\eps\ll\delta$,
{used for the construction}
in Proposition~\ref{propproj}, is the one on which we construct a piecewise affine approximation of $u$. {This is achieved in each cube at scale $\delta$ by applying Proposition~\ref{propproj} to the extensions, obtained via Theorem~\ref{propextension}, of $u$ itself restricted to domains separated by the regular part of its jump set}.
In turn, the {regularization $L_z$ of the} jump set $J_u$ will be separately approximated using piecewise affine elements in $\R^{n-1}$ using again the scale $\eps$.
{Figure~\ref{figgrids2} shows a sketch of the construction, the different parts will become clear during the proof.}

\begin{figure}
 \begin{center}
  \includegraphics[width=10cm]{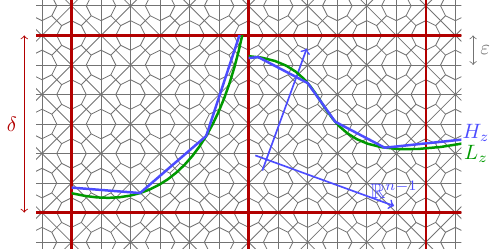}
 \end{center}
\caption{{Sketch of the grids used in the proof of Theorem~\ref{theodensityintro}.
The grid  $(\mathcal T',V')$ is taken in $\R^{n-1}$, {it} is then rotated. Similarly, $H_z$ and $L_z$ are graphs (of $\psi_z$ and $\varphi_z$, respectively) in these rotated coordinates.
}}
\label{figgrids2}
\end{figure}

\begin{proof}[Proof of Theorem~\ref{theodensityintro}]
Let $u\in SBV(\Omega;\R^m)$ and $\theta\in{(0,\frac12]}$. To simplify {the} notation we work at fixed $\theta$, and in the end choose a sequence $\theta_j\to0$.
It is not restrictive to assume additionally that
\begin{equation}\label{s<g0}
{t\leq g_0(t) \text{ for all } t\in [0,\infty)};
\end{equation}
indeed, this follows by proving first the theorem with $g_0(t)+t$ in place of $g_0$ and then deducing the statement for $g_0$ as a by-product.
By \eqref{e:g0 growth}, for any $\lambda>0$ (fixed below, it will depend on $\theta$ and $\delta$ but not on $\eps$ and $\delta'$;
{$\lambda=\theta\delta$ will do}) there is $C_\lambda\ge1$ such that
\begin{equation}\label{eqg0slambdapff}
 g_0({t})\le \lambda+C_\lambda {t} \text{ for all }  {t\in [0,\infty)}.
\end{equation}
This will be used to estimate terms of the form $g_0({|}[w]{|})$  on sets of finite $n-1$-dimensional measure in terms of the jump.
\newcommand{\itemtit}[1]{\item{\emph{#1.}}}
\begin{schrittlist}
 \itemtit{Choice of the scale $\delta$ on which $u$ is regular and of the sets $A_\delta$, $A_\delta^*$}

By {Theorem}~\ref{propextension} we can assume that $u\in SBV(\R^n;\R^m)$ {with $|Du|(J_u\cap \partial \Omega)=0$,}
\begin{equation}
\int_{\Omega'}|\nabla u|^p\,dx\leq\int_{\Omega}|\nabla u|^p\,dx+\theta\,,
\end{equation}
and
\begin{equation}
\int_{\Omega'\cap J_u}g_0(|[u]|)d\calH^{n-1}\leq\int_{\Omega\cap J_u}g_0(|[u]|)d\calH^{n-1}+\theta\,,
\end{equation}
for some {bounded} open set $\Omega'$ with $\overline\Omega\subset\Omega'$ and $|\Omega'|\le 2 |\Omega|$.
We choose {$\delta_0\in(0,\theta]$} such that
{$3\delta_0\sqrt n\le \dist(\Omega,\partial\Omega')$ and}
\begin{equation}\label{eqchoicedelta0}
\int_{(\partial\Omega)_{3\delta_0\sqrt n}}|\nabla u|^p dx +
 \mu_u( (\partial\Omega)_{3\sqrt n \delta_0})\le \theta,
\end{equation}
{where $\muj:=
g_0(|[u]|)\calH^{n-1}\LL J_u$ as in \eqref{e:muj}.}
If $\calH^{n-1}(J_u{\cap\Omega})<\infty$,
{then Theorem~\ref{propextension} gives $\calH^{n-1}(J_u)<\infty$ and we may}
also require
\begin{equation}\label{eqchoicedelta02}
\calH^{n-1}(J_u\cap (\partial\Omega)_{3\sqrt n \delta_0})\le \theta.
\end{equation}
By Proposition~\ref{propdeltafromtheta}, used with {$\delta_0$}
in place of $\theta$, there is
$\delta\in (0,{\delta_0}]\subseteq(0,\theta]$ such that, with
${A_\delta}:=\{z\in (\delta\Z^n)\cap\Omega: \dist(z,{\partial\Omega}){>}{\delta\sqrt n}\}$, there are
$R:{A_\delta}\to\SO(n)$,  $s:{A_\delta}\to \R^m$,  $\eta:{A_\delta}\to\R^{m\times n}$, $\varphi:A_\delta\to {C^1_c}(\R^{n-1})$, $x:A_\delta\to\R^n$
{which, setting}
\begin{equation}\label{eqdefLz}
 L_z:=x_z+R_z\{(y',\varphi_z(y')): y'\in\R^{n-1}\}
\end{equation}
and $\Qlargez:=z+(-\delta,\delta)^n$,
{satisfy}
 $\|D\varphi_z\|_{L^\infty}\le \theta$ and
\begin{equation}\label{e:approx scala delta}
\begin{split}
\sum_{z\in {A_\delta}} &\int_{\Qlargez}|\nabla u-\eta_z|^p dx +
\sum_{z\in {A_\delta}}
\int_{\Qlargez\cap J_u\setminus L_z} 
g_0(|[u]|)d\calH^{n-1}\\
+&
\sum_{z\in {A_\delta}}\int_{\Qlargez\cap L_z}[g_0(|[u]-s_z|){+g_0(|[u]|)|\nu_u-R_ze_n|]}d\calH^{n-1}\le C\theta\,.
\end{split}
\end{equation}
{Here and in what follows we do not explicitly indicate the}
dependence {of $C$} on $\mu_u(\Omega)$ and $|\Omega|$.
If $\calH^{n-1}(J_u)<\infty$ we have in addition
\begin{equation}\label{e:approx scala delta02}
\sum_{z\in {A_\delta}}\calH^{n-1}(\Qlargez\cap (J_u{\triangle} L_z))\le C\theta.
\end{equation}

We further define $A_\delta^*:=\{z\in (\delta \Z)^n: \dist(z,\partial\Omega)\le {\delta\sqrt n}\}$.
We observe that
${Q^*_z}\subseteq B_{\delta\sqrt n}(z)$, so that
by \eqref{eqchoicedelta0}
\begin{equation}\label{e:approx scala delta2}
\begin{split}
\sum_{z\in {A_\delta^*}} \int_{\Qlargez}|\nabla u|^p dx &+
\sum_{z\in {A_\delta^*}}
\int_{\Qlargez\cap J_u} 
g_0(|[u]|)d\calH^{n-1}\\
&\le C(\int_{(\partial\Omega)_{2\delta\sqrt n}}|\nabla u|^p dx +\mu_u((\partial\Omega)_{2\delta\sqrt n}))\le C\theta\,.
\end{split}
\end{equation}

For $\gamma\in B_{\sfrac\delta4}$ and $z\in A_\delta{\cup A_\delta^*}$
we define $Q_z^\gamma:=\gamma+z+(-\delta/2,\delta/2)^n\subseteq \Qlargez$,
and observe that since $2\delta\sqrt n\le {2\delta_0}\sqrt n\le
\dist(\Omega,\partial\Omega')$
we have $\Omega\subset \bigcup_{z\in A_\delta{\cup A_\delta^*}}\overline{Q_z^\gamma}{\subset \Omega'}$.
Further, for $\calL^n$-almost every choice of $\gamma\in B_{\sfrac\delta4}$
we have
\begin{equation}\label{eqmujbdryv}
\calH^{n-1}\left(J_u\cap \bigcup_{z\in A_\delta\cup A_\delta^*}\partial Q^\gamma_z\right)=0\quad\text{and}\quad \calH^{n-1}\left(\bigcup_{z\in {A_\delta}} (L_z\cap \partial Q_z^\gamma)\right)=0.
\end{equation}
This follows from the fact that $\calH^{n-1}\LL{J_u}$ and
$\calH^{n-1}\LL\bigcup_{z\in A_\delta}L_z$ are $\sigma$-finite.
In the rest of the proof $\gamma$ is a fixed value with property \eqref{eqmujbdryv} {and we write $Q_z$ in place of $Q_z^\gamma$}.

\itemtit{Approximation of the interface}\label{stepapproxif}

Let $\eps\in(0,\frac\delta2)$.
For every $z \in A_\delta$, we let
\begin{equation}\label{eqdeflzpiu}
 L_z^+:={x_z+}R_z\{(y',y_n): y'\in\R^{n-1},\ 
 y_n>\varphi_z(y')\},
\end{equation}
so that $L_z=\partial L_z^+$,
and then let $L_z^-:=\R^n\setminus L_z\setminus L_z^+$.
Fix a triangulation $(\mathcal T',V')$ of $\R^{n-1}$, 
with $V'=\eps\Z^{n-1}$, as in Lemma~\ref{lemmatriangRn}.
 We define $\psi_z:\R^{n-1}\to\R$ setting $\psi_z=\varphi_z$ on $V'$, and $\psi_z$ affine in each element of $\mathcal T'$.

 We stress that the triangulation 
 $(\mathcal T',V')$ used above to approximate $L_z$
 is not related to the triangulation 
 $(\mathcal T_{\eps,\zeta}, V_{\eps,\zeta})$ 
used in Proposition~\ref{propproj}
 for the definition of $\Pi_{\eps,\zeta}$.
 The usage of the same scale $\eps$ for both triangulations is only to avoid having one more small parameter. In any case, it would be crucial for both scales
to be much smaller than $\delta$.

  {We claim that there is a modulus of continuity}
 $\omega_\eps$,  infinitesimal as $\eps\downarrow 0$, {such that for all $z\in A_\delta$}
 we have
\begin{equation}\label{e:varphi psi}
\|\varphi_z-\psi_z\|_{L^\infty(\R^{n-1})}\le {\eps\omega_\eps} \mbox{ and } \|D\varphi_z-D\psi_z\|_{L^\infty(\R^{n-1})}\le {\omega_\eps}.
\end{equation}
In the following we 
shall assume that $\eps$ is sufficiently small to ensure
$\omega_\eps\le\theta$.
To prove \eqref{e:varphi psi}, we observe that
since $\varphi_z\in {C^1_c(\R^{n-1})}$ there is a modulus of continuity $\hat\omega_\eps$ such that
$|D\varphi_z(y)-D\varphi_z({\tilde y})|\le  \hat\omega_\eps$ whenever  $|y-{\tilde y}|\le \eps\sqrt n$. As there are finitely many choices of $z$, we can assume that $\hat\omega_\eps$ does not depend on $z$. Consider now an element $\tau'\in\mathcal T'$.
For every edge $(a,b)$ of $\tau'$
we have $D\psi_z|_{\tau'}(b-a)=\varphi_z(b)-\varphi_z(a)$, so that
\begin{equation*}
\left| (D\psi_z|_{\tau'}-
D\varphi_z(b))(b-a)\right|\le \hat\omega_\eps |b-a|
\end{equation*}
for all edges $(a,b)$ of $\tau'$.
As the simplexes $\tau'$ are uniformly nondegenerate this implies $|D\psi_z-D\varphi_z|\le C\hat\omega_\eps$, for
some constant $C$ depending only on $n$. This proves \eqref{e:varphi psi}.

Using this interpolation  {and} {a shift $\beta\in (-\eps,\eps)$ we define}
the set
\begin{equation}\label{eqdefHzp}
 H_z^+:={x_z+}R_z\{(y',y_n): y'\in\R^{n-1},\ 
 y_n>\psi_z(y'){+\beta}\},
\end{equation}
which is a polyhedral approximation of $L_z^+$,
{and} $H_z:=\partial H_z^+$, $H_z^-:=\R^n\setminus H_z\setminus H_z^+$
(see Figure~\ref{figgrids2}). {We choose $\beta$ such that}{
	\begin{equation}\label{eqmujbdryv2}
	\calH^{n-1}\left(\bigcup_{z\in {A_\delta}} (H_z\cap \partial Q_z)\right)=0.
	\end{equation}
Condition \eqref{eqmujbdryv2}, which {holds for almost all $\beta$}, will be needed to estimate the (unilateral) $\calH^{n-1}$-difference between the jump of the approximation and $J_u$, see {text after \eqref{eqHz} and the proof of \eqref{Jw-Ju2}.}}

\begin{figure}
 \begin{center}
  \includegraphics[width=5cm]{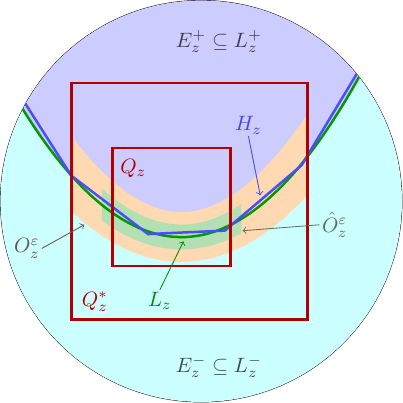}
 \end{center}
\caption{{Sketch of geometry in the construction of $U_z^+$ and  $U_z^-$ in Step~\ref{stepcostrwzzeta} in the proof of Theorem~\ref{theodensityintro}, assuming $n=2$, $R_z=\Id$, and that
$L_z$ is the graph of a parabola. The set $E_z^+$ is the area above $L_z$ (inside the  ball $B_z$), $O_z^\eps$ 
is a neighborhood of $L_z$ intersected with the larger cube $Q_z^*$, and 
$\hat O_z^\eps$ is a smaller neighborhood of $L_z$ intersected with a neighborhood of $Q_z$. The set $H_z$ is an approximation of $L_z$ with the graph of a piecewise affine function, and the part inside $Q_z$ belongs to $\hat O_z^\eps$.
Figure~\ref{figgrids2} shows how this construction interacts with the rest.}}
\label{figgrids3}
\end{figure}

\itemtit{Construction of $w_{z,\zeta}$ and $w_\zeta$}\label{stepcostrwzzeta}
{In this step we} define an approximation ${w_{z,\zeta}}$ on each cube $Q_z$, for $z\in A_\delta$ and $\zeta\in B_\eps$.
{This requires two different extensions of $u$ on different sets, a sketch is given in Figure~\ref{figgrids3}}.

If $L_z\cap Q_z^*=\emptyset$ we set $w_{z,\zeta}=\Pi_{\eps,\zeta} u$.
The other case is more complex.  We pick $y_z\in L_z\cap {Q_z^*}$,
let $B_z:=B_{3\sqrt n \delta}(y_z)$,
so that
 $Q_z^*\subset\subset B_z\subset \Omega'$,
and consider the sets $E_z^\pm:=B_z\cap L_z^\pm$.
One checks that,
 since $\varphi_z$ is $\theta$-Lipschitz {with $\theta\leq \frac12$,}
 the sets $E_z^\pm$ are
Lipschitz with a constant {$C$}. For this it is convenient to use that a bounded open set {$A$} is Lipschitz if and only if there are a nontrivial open one-sided cone ${E}$ and $r>0$ such that
$B_r(x)\cap (x+{E})\subset {A}$
and $B_r(x)\cap (x-{E})\cap {A}=\emptyset$
for all $x\in\partial {A}$, see \cite[Remark after Def.~2.60]{AFP}.

Let $M>0$ be fixed, {it will be chosen below depending only on the dimension $n$}.
By {Theorem}~\ref{propextension} there are $U^\pm_z\in SBV(\R^n;\R^m)$
which extend the restriction of $u$ to $E_z^\pm$, respectively. In particular, we have
$U^\pm_z=u$ on $E_z^\pm$,
and $|{DU^+_z}|(\partial E_z^+)=
 |{DU^-_z}|(\partial E_z^-)=0$.
For $\eps$ sufficiently small, letting
${O_z^\eps:= (L_z)_{2M\eps}\cap Q_z^*}$,
from $\bigcap_{\eps>0} O_z^\eps=L_z\cap Q_z^*
\subset \partial E_z^+$ we obtain
\begin{equation}\label{eqeUpm}
 E_{|\cdot|^p,g_0} [U^+_z, {{O_z^\eps}}]
+|DU^+_z| ({{O_z^\eps}})
{+ \|\nabla u\|_{L^p(O_z^\eps)}^p+
|Du|(O_z^\eps\setminus L_z)}
\le\frac{\delta^n\theta}{C_\lambda}
\end{equation}
{for all $z$,}
and the same for $U^-_z$.
If $\calH^{n-1}(J_u)<\infty$, then we also have {for $\eps>0$ small enough}
\begin{equation}\label{eqeUpm1}
{\sum_{z\in A_\delta}}
\calH^{n-1}(J_{U^+_z}\cap {O_z^\eps})
 + \calH^{n-1}(J_{U^-_z}\cap  O_z^\eps)
\le {C\theta.}
\end{equation}
{The function $w_{z,\zeta}$ will be defined as a discretization of $U_z^+$ on
$H_z^+\cap Q_z$, and similarly with the other sign.
By Proposition~\ref{propproj}\ref{propprojlocal} it depends on
$U_z^\pm$ on a small neighborhood of the
sets $H_z^\pm\cap Q_z$, and by
Proposition~\ref{propproj}\ref{propprojbounds} the relevant properties of $w_{z,\zeta}$  can be estimated by
corresponding properties of
$U_z^\pm$ on $2\cpropjbounds\eps$-neighborhoods of
$H_z^\pm\cap Q_z$ (see also~\eqref{e:stima bulkomega}). Therefore we  need to estimate $U^\pm_z$ on these sets.

Let $\hat O_z^\eps:=(L_z)_{M\eps}\cap (Q_z)_{M\eps}$.
From
\eqref{eqdeflzpiu},
\eqref{e:varphi psi} and \eqref{eqdefHzp} we obtain
$H_z\subseteq (L_z)_{2\eps}$ and
$(H_z^+)_{2\cpropjbounds\eps}
\subseteq L_z^+\cup (L_z)_{(2+2\cpropjbounds)\eps}$, which imply
\begin{equation}\label{eqHQzceps}
(H_z^+\cap Q_z)_{2\cpropjbounds\eps}\subseteq (L_z^+\cap Q_z^*)\cup \hat O_z^\eps
\end{equation}
and the same for the other sign,
provided that $M\ge 2+2\cpropjbounds$ and $\eps$ is sufficiently small.}
Recalling \eqref{e:approx scala delta}, {\eqref{eqeUpm}} in particular implies
\begin{equation}\label{eqestjumpUpm}
\begin{split}
& \sum_{z\in A_\delta}
 \int_{J_{U_z^+}\cap {(H_z^+\cap Q_z)_{2\cpropjbounds\eps}}} g_0(|[U_z^+]|)d\calH^{n-1}\\
 &+\sum_{z\in A_\delta}
 \int_{J_{U_z^-}\cap {(H_z^-\cap Q_z)_{2\cpropjbounds\eps}}} g_0(|[U_z^-]|)d\calH^{n-1}\le C\theta
\end{split}
 \end{equation}
and
\begin{equation}
 \sum_{z\in A_\delta}
 \int_{{(H_z^+\cap Q_z)_{2\cpropjbounds\eps}}}
 |\nabla U_z^+-\eta_z|^p dx
+\int_{{(H_z^-\cap Q_z)_{2\cpropjbounds\eps}}}
|\nabla U_z^--\eta_z|^p dx
 \le C \theta.
\end{equation}

We next estimate the difference between $u$, $U^+_z$ and $U^-_z$ around $H_z$, this will be important after \eqref{eqIz2area} (cf. \eqref{e:Iz2Bi}-\eqref{e:Iz2Bi second line}).
We pick $x_1,\dots, x_K\in L_z$ such that
\begin{equation}\label{eqdefBi}
L_z\cap Q_z\subseteq\bigcup_i B_i, \hskip 1cm
 B_i:=B_{M\eps}(x_i),
\end{equation}
and
\begin{equation}\label{eqdefBi2}
\hat O_z^\eps\subseteq\bigcup_i B_i^*, \hskip 1cm
 B_i^*:=B_{2M\eps}(x_i).
\end{equation}
{This implies in particular $K\leq \frac{{C}{\delta^{n-1}}}{\eps^{n-1}},$ for a constant ${C}$.}
We can pick the points so that
the larger balls $B_i^*$ have finite overlap uniformly in $\eps$,
namely $\sum_{i=1}^K\chi_{B_i^*}\leq {C}$, and that they are all contained in $(Q_z)_{3M\eps}$.
For $\eps$ sufficiently small, $B_i^*\subseteq O_z^\eps$.
As $\varphi_z$ is ${\frac{1}{2}}$-Lipschitz, the sets $B_i^*\cap L_z^+$ and $B_i^*\cap L_z^-$ are uniformly Lipschitz, therefore there is a constant {$C$} such that for any $i$ there are $h^+_i$, $h^-_i\in\R^m$ with
\begin{equation}\label{equhiplbilp}
 \frac1\eps \|u-h^+_i\|_{L^1(B_i^*\cap L_z^+)}
 +\|Tu-h^+_i\|_{L^1(B_i^*\cap L_z;\calH^{n-1})}
 \le C |Du|(B_i^*\cap L_z^+).
\end{equation}
In \eqref{equhiplbilp} we write $Tu$ for the inner trace on the boundary of $B_i^*\cap L_z^+$, which on $B_i^*\cap L_z$ coincides with $u^+$.
The corresponding estimate holds with the other sign (then with $Tu=u^-$).
This in particular implies
\begin{equation}\label{eqestbilzuuhh}
 \int_{B_i^*\cap L_z}|[u]-(h_i^+-h_i^-)|d\calH^{n-1}
\le   C  |Du|(B_i^*\setminus L_z).
\end{equation}

We observe that $\Lip(\varphi_z)\le{\frac{1}{2}}$ also implies for some constant {$C$}
\begin{equation}\label{eqvolbistlz}
 \frac{\eps^n}{{C}}\le |B_i^*\cap L_z^+|,\hskip5mm
 \frac{\eps^n}{{C}}\le |B_i^*\cap L_z^-|,
\end{equation}
as well as
\begin{equation}\label{eqestimateBiLzHz}
 \frac{\eps^{n-1}}{{C}}\le \calH^{n-1}(B_i\cap L_z)\le {C}\eps^{n-1},\hskip3mm
 \frac{\eps^{n-1}}{{C}}\le \calH^{n-1}(B_i\cap H_z)\le {C}\eps^{n-1}.
\end{equation}
By Poincar\'e's inequality applied to $U_z^+$ on $B_i^*$, using $u=U^+_z$ on $B_i^*\cap L_z^+$, \eqref{equhiplbilp} and \eqref{eqvolbistlz},
\begin{equation}\label{eqpointUpm}
 \|U^{{+}}_z-h_i^{{+}}\|_{L^1(B_i^*)}
 \le C \eps |DU^{{+}}_z|(B_i^*)
\end{equation}
and analogously for $U_z^-$ and $h_i^-$,
so that
\begin{equation}\label{eqUzpUzmhhfd}
 \int_{B_i^*}|U_z^+-U_z^--h_i^++h_i^-|dx
 \le C\eps ( |DU^+_z|+|DU^-_z|)(B_i^*).
\end{equation}
Finally, a direct application of Poincar\'e's inequality to $U^+_z-u$ on $B_i^*$, {using $u=U^+_z$ on $B_i^*\cap L_z^+$ and \eqref{eqvolbistlz},} leads to
\begin{equation}\label{eqUzpUzmhhfdUu}
 \int_{B_i^*}|U_z^+-u|dx
 \le C\eps ( |DU^+_z|+|Du|)(B_i^*),
\end{equation}
obviously the same holds for $U^-_z$.
Summing \eqref{eqUzpUzmhhfdUu} over all balls shows that
\begin{equation}\label{eql1distance}
 \|U_z^+-u\|_{L^1({\hat O_z^\eps})} \le C\eps (|DU^+_z|+|Du|)(O_z^\eps)
\end{equation}
and the same for $U_z^-$. Summing instead \eqref{eqUzpUzmhhfdUu} only over the balls with centers contained in ${(\partial Q_z)_{3 M\eps}\cap L_z}$ gives
\begin{equation}\label{eqUpUmsz2}
 \|U_z^+-u\|_{L^1((\partial Q_z)_{M\eps}\cap (L_z)_{M\eps})}
 \le C \eps {\left(|DU_z^+|
 +|Du|\right)
 ((\partial Q_z)_{5M\eps}\cap O_z^\eps)},
\end{equation}
and the same bound for $U_z^-$.

For $z\in A_\delta$ with $L_z\cap {Q_z^*}\neq\emptyset$ we define
$w_{z,\zeta}:{\R^n}\to\R^m$ by
\begin{equation}\label{e:wzzeta}
 w_{z,\zeta}:=\begin{cases}
\Pi_{\eps,\zeta} U^+_z&\text{ in }
{H_z^+},\\
\Pi_{\eps,\zeta} U^-_z&\text{ in }
 {H_z^-,}
\end{cases}
\end{equation}
we recall that if instead $L_z\cap {Q_z^*}=\emptyset$ we
had set  $w_{z,\zeta}=
\Pi_{\eps,\zeta} u$.
In both cases, the
function $w_{z,\zeta}$ is piecewise affine.
For almost all $\zeta\in B_\eps$ we have
\begin{equation}\label{eqHz}
{\calH^{n-1}(J_{\Pi_{\eps,\zeta} f}\cap H_z)=
\calH^{n-1}(J_{\Pi_{\eps,\zeta} f}\cap \partial Q_z)=0}
\end{equation} 
for any function $f$, and in particular for $U^+_z$ and $U^-_z$. With this choice, and recalling \eqref{eqmujbdryv2},
$J_{w_{z,\zeta}}\cap \overline Q_z$ splits (up to $\calH^{n-1}$-null sets) into the disjoint union of
$J_{\Pi_{\eps,\zeta} U^+_z}\cap H_z^+\cap  Q_z$,
$J_{\Pi_{\eps,\zeta} U^-_z}\cap H_z^-\cap Q_z$,
and a subset of $H_z\cap  Q_z$,
 with
\begin{equation}\label{eqsaltowzzeta}
[w_{z,\zeta}]=
\Pi_{\eps,\zeta} U^+_z
-\Pi_{\eps,\zeta} U^-_z, \hskip4mm \text{$\calH^{n-1}$-a.e. on $H_z\cap Q_z$.}
\end{equation}
{
By \eqref{e:varphi psi} and the fact that $\varphi_z$ is $\theta$-Lipschitz we also obtain
\begin{equation}\label{eqnuwzeta}
     |\nu_{w_{z,\zeta}}-R_ze_n|\le C\theta
     \hskip5mm
     \text{$\calH^{n-1}$-a.e.  on } Q^*_z\cap H_z\cap J_{w_{z,\zeta}}.
    \end{equation}}

If $\calH^{n-1}(J_u)<\infty$ then using \eqref{e:stima superficiearea}
separately on $U^+_z$ and $U^-_z$ (cf. the discussion to get \eqref{e:stima bulkomega}), then \eqref{eqHQzceps},
$U^+=u$ on $L_z^+\cap Q_z$, and finally
\eqref{e:approx scala delta02} and \eqref{eqeUpm1} we obtain
\begin{equation}\label{eqjumpnotH}
\begin{split}
 &\strokedint_{B_\eps} \sum_{z\in A_\delta}
 \calH^{n-1}(\overline Q_z\cap H^+_z\cap
 J_{\Pi_{\eps,\zeta} U^+_z}
 )+\calH^{n-1}(\overline Q_z\cap H^-_z\cap
 J_{\Pi_{\eps,\zeta} U^-_z}
 )d\zeta\\
&\le  C\sum_{z\in A_\delta}
 \calH^{n-1}(Q_z^*\cap
 J_{u}\setminus L_z
 )+
\calH^{n-1}( O_z^\eps\cap
 J_{U^+_z}
 )+\calH^{n-1}(O_z^\eps\cap
 J_{U^-_z} )
 \le C\theta.
 \end{split}
\end{equation}

{We define
$w^0_\zeta:=\Pi_{\eps,\zeta} u$ and then}
$w_\zeta\in SBV(\bigcup_{z\in {A_\delta\cup A_\delta^*}}Q_z;\R^m)$ by setting
 $w_\zeta:={w_{z,\zeta}}$ on $Q_z$ if $z\in A_\delta$, and $w_\zeta:=w^0_\zeta$ if $z\in A_\delta^*$. For any $\zeta\in B_\eps$, the function $w_\zeta$ is piecewise affine
 {and obeys property~\ref{theodensityintrosimpl}.} This concludes the construction of $w_\zeta$.

\itemtit{Estimates on $w_\zeta$ and $\nabla w_\zeta$}

We first check that we have not added too much jump on the boundary between
adjacent cubes by replacing $w^0_{\zeta}$ by $w_{z,\zeta}$ {and compute
with \eqref{eqmujbdryv2} and \eqref{eqHz}}
\begin{equation}\begin{split}
 \int_{\partial Q_z} |w^0_\zeta-w_{z,\zeta}| d\calH^{n-1}=
&
 { \int_{\partial Q_z\cap H_z^+} |\Pi_{\eps,\zeta} u-\Pi_{\eps,\zeta}U^+_z| d\calH^{n-1}}\\
& +
{  \int_{\partial Q_z\cap H_z^-} |\Pi_{\eps,\zeta} u-\Pi_{\eps,\zeta}U^-_z| d\calH^{n-1}}
 \end{split}\end{equation}
for $z\in A_\delta$.
By estimate \eqref{eqpsurface} with 
$\Sigma=\partial Q_z\cap H_z^+$,
\begin{equation}\begin{split}
 \strokedint_{B_\eps}\int_{\partial Q_z\cap H_z^+} |w^0_\zeta-w_{z,\zeta}| d\calH^{n-1}d\zeta
\le &
\frac{C}{\eps}
\|u-U^+_z\|_{L^1((\partial Q_z\cap H_z^+)_{\cpropjbounds\eps})} \\
&+ 
C   |D(u-U^+_z)|( (\partial Q_z\cap H_z^+)_{2\cpropjbounds\eps}).
\end{split}\end{equation}
Recalling that $U_z^+=u$ on $E_z^+$
and \eqref{eqHQzceps},
 both domains can be restricted to $(\partial Q_z)_{M\eps} \cap (L_z)_{M\eps}$. The first term is estimated by 
\eqref{eqUpUmsz2}, and we conclude
\begin{equation}
 \strokedint_{B_\eps}\int_{\partial Q_z} |w^0_\zeta-w_{z,\zeta}| d\calH^{n-1}d\zeta
\le   
C   (|Du|+|DU^+_z|+|DU^-_z|)( (\partial Q_z)_{5M\eps}\cap  O_z^\eps)
\end{equation}
so that, using \eqref{eqmujbdryv},
\begin{equation}\label{eqdiff}
\limsup_{\eps\to0}\sum_{z\in A_\delta}
\strokedint_{B_\eps}
\int_{\partial Q_z} |w^0_\zeta-w_{z,\zeta}| d\calH^{n-1}d\zeta{=0}.
\end{equation}

We next address the $L^1$ convergence.
For any $z\in A_\delta$ {for which $L_z\cap {Q_z^*}\neq \emptyset$,} by the definition of $w_\zeta$ and Proposition~\ref{propproj}\ref{propprojaffine},
we have for {$\calL^n$}-almost every $\zeta{\in B_\eps}$
\begin{equation}\begin{split}\label{stimaL1}
 \int_{Q_z} |w_\zeta-u|dx
 =&  \int_{Q_z\cap H_z^+}  |\Pi_{\eps,\zeta} U_z^+-u|dx+\int_{Q_z\cap H_z^-}  |\Pi_{\eps,\zeta} U_z^--u|dx\\
\le &  {\int_{Q_z\cap H_z^+}  \left(|\Pi_{\eps,\zeta} U_z^+-U_z^+|
+| U_z^+-u|\right)}
dx\\
 &+
{\int_{Q_z\cap H_z^-}  \left(|\Pi_{\eps,\zeta} U_z^--U_z^-|
+| U_z^--u|\right)
dx.}
 \end{split}
 \end{equation}
{We recall that $U^+_z=u$ on $Q_z\cap H_z^+\setminus \hat O_z^\eps\subseteq E_z^+$
and \eqref{eql1distance} to obtain
\begin{equation}\begin{split}
\int_{Q_z\cap H_z^+}  
| U_z^+-u|dx
\le &
 \| U_z^+-u\|_{L^1(\hat O_z^\eps)} \\
 \le& C\eps 
 (|DU^+_z|+|DU^-_z|+|Du|)(O_z^\eps).
\end{split}\end{equation}
With
\eqref{eqnorm1} in Proposition~\ref{propproj},}
\begin{equation}\label{eqPiUzUz}
{\begin{split}
\strokedint_{B_\eps}\int_{Q_z\cap H_z^+}  |\Pi_{\eps,\zeta} U_z^+-U_z^+|
dx\,d\zeta
&\le C\eps |DU_z^+|((Q_z\cap H_z^+)_{2\cpropjbounds\eps})\\
&\le C\eps (|Du|(Q_z^*) + |DU_z^+|(O_z^\eps))
\end{split}}\end{equation}
{and an analogous estimate holds for the term with the other sign.}
{Recalling \eqref{eqeUpm} we obtain}
 \begin{equation}
{ \sum_{z\in A_\delta:\,L_z\cap Q_z\neq\emptyset}}
\strokedint_{B_\eps}\int_{{Q_z}}| w_\zeta- u|dx\,d\zeta
\leq  C\theta
\end{equation}
for $\eps$ sufficiently small.
If $L_z\cap Q_z^*=\emptyset$ or $z\in A_\delta^*$ the
computation is simpler as $w_\zeta=w_\zeta^0=\Pi_{\eps,\zeta}u$,
and only \eqref{eqPiUzUz} with $u$ in place of $U^+_z$ appears. From this we conclude that
\begin{equation}\label{eqnorm2}
\strokedint_{B_\eps}\int_\Omega| w_\zeta- u|dx\,d\zeta\leq C\theta.
\end{equation}
 
Moreover, from {$w_\zeta=w_\zeta^0$ on
$Q_z\setminus \hat O_z^\eps$}, from \eqref{eqeUpm} and estimate \eqref{e:stima bulk} in Proposition~\ref{propproj}\ref{propprojbounds} we deduce
\begin{align}\label{e:energia bulk w} 
&\strokedint_{B_\eps}\int_\Omega|\nabla w_\zeta-\nabla u|^pdx\,d\zeta\notag\\
&\leq {C}
\strokedint_{B_\eps}\sum_{z\in A_\delta{\cup A_\delta^*}}{\Big(}\int_{Q_z}|{\nabla w_{\zeta}^0}-\nabla u|^pdx{+\int_{{\hat O_z^\eps}}(|\nabla w_\zeta|^p+|\nabla u|^p)dx\Big)}d\zeta
\notag\\
&\leq C\sum_{z\in A_\delta{\cup A_\delta^*}}\int_{(Q_z)_{2\cpropjbounds\eps}} |\nabla u-\eta_z|^p dx{+C\theta}\leq C\theta,
\end{align}
where the last estimate follows from {\eqref{e:approx scala delta}}, and the constant
$C>0$ depends {on $u$,} on the dimension $n$, {on $p$ and on $\Omega$}.

\itemtit{Definition of the deformation $\Phi$}

We select $\delta'\in(0,\delta)$ and define $\Phi:\R^n\to\R^n$ by
\begin{equation}\label{e:def Phi bilip}
\Phi(x):=x+\sum_{z\in A_\delta: {L_z\cap Q_z^*\ne\emptyset}}
    \alpha_z(x)R_z\left(\psi_z({(R_z^T{(x-x_z)})'})+{\beta}-\varphi_z({(R_z^T{(x-x_z)})'})\right) e_n,
\end{equation}
where $(R_z^Tx)'$ denotes the first $n-1$ components of the vector $R_z^Tx$ and
we fixed a function $\alpha_z\in C^\infty_c(Q_z;[0,1])$ with $\alpha_z=1$ on
$Q_z'':=z+\gamma+(-\frac{\delta+\delta'}4,\frac{\delta+\delta'}4)^n$,
$|\Dnabla\alpha_z|\le 6/(\delta-\delta')$.
The map $\Phi$ is Lipschitz since $\alpha_z$, $\psi_z$, $\varphi_z$ are, by
definition $\Phi(x)=x$ for $x\not\in\Omega$, by \eqref{e:varphi psi}
$|\Phi(x)-x|\le {2\eps}$ for all $x$, we can assume ${2\eps}\le (\delta-\delta')/4$.
For $\eps$ sufficiently small, $\Phi(Q_z)=Q_z$ for all $z\in\delta\mathbb{Z}^n$ by construction, and if
we define $Q_z':=z+\gamma+(-\frac{\delta'}2,\frac{\delta'}2)^n\subset\subset
Q_z''\subset\subset Q_z$, from \eqref{eqdefLz}, \eqref{e:varphi psi} and \eqref{eqdefHzp} 
we obtain
\begin{equation}\label{eqHQPhiL0}
 H_z{\cap Q_z'}=\Phi(L_z){\cap Q_z'}, {\text{ and }}{
 \Phi^{-1}(H_z){\cap Q_z'}=L_z{\cap Q_z'}.}
\end{equation}
In order to show that $\Phi$ is invertible with Lipschitz inverse it suffices to prove that {$\Dnabla\Phi$ is uniformly close to the identity.} Indeed, using \eqref{e:varphi psi}
we obtain
\begin{align}
 \|\Dnabla\Phi-\Id\|_\infty\le& \max_{z\in A_\delta} \|\Dnabla\psi_z-\Dnabla\varphi_z\|_\infty
 + \max_{z\in A_\delta} \|\Dnabla\alpha_z\|_\infty
 (\|\psi_z-\varphi_z\|_\infty{+{|\beta|}})
\nonumber \\
\le& \omega_\eps +  {\frac{{12\eps}}{\delta-\delta'}}\,.
 \label{eqboundDPhi0}
\end{align}
 In particular, if $\eps$ is sufficiently small
 on a scale depending on $\delta$ and 
 $\delta'$, we can ensure
\begin{equation}\label{eqboundDPhi}
 \|\Dnabla\Phi-\Id\|_\infty\le {\frac12}\theta.
\end{equation}
Therefore,
\[(1-{\frac12}\theta)|x-y|\leq|\Phi(x)-\Phi(y)|\leq (1+\theta)|x-y|,\]
which implies that $\Phi$ is globally bilipschitz.
{Property~\ref{theodensityintroLp} follows.}

\itemtit{Estimate of the jump energy}

We are now able to estimate the energy of the jump contribution. We start to decompose it as 
 \begin{align}\label{eqsaltoPhi}
\strokedint_{B_\eps}&\int_{\Omega\cap(J_u\cup \Phi^{-1}(J_{w_\zeta}))}g_0(|[u]-[w_{\zeta}]\circ\Phi|)\,d\calH^{n-1}d\zeta\notag\\
\le&\strokedint_{B_\eps}
\sum_{z\in A_\delta\cup A_\delta^*}
\int_{(J_u\cup\Phi^{-1}(J_{w_\zeta}))\cap Q_z}g_0(|[u]-{[w_{\zeta}]}\circ\Phi|)\,d\calH^{n-1}d\zeta\notag\\
 &+\strokedint_{B_\eps}\sum_{z\in A_\delta}\int_{\partial Q_z}g_0(|w_{z,\zeta}-w^0_\zeta|)d\calH^{n-1}d\zeta
\end{align}
where we
separated the boundary contributions from the ones in the interior, then used $\Phi(x)=x$ on $\partial Q_z$, 
\eqref{eqmujbdryv} to infer that $[u]=0$ almost everywhere on $\partial Q_z$, {\eqref{eqHz} to infer $[w_{z,\zeta}]=[w_z^0]=0$ almost everywhere on $\partial Q_z$,}
and finally used
$|[w_\zeta]|\le |w_{z,\zeta}-w_\zeta^0|+
|w_{z',\zeta}-w_\zeta^0|$ and subadditivity of $g_0$
on $\partial Q_z\cap \partial Q_{z'}$ for
$z,z'\in A_\delta$, 
$|[w_\zeta]|= |w_{z,\zeta}-w_\zeta^0|$
on $\partial Q_z\cap \partial Q_{z'}$ for
$z\in A_\delta$, $z'\in A_\delta^*$, 
and $|[w_\zeta]|=0$ on 
 $\partial Q_z\cap \partial Q_{z'}$ for
$z, z'\in A_\delta^*$.

We start from the boundary term. From \eqref{eqg0slambdapff} we get
\begin{multline}\label{eqbor}
\strokedint_{B_\eps}\sum_{z\in A_\delta}\int_{\partial Q_z}g_0(|w_{z,\zeta}-w^0_\zeta|)d\calH^{n-1}d\zeta\\
\leq
 \# A_\delta \calH^{n-1}(\partial Q_z)
\lambda + C_\lambda 
{\strokedint_{B_\eps}}\sum_{z\in A_\delta}
\int_{\partial Q_z} |w_{z,\zeta}-w_\zeta^0| d
\calH^{n-1}{d\zeta} \leq {C\frac\lambda\delta } + {C}\theta,
\end{multline}
by \eqref{eqdiff} and choosing $\eps$ sufficiently small. {For $\lambda\le \delta\theta$} the entire term is bounded by $C\theta$.

We now turn to the {first term of \eqref{eqsaltoPhi}}. We start from $z\in A_\delta$.
{Splitting
\begin{align*}
J_u\cup \Phi^{-1}(J_{w_\zeta}) =&
(J_u\cup \Phi^{-1}(J_{w_\zeta}))\cap (L_z\cup \Phi^{-1}(H_z))\\
&\cup (J_u\cup \Phi^{-1}(J_{w_\zeta}))\setminus (L_z\cup \Phi^{-1}(H_z))
\end{align*}
and using the subadditivity of $g_0$ to estimate the integral on the second set, we get
}
\newcommand\I{\mathrm{I}}
\newcommand\II{\mathrm{II}}
\newcommand\III{\mathrm{III}}
\newcommand\IV{\mathrm{IV}}
\newcommand\V{\mathrm{V}}
\begin{equation}
\strokedint_{B_\eps} \int_{(J_u\cup\Phi^{-1}(J_{w_\zeta}))\cap Q_z}g_0(|[u]-{[w_{\zeta}]}\circ\Phi|)\,d\calH^{n-1}\,d\zeta
 \le \I_z+\II_z+\III_z
  \end{equation}
with
\begin{equation}
\begin{split}
 \I_z:=&
  \strokedint_{B_\eps}
 \int_{(L_z\cup\Phi^{-1}(H_z))\cap Q_z}g_0(|[u]-[w_{\zeta}]\circ\Phi|)\,d\calH^{n-1} d\zeta,\\
 \II_z:=&
 \int_{J_u\cap Q_z
 \setminus  L_z}g_0(|[u]|)\,d\calH^{n-1},
 \\
 \III_z:=&
 \strokedint_{B_\eps}
 \int_{\Phi^{-1}(J_{w_\zeta}\setminus H_z)\cap Q_z
 }g_0(|[w_{\zeta}]\circ\Phi|)\,d\calH^{n-1}\,d\zeta.
\end{split}
\end{equation}
We start from $\I_z$. We add and subtract $s_z$, write
\begin{equation}\begin{split}
\I_z\le & \strokedint_{B_\eps}
 \int_{(L_z\cup\Phi^{-1}(H_z))\cap Q_z}g_0(|[u]-s_z{|})+
 g_0(|[w_{\zeta}]\circ\Phi-s_z|)
 \,d\calH^{n-1} d\zeta
 \end{split}
\end{equation}
{and observe that
\begin{equation*}
 g_0(|[u]-s_z|)
\chi_{(\Phi^{-1}(H_z)\setminus L_z)\cap Q_z} \le
 g_0(|[u]|)\chi_{(J_u\setminus L_z)\cap Q_z}
 +g_0(|s_z|)\chi_{(\Phi^{-1}(H_z)\setminus L_z)\cap Q_z},
\end{equation*}
and similarly for the other term. Therefore}
\begin{equation}\begin{split}
\I_z\le & \I_z^1+\I_z^2+\I_z^3 {+\II_z+\III_z},
 \end{split}
\end{equation}
where
\begin{equation}
  \I_z^1:=\int_{L_z\cap Q_z^*}g_0(|[u]-s_z|)\,d\calH^{n-1},
\end{equation}
\begin{equation}
  \I_z^2:=
   \strokedint_{B_\eps}
  \int_{\Phi^{-1}(H_z)\cap Q_z}g_0(|[w_{\zeta}]\circ\Phi-s_z|)\,d\calH^{n-1}d\zeta,
\end{equation}
and
\begin{equation}
  \I_z^3:=
   g_0(|s_z|)
  \calH^{n-1}( (L_z\triangle \Phi^{-1}(H_z))\cap Q_z).
\end{equation}
{First note that
\begin{equation}\label{eqI1zIIz}
\sum_{z\in A_\delta}(\I^1_z+\II_z)\leq C\theta
\end{equation}
thanks to} \eqref{e:approx scala delta}. For $\I^2_z$ we use
first the Coarea formula, \eqref{eqboundDPhi0}
and \eqref{eqsaltowzzeta}
to obtain
\begin{equation}\label{eqIz2area}
\begin{split}
  \I_z^2\le& 2 \strokedint_{B_\eps}
\int_{H_z\cap Q_z}g_0( |\Pi_{\eps,\zeta}(U_z^+-U_z^-)-s_z| )\,d\calH^{n-1}d\zeta.
\end{split}\end{equation}
We cover $H_z\cap Q_z$ with the balls $B_i$ introduced in \eqref{eqdefBi}, and start from estimating the term
\begin{equation}
 \I_z^2(B_i):= \strokedint_{B_\eps}
\int_{H_z\cap B_i}g_0( |\Pi_{\eps,\zeta}(U_z^+-U_z^-)-s_z| )\,d\calH^{n-1}d\zeta.
\end{equation}
 By subadditivity,
\begin{equation}\label{eqsubpiii}
g_0( |\Pi_{\eps,\zeta}(U_z^+-U_z^-)-s_z| )
\le  g_0(|h_i^+-h_i^--s_z|)+
g_0( |\Pi_{\eps,\zeta}(U_z^+-U_z^-)-h_i^++h_i^-| ).
\end{equation}
The first term, using
 \eqref{eqestimateBiLzHz} twice and subadditivity, leads to
\begin{equation}\label{eqsubpiii2}\begin{split}
 &\int_{H_z\cap B_i}
 g_0(|h_i^+-h_i^--s_z|)d\calH^{n-1}
 \le {C} \eps^{n-1}g_0(|h_i^+-h_i^--s_z|)\\
 &\le {C} \int_{L_z\cap B_i}
[g_0(|[u]-s_z|)+ g_0(|h_i^+-h_i^--[u]|)]d\calH^{n-1},
\end{split}
\end{equation}
where the first integral is controlled by $\I_z^1$.
Using \eqref{eqg0slambdapff} in the
second term of \eqref{eqsubpiii} and the second term of \eqref{eqsubpiii2},
for any $\lambda>0$
we have
\begin{equation}\label{e:Iz2Bi}\begin{split}
  \I_z^2(B_i)\le& {C} \int_{L_z\cap B_i}
g_0(|[u]-s_z|)d\calH^{n-1}+{C}\lambda \eps^{n-1}\\
&+C_\lambda
\int_{L_z\cap B_i} |h_i^+-h_i^--[u]|\,d\calH^{n-1}\\
& + C_\lambda
 \strokedint_{B_\eps}
\int_{H_z\cap B_i} |\Pi_{\eps,\zeta}(U_z^+-U_z^--h_i^++h_i^-) |\,d\calH^{n-1}\,d\zeta.
\end{split}\end{equation}
The term  in the second line can be estimated with \eqref{eqestbilzuuhh}. For the last line we use first \eqref{eqpsurface} and then
\eqref{eqUzpUzmhhfd}, and obtain
\begin{equation}\label{e:Iz2Bi second line}\begin{split}&
\strokedint_{B_\eps}
\int_{H_z\cap B_i}
 |\Pi_{\eps,\zeta}(U_z^+-U_z^--h_i^++h_i^-)| d\calH^{n-1}\,d\zeta\\
&
\le \frac{C}{\eps}
\int_{B_i^*}|U_z^+-U_z^--h_i^++h_i^-|dx
+C (|DU_z^+|+|DU_z^-|)(B_i^*)\\
&\le
C (|DU_z^+|+|DU_z^-|)(B_i^*).
\end{split}\end{equation}
Using that  $\Lip(\varphi_z)\le{\frac{1}{2}}$, \eqref{eqestimateBiLzHz} and $\sum_{i=1}^K\chi_{B_i^*}\leq {C}$, summing over $i$ yields
\begin{equation}\begin{split}
  \I_z^2\le&
{C}\I^1_z+C\lambda \delta^{n-1}
  +{C_\lambda}|Du|(O_z^\eps\setminus L_z)
  + {C_\lambda} (|DU_z^+|+|DU_z^-|)(O_z^\eps),
\end{split}\end{equation}
so that by \eqref{eqeUpm} summing on $z\in A_\delta$ we find for $\lambda\le\delta \theta $
and $\eps$ sufficiently small
\[
\sum_{z\in A_\delta}I_z^2\leq C\theta\,.
\]
We next turn to $\I_z^3$, and observe that by \eqref{eqHQPhiL0}
{and \eqref{eqboundDPhi0}}
\begin{equation}\begin{split}
  \calH^{n-1}( (L_z\triangle \Phi^{-1}(H_z))\cap Q_z)\le& 
  \calH^{n-1}( L_z\cap Q_z\setminus Q_z')
+  \calH^{n-1}( \Phi^{-1}(H_z)\cap Q_z\setminus Q_z')\\
\le& 
  \calH^{n-1}( L_z\cap Q_z\setminus Q_z')
+ 2 \calH^{n-1}( H_z\cap Q_z\setminus Q_z').
\end{split}\end{equation}
Therefore
\begin{equation*}
 \I_z^3
 \le {4} \max_{z\in A_\delta} g_0(|s_z|) \calH^{n-1}\left((H_z\cup L_z)\cap (Q_z\setminus Q_z')\right).
\end{equation*}
As $\calH^{n-1}(\bigcup_{z\in A_\delta} (H_z\cup L_z){\cap Q_z})<\infty$ and
$\bigcup_{\delta'<\delta} Q_z'=Q_z$,
choosing $\delta'$ sufficiently close to $\delta$ we have
\begin{equation}\label{eqchoicedeltaprime}
{\sum_{z\in A_\delta} \calH^{n-1}\left((H_z\cup L_z)\cap (Q_z\setminus Q_z')\right)\le\theta\text{ and }}
 \sum_{z\in A_\delta} \I_z^3\le  \theta.
\end{equation}
Similarly, if $\calH^{n-1}(J_u)<\infty$, for $\delta'$ sufficiently close to $\delta$ we have
\begin{equation}\label{eqsaltopiccolo}
\calH^{n-1}(J_u\cap\bigcup_{z\in A_\delta} (Q_z\setminus {Q_z'}))\leq \theta.
\end{equation}

For $\III_z$ we use the Coarea formula {and \eqref{eqboundDPhi0}}. We obtain
\begin{equation}
 \III_z\le 
2\strokedint_{B_\eps}
  \int_{J_{w_{{z,\zeta}}}\cap Q_z {\setminus H_z }}
 g_0(|[w_{z,\zeta}]|)\,d\calH^{n-1}\,d\zeta. \end{equation}  
Therefore $\III_z\le 2\III_z^++2\III_z^-$, with
\begin{equation}
 \III_z^+:=\strokedint_{B_\eps}
  \int_{J_{\Pi_{\eps,\zeta} U_z^+}\cap Q_z\cap H_z^+}
{ g_0(|[\Pi_{\eps,\zeta} U_z^+]|)}
 \,d\calH^{n-1}\,d\zeta, \end{equation}
 {and similarly for $\III_z^-$.}
 We use {inequalities \eqref{e:stima superficie}}
 {to infer}
 \begin{align}\label{e:stima g_0.1}
\sum_{z\in A_\delta}\III_z^+\leq& 
 C\sum_{z\in A_\delta}\int_{J_{U_z^+}\cap (H_z^{{+}}\cap Q_z)_{2\cpropjbounds\eps}} g_0(|[U_z^+]|)\,d\calH^{n-1},
\end{align}
and the same for $\III_z^-$.
{Both}
can be estimated via \eqref{eqestjumpUpm},
{we conclude that}
\begin{equation}\label{eqestiiiifin}
 \sum_{z\in A_\delta}
 \III_z \le C\theta.
\end{equation}

{We} next treat {the bulk term in \eqref{eqsaltoPhi}} {in the case} $z\in A_\delta^*$. Since $\Phi(x)=x$ and ${w_\zeta}={w_\zeta^0}$ on $Q_z$, {recalling Proposition \ref{propproj} \ref{propprojjump},}
\begin{equation}
\strokedint_{B_\eps} \int_{(J_u\cup J_{w_\zeta^0})\cap Q_z}g_0(|[u]-{[w_{\zeta}^0]}{|})\,d\calH^{n-1}\,d\zeta
 \le \IV_z+\V_z,
  \end{equation}
where
\begin{equation}
\begin{split}
 \IV_z:=&
 \int_{J_u\cap Q_z
 }g_0(|[u]|)\,d\calH^{n-1},
 \\
 \V_z:=&
 \strokedint_{B_\eps}
 \int_{J_{w_\zeta^0}\cap Q_z 
 }g_0(|[w_\zeta^0]|)\,d\calH^{n-1}\,d\zeta.
\end{split}
\end{equation}
As for $\III_z^{{+}}$,
 we use inequality \eqref{e:stima superficie} in Proposition~\ref{propproj}
 and obtain
\begin{equation}
 \V_z\le
 C
 \int_{J_{u}\cap (Q_z)_{2\cpropjbounds\eps}
 }g_0(|[u]|)\,d\calH^{n-1}
\end{equation}
 so that
 \begin{equation}\label{eqIVandV}
  \sum_{z\in A_\delta^*}
  (\IV_z+\V_z)\le 
  C\mu_u (\bigcup_{z\in A_\delta^*} (Q_z)_{2\cpropjbounds\eps})\le
  C \mu_u( (\partial\Omega)_{3\sqrt n \delta})\le C\theta,
 \end{equation}
{for $\varepsilon$ sufficiently small,} where in the last step we used \eqref{eqchoicedelta0}.
Combining the previous estimates, {\eqref{eqsaltoPhi} yields}
\begin{equation}\label{theodensityintrog3}
\strokedint_{B_\eps}\int_{{\Omega\cap}(J_u\cup \Phi^{-1}(J_{w_\zeta}))}g_0(|[u]-[w_\zeta]\circ\Phi|)\,d\calH^{n-1}d\zeta\leq C\theta\,.
\end{equation}

We claim next that for $\eps$ sufficiently small
\begin{equation}\label{e:theodensityintrog4}
\strokedint_{B_\eps}\int_{\Omega\cap(J_u\cup \Phi^{-1}(J_{w_\zeta}))}g_0(|[u]|+|[w_\zeta]\circ\Phi|)\,
\bigl|\nu_u - \nu_{w_\zeta}\circ\Phi\bigr|\,d\calH^{n-1}d\zeta\leq C\theta\,.
\end{equation}
Thanks to subadditivity and monotonicity of $g_0$,
\eqref{theodensityintrog3} implies that
it suffices to prove
\begin{equation}\label{e:theodensityintrog41}
\strokedint_{B_\eps}\int_{\Omega\cap(J_u\cap \Phi^{-1}(J_{w_\zeta}))}g_0(|[u]|)\,
\bigl|\nu_u - \nu_{w_\zeta}\circ\Phi\bigr|\,d\calH^{n-1}d\zeta\leq C\theta\,.
\end{equation}
 Similarly, by
\eqref{eqI1zIIz},
{\eqref{eqchoicedeltaprime}} and
\eqref{eqIVandV} it suffices prove
\begin{equation}\label{e:theodensityintrog42}
\sum_{z\in A_\delta}
\strokedint_{B_\eps}\int_{Q_z\cap L_z\cap \Phi^{-1}(H_z)}
g_0(|[u]|) |\nu_u-\nu_{w_\zeta}\circ\Phi| d\calH^{n-1}
\,d\zeta\leq C\theta\,.
\end{equation}
From \eqref{eqnuwzeta}
we obtain that
$|\nu_{w_\zeta}\circ\Phi-
R_ze_n|\le \theta$ almost everywhere on $Q_z\cap\Phi^{-1}(H_z)$. The claim follows then from
\eqref{e:approx scala delta} and
integrability of {$g_0(|[u]|)$}.

\itemtit{Choice of $\zeta$, conclusion of the proof}

From \eqref{eqnorm2}, \eqref{e:energia bulk w} , \eqref{theodensityintrog3} and \eqref{e:theodensityintrog4}, it is easy to check that there is a subset
$\tilde B\subset B_\eps$, with $|\tilde B|/|B_\eps|>1/2$, such that for all $\zeta\in \tilde B$ {\eqref{eqHz} holds and} we have
\[\int_\Omega| w_\zeta- u|dx\leq C\theta,\]
\[\int_\Omega|\nabla w_\zeta-\nabla u|^pdx\leq C\theta,\]
\[\int_{{\Omega\cap}(J_u\cup \Phi^{-1}(J_{w_\zeta}))}g_0(|[u]-[w_{\zeta}]\circ\Phi|)\,d\calH^{n-1}\leq C\theta,\]
\[\int_{{\Omega\cap}(J_u\cup \Phi^{-1}(J_{w_\zeta}))}g_0(|[u]|+|[w_{\zeta}]\circ\Phi|)\,|\nu_u-\nu_{w_\zeta}\circ\Phi|\,d\calH^{n-1}\leq C\theta.\]
{If $\calH^{n-1}(J_u)<\infty$ then, using \eqref{eqjumpnotH}, we can choose $\tilde B$  so that additionally
\begin{equation*}
  \sum_{z\in A_\delta}
 \calH^{n-1}(\overline Q_z\cap
 J_{w_\zeta} \setminus H_z
 ) \le C\theta,
\end{equation*}
which by the Coarea formula as usual implies
\begin{equation}\label{eqHn1noH}
  \sum_{z\in A_\delta}
 \calH^{n-1}(\overline Q_z\cap
\Phi^{-1}( J_{w_\zeta} \setminus H_z)
 ) \le C\theta.
\end{equation}}

Properties
\ref{theodensityintroL1}, \ref{theodensityintroLp} and \ref{theodensityintrog0} follow;
{\ref{theodensityintrosimpl} and \ref{theodensityintrobilip} had already been proven.}
Property \ref{theodensityintrosbv0} is immediate.

It remains to prove \ref{theodensityintroHn1}. We assume that $\calH^{n-1}(J_{u})<\infty$
and start from a bound
on $\Phi^{-1}(J_{w_\zeta})\setminus J_u$.
We split the jump set of $w_\zeta$ into the contribution inside each cube $Q_z$, for $z\in A_\delta^*\cup A_\delta$, and then
for each $z\in A_\delta$,
we split the jump set $J_{w_\zeta}$ into the part in $H_z$ and the rest. We obtain
\begin{equation}\begin{split}
\label{Jw-Ju1}&
\calH^{n-1}(\Omega\cap\Phi^{-1}(J_{w_\zeta})\setminus J_u)
\le \sum_{z\in A^*_\delta}\calH^{n-1}( \overline Q_z\cap \Phi^{-1}(J_{w_\zeta}))
\\
	&\hskip2mm+\sum_{z\in A_\delta}\Big(\calH^{n-1}(
\overline Q_z\cap
\Phi^{-1}(H_z)\setminus J_u)+\calH^{n-1}
(\overline Q_z\cap
\Phi^{-1}(J_{w_\zeta}\setminus H_z ))
\Big).
\end{split}\end{equation}In the first term in the second line,
we
use \eqref{eqmujbdryv2} to drop the part on $\partial {Q_z}$ and then separate the contributions inside and outside $L_z$.
Equation \eqref{eqHQPhiL0} ensures that $Q_z'\cap \Phi^{-1}(H_z)\setminus L_z=\emptyset$.
We obtain
\begin{equation}\label{Jw-Ju2}\begin{split}
& \sum_{z\in A_\delta}\calH^{n-1}(
\overline Q_z\cap
\Phi^{-1}(H_z)\setminus J_u)
\\
&\le \sum_{z\in A_\delta}\Big(\calH^{n-1}(Q_z\cap L_z\setminus J_u)+\calH^{n-1}(Q_z \cap \Phi^{-1}(H_z)\setminus Q_z')\Big)\le C\theta,
\end{split}
\end{equation}
where the last inequality follows from \eqref{e:approx scala delta02}, \eqref{eqchoicedeltaprime} and the Coarea formula.
The
other two terms in \eqref{Jw-Ju1}
can be bounded by \eqref{eqHn1noH} and  \eqref{eqchoicedelta02}, and we conclude
\begin{equation}\label{Jw-Ju3}
\calH^{n-1}(\Omega\cap\Phi^{-1}(J_{w_\zeta})\setminus J_u)
\le  C\theta.
\end{equation}

The converse inequality is proven by a different argument based on lower semicontinuity.
As discussed in the first lines of the proof, taking a sequence $\theta_j\to0$ we obtain a sequence $w_j$ which has all stated properties, except that \ref{theodensityintroHn1}
is replaced by the weaker assertion
\begin{equation}\label{eqviweak}
\limsup_j
 \calH^{n-1}(\Omega\cap \Phi_j^{-1}(J_{w_j})\setminus J_u)=0.
\end{equation}
In particular $w_j$ converges to $u$ in $L^1(\Omega;\R^m)$,
and since $\nabla w_j$ converges to $\nabla u$ strongly in $L^p(\Omega;\R^{m\times n})$, {with $p\in[1,\infty)$ given}, there is a function $f:\R^{m\times n}\to[0,\infty)$ with superlinear growth at infinity such that
\begin{equation}
 \limsup_j
 \int_\Omega f(\nabla w_j) dx
  + \calH^{n-1}(\Omega\cap J_{w_j})<\infty
\end{equation}
(if $p>1$, then $f(\xi):=|\xi|^p$ itself will do; {if $p=1$, de la Vallée-Poussin Theorem gives the conclusion}).
By the $SBV$ closure and lower semicontinuity theorem \cite[Theorem 4.7]{AFP}, we deduce
\begin{equation}\label{eqlsch}
 \calH^{n-1}(\Omega\cap J_u)\le\liminf_j
 \calH^{n-1}(\Omega\cap J_{w_j})=\liminf_j
 \calH^{n-1}(\Omega\cap \Phi_j^{-1}(J_{w_j})),
\end{equation}
where the last equality can be obtained from the area formula and \ref{theodensityintrobilip}.
By additivity of $\calH^{n-1}$,
\begin{equation}\begin{split}
 \calH^{n-1}(\Omega\cap J_u\setminus \Phi_j^{-1}(J_{w_j}))
=& \calH^{n-1}(\Omega \cap J_u )
- \calH^{n-1}(\Omega\cap \Phi_j^{-1}(J_{w_j}))\\
&+ \calH^{n-1}(\Omega\cap  \Phi_j^{-1}(J_{w_j})\setminus J_u).
\end{split}\end{equation}
Using first  \eqref{eqlsch}
and then \eqref{eqviweak}, 
\begin{equation}
 \limsup_j
 \calH^{n-1}(\Omega\cap J_u\setminus \Phi_j^{-1}(J_{w_j}))
 \le
\limsup_j
 \calH^{n-1}(\Omega\cap \Phi_j^{-1}(J_{w_j})\setminus J_u)=0,
\end{equation}
which concludes the proof.
\end{schrittlist}
\end{proof}

\section*{Acknowledgements}

SC gratefully
thanks the University of Florence for the warm hospitality of the DiMaI ``Ulisse
Dini'', where part of this work was carried out.
MF gratefully acknowledges the warm hospitality of the Institute of Applied Mathematics of the University of Bonn, where part of this work was carried out.

\addcontentsline{toc}{section}{References}
\bibliographystyle{alpha-noname}
\bibliography{cfi}

\end{document}